\documentclass{amsart}

\usepackage[colorlinks=true, urlcolor=black, citecolor=black, linkcolor=black, hyperfootnotes=true]{hyperref}
\usepackage{amssymb}
\usepackage{aliascnt}
\usepackage{accents}
\usepackage{mathrsfs}

\usepackage{bookmark}

\numberwithin{equation}{section}

\newtheorem{thm}{Theorem}[section]

\newaliascnt{prp}{thm}
\newtheorem{prp}[prp]{Proposition}
\aliascntresetthe{prp}

\newaliascnt{cor}{thm}
\newtheorem{cor}[cor]{Corollary}
\aliascntresetthe{cor}

\newaliascnt{lem}{thm}
\newtheorem{lem}[lem]{Lemma}
\aliascntresetthe{lem}

\theoremstyle{definition}

\newaliascnt{dfn}{thm}
\newtheorem{dfn}[dfn]{Definition}
\aliascntresetthe{dfn}

\newaliascnt{xpl}{thm}
\newtheorem{xpl}[xpl]{Example}
\aliascntresetthe{xpl}

\newaliascnt{rmk}{thm}
\newtheorem{rmk}[rmk]{Remark}
\aliascntresetthe{rmk}

\newtheorem*{rmk*}{\textnormal{\emph{Remark}}}

\author{Tristan Bice}
\thanks{The author is supported by the GA\v{C}R project EXPRO 20-31529X and RVO: 67985840.}
\address{Institute of Mathematics of the Czech Academy of Sciences, \v{Z}itn\'a 25, Prague}
\email{bice@math.cas.cz}
\keywords{Gelfand duality, \'etale groupoid, Fell bundle, Cartan subalgebra}

\makeatletter
\@namedef{subjclassname@2020}{\textup{2020} Mathematics Subject Classification}
\makeatother
\subjclass[2020]{16G30, 20M30, 22A22, 46L05, 46L85, 54D80}

\title[Dauns-Hofmann-Kumjian-Renault Duality]{Dauns-Hofmann-Kumjian-Renault Duality for\\ Fell Bundles and Structured C*-Algebras}

\begin{document}

\begin{abstract}
We unify the classic Dauns-Hofmann representation with Kumjian and Renault's Weyl groupoid representation.  More precisely, we use ultrafilters to represent C*-algebras with some additional structure on Fell bundles over locally compact \'etale groupoids.  Our construction is even functorial and thus a fully-fledged non-commutative extension of the classic Gelfand duality.
\end{abstract}

\maketitle

\section*{Introduction}

\subsection*{Historical Background}

A defining moment in the history of operator algebras came with the now famous Gelfand representation (see \cite{Gelfand1941} and \cite{GelfandNaimark1943}).  This shows that commutative C*-algebras are nothing more than continuous complex-valued functions on locally compact Hausdorff spaces.  As a result, non-commutative C*-algebras began to be viewed as functions on `non-commutative spaces', a general philosophy that has led to various C*-algebraic analogs of topological properties.

Naturally, there has been a long-standing quest to properly extend the Gelfand representation and illuminate these elusive `non-commutative spaces'.  The late 60's and early 70's saw a flurry of activity in this direction, with a number of proposed non-commutative Gelfand dualites (see \cite{Takesaki1967}, \cite{DaunsHofmann1968}, \cite{Akemann1969}, \cite{Bichteler1969}, \cite{Akemann1970}, \cite{Akemann1971} and \cite{GilesKummer1971}), each with their own idea of what non-commutative spaces should be.  Some of the most influential work here is due to Akemann, whose idea was to replace open sets with open projections.  This yields an appealing kind of non-commutative topology, which has proved useful in analysing certain aspects of C*-algebras, like their their Cuntz semigroups (see \cite{OrtegaRordamThiel2012}) and hereditary C*-subalgebra lattices (see \cite{AkemannBice2014}).  However, the open projections may provide little insight into how the C*-algebra was originally built, which seems to have limited their use in other contexts.

Another important Gelfand extension from this period is due to Dauns and Hofmann (see \cite{DaunsHofmann1968} and \cite{Hofmann2011}).  Their idea was to allow the continuous functions to take values, not just in $\mathbb{C}$, but in other C*-algebras.  More precisely, given a C*-algebra $A$, they constructed a bundle of (potentially simpler) C*-algebras on which $A$ gets represented as continuous sections.  This is also an appealing Gelfand extension, the only issue being that it says nothing when $A$ is already a simple unital C*-algebra.  As these have been the primary preoccupation of many a C*-algebraist, particularly those devoted to Elliott's classification program (see \cite{Winter2018}), the use of the Dauns-Hofmann representation has again been somewhat limited.

Fast forward to the 80's and we see another turning point in the history of operator algebras with Renault's breakthrough thesis on \'etale groupoids \cite{Renault1980}.  His idea was that by endowing a space with a suitable partial product, we can produce even non-commutative C*-algebras from continuous complex-valued functions.  By now we know that most naturally occuring C*-algebras do indeed arise from \'etale groupoids (especially if we consider slightly more general twisted groupoids \textendash\, see the comments below).  We also know that \'etale groupoids form a natural intermediary between C*-algebras and various combinatorial objects used to define them, like directed graphs and dynamical systems (see \cite{Exel2008}).  As such, groupoid models for C*-algebras can provide key insights into how they were originally built.  

While Renault obtained C*-algebras from \'etale groupoids, Kumjian \cite{Kumjian1986} showed us how to reverse the process, constructing a `Weyl groupoid' on which to represent a given C*-algebra $A$ (with further refinements again by Renault in \cite{Renault2008}).  The only caveat is that we need a little extra structure coming from a Cartan subalgebra of $A$.  On the other side, the Weyl groupoid also comes with a twist or, equivalently, a Fell line bundle.  In recent years there has also been a big push to extend the Weyl groupoid construction in various directions - see \cite{BrownClarkanHuef2017}, \cite{BrownloweCarlsenWhattaker2017}, \cite{CarlsenRuizSimsTomforde2017}, \cite{CarlsenRout2018}, \cite{ChoiGardellaThiel2019}, \cite{Steinberg2019}, \cite{KwasniewskiMeyer2020}, \cite{BiceClark2020}, \cite{ArmstrongCastroClarkCourtneyLinMcCormickRamaggeSimsSteinberg2021} and \cite{Bice2021}.  From this together with the vast literature on groupoid C*-algebras that has emerged over the past 30-40 years, it seems safe to say that \'etale groupoids are the best candidate for these long sought after `non-commutative spaces'.

However, there is still one key aspect of the classic Gelfand representation that is missing here.  Specifically, homomorphisms of commutative C*-algebras correspond to continuous functions of their spectra, i.e. the Gelfand representation is functorial.  Until now, it was not clear whether any non-trivial functoriality also applies to the Weyl groupoid representation.  In fact, it was not uncommon to hear C*-algebraists claim that there are no natural morphisms for \'etale groupoids, a myth that was finally dispelled in \cite{AustinMitra2018}, \cite{AustinGeorgescu2019} and \cite{Li2020}.  Also, the original Weyl groupoid construction is perhaps not as general as one might like, dealing only with effective groupoids and their line bundles.  Some of the work above yields extensions to more general groupoids, although some restriction on the isotropy has remained.  Other work like \cite{Exel2011} and \cite{KwasniewskiMeyer2020} looks at more general Fell bundles, at least over inverse semigroups.  If we could construct Fell bundles over \'etale groupoids instead, including the C*-bundles mentioned above, then we could even unify Kumjian and Renault's construction with the earlier work of Dauns and Hofmann.

This is exactly what we do in the present paper.  Just as important as the representation itself is the way it is constructed.  Rather than trying to bootstrap the original Gelfand construction with a groupoid of germs formed from normalisers, we take a more flexible abstract approach.  Specifically, the points of our Weyl groupoid will be ultrafilters with respect to a `domination relation' defined from the product structure.  This harks back to classic topological dualities due to Stone \cite{Stone1936}, Wallman \cite{Wallman1938} and Milgram \cite{Milgram1949}, and is in the same spirit as other non-commutative dualities obtained more recently in \cite{Lawson2012}, \cite{LawsonLenz2013}, \cite{KudryavtsevaLawson2016} and \cite{BiceStarling2018}.  Essentially the same construction even applies in more general algebraic contexts, as we have examined in \cite{Bice2020Rep} and \cite{Bice2020Rings}.  Indeed, the present paper builds on and could be viewed as the culmination of this previous work.

\newpage

\subsection*{Motivation}

To motivate our construction, let us go back and reexamine the original Gelfand representation.  Here we start with a commutative C*-algebra $A$ which we represent as continuous functions on the character space of $A$, i.e. the space of all $\mathbb{C}$-valued homomorphisms on $A$ with the weak* topology.  Equivalently, we can consider kernels of characters, which are precisely the maximal proper closed ideals
\[\mathcal{I}(A)=\{I\subseteq A:I\text{ is a maximal with }A\neq I=\mathrm{cl}(I)=I+I=AI\}.\]
The weak* topology corresponds to the hull-kernel topology generated by the sets
\[\mathcal{I}_a=\{I\in\mathcal{I}(A):a\notin I\}.\]

What if we were to consider the complements $A\setminus I$ of these ideals instead?  It turns out these have their own intrinsic characterisation based purely on the product structure of $A$.  To see this, define the \emph{domination} relation $<$ on $A$ by
\[a< b\qquad\Leftrightarrow\qquad\exists c\in A\ (a=abc).\]
This relation is transitive and thus it is natural to consider the \emph{filters} on $A$ w.r.t. domination, i.e. the non-empty down-directed up-sets $U\subseteq A$ (see \eqref{Filter} below).  In particular, we can consider the \emph{ultrafilters}
\[\mathcal{U}(A)=\{U\subseteq A:U\text{ is a maximal proper filter}\}.\]
Taking complements yields a bijection between $\mathcal{I}(A)$ and $\mathcal{U}(A)$, which is also a homeomorphism when we consider the topology on $\mathcal{U}(A)$ generated by the sets
\[\mathcal{U}_a=\{U\in\mathcal{U}(A):a\in U\}.\]

\begin{rmk*}
It is already somewhat surprising that we can identify the maximal proper closed ideals of $A$ just from the product, i.e. even after stripping away addition, scalar multiplication and the norm on $A$.  Similar observations were made for real-valued functions on locally compact spaces in \cite{Milgram1949}, although Milgram's work never received the same degree of attention as Gelfand's and seems to have been largely forgotten by operator algebraists today.
\end{rmk*}

The advantage of ultrafilters over ideals/characters becomes even more apparent when we consider non-commutative C*-algebras and their Cartan subalgebras.  This is because it is easy to modify the domination relation using the Cartan subalgebra and give the resulting ultrafilters a groupoid structure, thus providing a simplified construction of Kumjian and Renault's Weyl groupoid.

Specifically, given a Cartan subalgebra $C$ of a C*-algebra $A$, we define domination $<$ on the normaliser semigroup $S=\{s\in A:sCs^*+s^*Cs\subseteq C\}$ by
\[a<b\qquad\Leftrightarrow\qquad\exists s\in S\ (as,sa,bs,sb\in C\text{ and }a=asb=bsa).\]
Again we consider ultrafilters $\mathcal{U}(S)$ w.r.t. $<$ with the topology generated by $(\mathcal{U}_a)_{a\in S}$ where $\mathcal{U}_a=\{U\in\mathcal{U}(S):a\in U\}$.  The groupoid structure of $\mathcal{U}(S)$ then arises immediately from the product on $S$ \textendash\, for any $T,U\in\mathcal{U}(S)$ with $0\notin TU$, we define
\[T\cdot U=(TU)^<=\{s>tu:t\in T\text{ and }u\in U\},\]
i.e. $T\cdot U$ is defined to be the ultrafilter generated by $TU$ whenever possible.  By utilising the rest of the C*-algebra structure of $A$ and the expectation $\Phi:A\rightarrow C$, we can then construct a Fell line bundle over $\mathcal{U}(S)$ on which to represent $A$.

This simplified Weyl groupoid construction is even functorial in the following sense.  Say we have C*-algebras $A$ and $A'$ with Cartan subalgebras $C$ and $C'$.  These have corresponding normaliser semigroups $S$ and $S'$ as well as `co-Cartan subspaces' $K=\Phi^{-1}\{0\}$ and $K'=\Phi'^{-1}\{0\}$, where $\Phi$ and $\Phi'$ are the expectations onto $C$ and $C'$.  Any C*-algebra homomorphism $\pi:A\rightarrow A'$ preserving the Cartan subalgebra, its normaliser semigroup and its co-Cartan subspace, i.e. such that
\[\pi[C]\subseteq C',\quad\pi[S]\subseteq S'\quad\text{and}\quad\pi[K]\subseteq K',\]
yields a proper star-bijective functor $\underline{\pi}$ in the opposite direction between the corresponding Weyl groupoids.  Specifically, whenever $U\in\mathcal{U}(S')$ is an ultrafilter with non-empty preimage $\pi^{-1}[U]$, we define $\underline{\pi}(U)$ to be the ultrafilter it generates
\[\underline{\pi}(U)=\pi^{-1}[U]^<.\]
We can then build this up to a `Fell morphism' of the corresponding Fell bundles.

From here, things can be generalised even further in a straightforward manner.  First let us switch the roles of $S$ and $C$, instead taking $S$ as a given subsemigroup of $A$ and defining $C=C^*(S_+)$, where $S_+=\{s^*s:s\in S\}$.  If we drop the requirement that $C$ be \emph{maximal} commutative, we can construct even non-effective Weyl groupoids of ultrafilters.  Instead of requiring $C$ to be commutative, we can just assume we are given some central C*-subalgebra $Z\subseteq\mathsf{Z}(C)$.  Modifying the domination relation accordingly, we can then construct even more general Fell bundles (i.e. non-line bundles, including the C*-bundles considered by Dauns and Hofmann) on which to represent $A$.  This generalised construction is again functorial and thus provides a fully-fledged non-commutative extension of the classic Gelfand duality, as we now set about showing.

\subsection*{Outline}
We start off with some preliminary material on Fell bundles and their sections in \autoref{FellBundles}.  The reader who is already familiar with Fell bundles may like to skip this, although our more concrete approach may be of some independent interest.  Indeed, this approach allows us to construct not just the reduced C*-algebra $\mathcal{C}_\mathsf{r}(\rho)$ but also a number of larger algebras and modules which still consist of concrete sections of the given Fell bundle $\rho$.  For example, in \autoref{BilocallyReduced} we identify the multiplier algebra of $\mathcal{C}_\mathsf{r}(\rho)$ as the `bilocally reduced' C*-algebra $\mathcal{B}_\mathsf{r}(\rho)$, at least when $\rho$ is a line bundle or a more general categorical Fell bundle.  In the end, our duality will concern Fell bundles which are not just categorical but even \emph{corical}, i.e. Fell bundles $\rho:B\twoheadrightarrow\Gamma$ where the core $B^\times$ of invertible elements in the total space is large and well-behaved (see \autoref{CategoricalCorical}).

In \autoref{StructuredCAlgebras} we introduce a C*-algebraic counterpart to Fell bundles which we call \emph{structured C*-algebras}.  These are just quadruples $(A,S,Z,\Phi)$ consisting of a C*-algebra $A$, a *-subsemigroup $S$, a commutative C*-subalgebra $Z$ and an expectation $\Phi$ on $A$ satisfying a few elementary conditions.  We then examine the key motivating examples in \autoref{Examples} and introduce appropriate morphisms in \autoref{StructuredMorphisms}.  Lastly, we discuss a number of related conditions they can satisfy in \autoref{StructuredProperties}, noting in \autoref{FellWellStructured} that they are indeed satisfied by structured C*-algebras coming from Fell bundles.

Next we introduce the domination relation in \autoref{Domination}, which will play a central role in our Weyl bundle construction.  We then investigate some of its basic properties, making frequent use of the continuous functional calculus.  Key among these are the interpolation properties in \autoref{Interpolative}, which in particular show that $S$ is a predomain with respect to domination -- see \eqref{Predomain}.  Following this, we introduce a closely related notion of \emph{compatibility} in \autoref{Compatibility} and examine some of its basic properties.

Restricting our attention to \emph{well-structured C*-algebras} in \autoref{Ultrafilters}, we first recall our construction of the ultrafilter groupoid in \autoref{UltrafilterGroupoid}, which will form the base of our Weyl bundle.  Then in \autoref{LCHultrafilters} we show how domination and compatibility correspond to compact containment and slice unions of the associated basic open sets of ultrafilters.  Following this, we show in \autoref{UltraFunctoriality} that the ultrafilter groupoid construction is functorial with respect to structure-preserving morphisms.

Next, in \autoref{WeylSeminorms}, we use ultrafilters to define \emph{Weyl seminorms} on our C*-algebra.  This is the first step to defining the total space of our Weyl bundle, as we go on to show in \autoref{TheWeylBundle}.  At the end in \autoref{WeylBundleFunctoriality}, we show that the Weyl bundle construction is also functorial with respect to structure-preserving morphisms.

Having constructed our Weyl bundle, we show how a structured C*-algebra is naturally represented upon it in \autoref{TheWeylRepresentation}.  In particular, in \autoref{WeylRep} we show that the Weyl representation is indeed a homomorphism on the C*-algebra $C^*(S^>)$ generated by dominated elements.  In \autoref{FaithfulWeyl}, we note that the Weyl representation is also an isomorphism when the expectation is faithful.

Finally in \autoref{CategoricalDuality}, we use all the theory developed thus far to construct a categorical equivalence between faithfully structured C*-algebras and corical Fell bundles, as summarised in \autoref{TheAdjunction}.  In the last section, we finish with a few remarks about potential future directions this work could take.

\part{Fell Bundles}\label{FellBundles}

To define Fell bundles, we must first review some basic terminology and notation.

\section{Preliminaries}

\subsection{Semigroupoids}

Most of the material here is fairly standard, the only difference being that with Fell bundles it is convenient to work with object-free notions of semigroupoids, categories, groupoids, etc..

\begin{dfn}\label{Semigroupoids}
A \emph{semigroupoid} is a set $S$ together with a partial associative binary operation $(a,b)\mapsto ab$, i.e. $(ab)c$ is defined iff $a(bc)$ is, in which case they are equal.

A \emph{semigroup} is a semigroupoid $S$ where the product $ab$ is defined for all $a,b\in S$.
\end{dfn}

Semigroupoids are sometimes called `partial semigroups' (e.g. see \cite{BergelsonBlassHindman1994}).

The pairs for which the product is defined are denoted by
\[S^2=\{(a,b)\in S\times S:ab\text{ is defined}\}.\]
We call $x\in S$ a \emph{unit} if $xa=a$ when $(a,x)\in S^2$ and $ax=a$ when $(x,a)\in S^2$.  Let
\[S^0=\{x\in S:x\text{ is a unit}\}.\]

\begin{prp}
For any $a\in S$, there is at most one $x\in S^0$ with $(a,x)\in S^2$.
\end{prp}

\begin{proof}
Say we have $x,y\in S^0$ with $a=ax=ay$.  Then $a=ay=(ax)y=a(xy)$, by associativity.  In particular, $(x,y)\in S^2$ so $x=xy=y$, as $x,y\in S^0$.
\end{proof}

Likewise, there can be at most one $x\in S^0$ with $(x,a)\in S^2$.  When they exist, we call these the \emph{source unit} $\mathsf{s}(a)$ and \emph{range unit} $\mathsf{r}(a)$ respectively, i.e.
\[a\mathsf{s}(a)=a=\mathsf{r}(a)a.\]
For any $(a,b)\in S^0$, if $b$ has a source unit $\mathsf{s}(b)$ then $ab=a(b\mathsf{s}(b))=(ab)\mathsf{s}(b)$, showing that $\mathsf{s}(ab)=\mathsf{s}(b)$.  Likewise, if $a$ has a range unit then $\mathsf{r}(ab)=\mathsf{r}(a)$.

\begin{prp}
If $\mathsf{s}(a)$ and $\mathsf{r}(b)$ exist then
\[(a,b)\in S^2\qquad\Rightarrow\qquad\mathsf{s}(a)=\mathsf{r}(b).\]
\end{prp}

\begin{proof}
Note that $(a,b)\in S^2$ implies $ab=(a\mathsf{s}(a))b=a(\mathsf{s}(a)b)$, i.e. $(\mathsf{s}(a),b)\in S^2$, which then in turn implies $\mathsf{s}(a)b=\mathsf{s}(a)(\mathsf{r}(b)b)=(\mathsf{s}(a)\mathsf{r}(b))b$, i.e. $(\mathsf{s}(a),\mathsf{r}(b))\in S^2$, and hence $\mathsf{s}(a)=\mathsf{s}(a)\mathsf{r}(b)=\mathsf{r}(b)$.
\end{proof}

Categories are semigroupoids where the converse also holds.

\begin{dfn}
A \emph{category} is a semigroupoid where $\mathsf{s}(a)$ and $\mathsf{r}(a)$ always exist and
\begin{equation}\label{CategoryDef}
\mathsf{s}(a)=\mathsf{r}(b)\qquad\Rightarrow\qquad(a,b)\in S^2.
\end{equation}
\end{dfn}

In Fell bundles and C*-algebras, involutions also play a key role.

\begin{dfn}\label{*Semigroupoid}
A \emph{*-semigroupoid} is a semigroupoid $S$ with an involution, i.e. a map $a\mapsto a^*$ on $S$ such that $a^{**}=a$ and $(ab)^*=b^*a^*$, for all $(a,b)\in S^2$.

Likewise, a \emph{*-category}/\emph{*-semigroup} is a category/semigroup with an involution.
\end{dfn}

If $S$ is a semigroupoid, an \emph{inverse} of $a\in S$ is an element $a^{-1}\in S$ with $aa^{-1},a^{-1}a\in S^0$ (in particular, $(a,a^{-1}),(a^{-1},a)\in S^2$).  The \emph{core} of $S$ is
\[S^\times=\{a\in S:a\text{ has an inverse}\}.\]

\begin{dfn}
A \emph{groupoid} $G$ is a semigroupoid of invertibles, i.e. $G=G^\times$.
\end{dfn}

\begin{prp}
Every groupoid is a *-category when we take $g^*=g^{-1}$.
\end{prp}

\begin{proof}
Say $G$ is a groupoid, so every $g\in G$ has an inverse $g^{-1}$.  In particular, every $x\in G^0$ has an inverse $x^{-1}=xx^{-1}=x(xx^{-1})=x$.  Taking $x=g^{-1}g$, for some $g\in G$, it follows that $g^{-1}g=(g^{-1}g)(g^{-1}g)=g^{-1}(g(g^{-1}g))$ and hence $\mathsf{s}(g)=g^{-1}g$.  Likewise $\mathsf{r}(g)=gg^{-1}$ so, in particular, all elements have source and range units.  Moreover, if $\mathsf{s}(g)=\mathsf{r}(h)$ then $g=g\mathsf{s}(g)=g\mathsf{r}(h)=g(hh^{-1})=(gh)h^{-1}$ so $(g,h)\in G^2$, showing that $G$ is a category.  Now it follows that $g\mapsto g^{-1}$ is an involution which means $G$ is also a *-category when we take $g^*=g^{-1}$.
\end{proof}

A map $\rho:S\rightarrow T$ between semigroupoids $S$ and $T$ is a \emph{homomorphism} if
\[\tag{Homomorphism}(a,b)\in S^2\qquad\Rightarrow\qquad\rho(ab)=\rho(a)\rho(b).\]
In particular, $(a,b)\in S^2$ implies $(\rho(a),\rho(b))\in T^2$.  An \emph{isocofibration} is a homomorphism where the converse also holds, i.e. $(a,b)\in S^2$ whenever $(\rho(a),\rho(b))\in T^2$.  If $S$ and $T$ are *-semigroupoids and $\rho$ also respects the involution, i.e. $\rho(a^*)=\rho(a)^*$, then $\rho$ is a \emph{*-isocofibration}.  A \emph{functor} is a unital homomorphism, i.e. $\rho[S^0]\subseteq T^0$.

\begin{prp}\label{HomoFun}
Every homomorphism to a groupoid is automatically a functor.
\end{prp}

\begin{proof}
Say $S$ is a semigroupoid, $G$ is a groupoid and $\rho:S\rightarrow G$ is a homomorphism.  For any $a\in S^0$ or even any idempotent $a=aa$, we see that $\rho(a)=\rho(a)\rho(a)$ and hence $\rho(a)\rho(a)^{-1}=\rho(a)\rho(a)\rho(a)^{-1}=\rho(a)$, showing that $\rho(a)\in\Gamma^0$.
\end{proof}

\begin{cor}
Every homomorphism between groupoids is a *-functor.
\end{cor}

\begin{proof}
Say $G$ and $H$ are groupoids and $\rho:G\rightarrow H$ is a homomorphism.  For any $g\in G$, we have $\rho(g^{-1})\rho(g)=\rho(g^{-1}g)\in H^0$, by \autoref{HomoFun}, and hence
\[\rho(g^{-1})=\rho(g^{-1})\rho(g)\rho(g)^{-1}=\rho(g)^{-1}.\qedhere\]
\end{proof}

To define Fell bundles, we will also need the following topological notions.

\begin{dfn}\label{TopologicalGroupoids}
A \emph{topological category} is a category carrying a topology making the source $\mathsf{s}$, range $\mathsf{r}$ and the product (on $S^2$ as a subspace of $S\times S$) continuous.

A \emph{topological *-semigroupoid} is a *-semigroupoid $S$ carrying a topology making both the involution and product continuous.

A \emph{topological groupoid} is a groupoid that is topological as a *-semigroupoid with $a^*=a^{-1}$, i.e. where both the inverse and product are continuous.

An \emph{\'etale groupoid} is a topological groupoid where $\mathsf{s}$ (or $\mathsf{r}$) is an open map.
\end{dfn}

Recall that a \emph{slice} of a groupoid $\Gamma$ is a subset $S\subseteq\Gamma$ on which $\mathsf{s}$ and $\mathsf{r}$ are injective or, equivalently, such that $SS^{-1}\cup S^{-1}S\subseteq\Gamma^0$.  \'Etale groupoids are precisely those with a basis of open slices which is closed under pointwise products and inverses.

Once we have defined appropriate categories of Fell bundles and structured C*-algebras, we will then set about constructing functors between them and then natural transformations between those.

\begin{dfn}
Given semigroupoid homomorphisms $\rho,\sigma:S\rightarrow T$, a \emph{natural transformation} from $\sigma$ to $\rho$ is yet another function $\eta:S\rightarrow T$ such that
\[\eta(a)\sigma(b)=\rho(a)\eta(b),\]
for all $(a,b)\in S^2$, in which case the above products are required to be defined, i.e.
\[(\eta(a),\sigma(b)),(\rho(a),\eta(b))\in T^2.\]
\end{dfn}

In the above situation, if $S$ and $T$ are categories and $\rho$ and $\sigma$ are functors then
\[\eta(\mathsf{r}(a))\sigma(a)=\rho(\mathsf{r}(a))\eta(a)=\eta(a)=\sigma(\mathsf{s}(a))=\rho(a)\eta(\mathsf{s}(a)),\]
for all $(a,b)\in S^2$, showing that $\eta$ is uniquely defined by its values on $S^0$.  On the other hand, if we are given a function $\eta:S^0\rightarrow T$ such that, for all $a\in S$,
\[\eta(\mathsf{r}(a))\sigma(a)=\rho(a)\eta(\mathsf{s}(a)),\]
then we may extend $\eta$ to a function on $S$ defined by this common value, i.e.
\[\eta(a)=\eta(\mathsf{r}(a))\sigma(a)=\rho(a)\eta(\mathsf{s}(a)).\]
This extension is then a natural transformation as can be seen by the computation
\[\eta(a)\sigma(b)=\eta(\mathsf{r}(a))\sigma(ab)=\rho(ab)\eta(\mathsf{s}(b))=\rho(a)\eta(b).\]
Consequently, we can (and will) identify natural transformations between functors on categories with their unit restrictions.  A \emph{natural isomorphism} is a natural transformation taking invertible values on all units, i.e. $\eta[S^0]\subseteq T^\times$.

Our duality between certain Fell bundles and structured C*-algebras will take the form of an adjunction which, appropriately restricted, becomes an equivalence.

\begin{dfn}
An \emph{adjunction} of $S$ and $T$ is a quadruple $(\rho,\sigma,\varepsilon,\eta)$ where
\begin{enumerate}
\item $\sigma:S\rightarrow T$ and $\rho:T\rightarrow S$ are semigroupoid homomorphisms,
\item\label{AdjunctionCounit} $\varepsilon:S\rightarrow S$ is a natural transformation from $\rho\circ\sigma$ to $\mathrm{id}_S$,
\item\label{AdjunctionUnit} $\eta:T\rightarrow T$ is a natural transformation from $\mathrm{id}_T$ to $\sigma\circ\rho$, and
\item the zigzag identities hold, i.e. for all $(r,s)\in S^2$ and $(t,u)\in T^2$,
\[\tag{Zigzag}\label{Zigzag}\sigma(rs)=\sigma(\varepsilon(r))\eta(\sigma(s))\qquad\text{and}\qquad\rho(tu)=\varepsilon(\rho(t))\rho(\eta(u)).\]
\end{enumerate}
An \emph{equivalence} of $S$ and $T$ is an adjunction where $\varepsilon$ and $\eta$ are natural isomorphisms.
\end{dfn}

As with natural transformations, we can restrict our attention to units when dealing with functors between categories.

\begin{prp}
If $S$ and $T$ are categories and $\sigma$ and $\rho$ are functors then, for $(\rho,\sigma,\varepsilon,\eta)$ to be an adjunction of $S$ and $T$, it suffices to verify \eqref{Zigzag} on units.
\end{prp}

\begin{proof}
If \eqref{AdjunctionCounit} and \eqref{AdjunctionUnit} above hold then $\varepsilon(s)=s\varepsilon(\mathsf{s}(s))$, for all $s\in S$, and $\eta(t)=\eta(\mathsf{r}(t))t$, for all $t\in T$.  So if \eqref{Zigzag} holds on units then, for all $(r,s)\in S$,
\[\sigma(\varepsilon(r))\eta(\sigma(s))=\sigma(r\varepsilon(\mathsf{s}(r)))\eta(\mathsf{r}(\sigma(s))\sigma(s))=\sigma(r)\sigma(\varepsilon(x))\eta(\sigma(x))\sigma(s)=\sigma(r)\sigma(x)\sigma(s)=\sigma(rs),\]
where $x=\mathsf{s}(r)=\mathsf{r}(s)$, thus verifying the first part of \eqref{Zigzag} for general $(r,s)\in S^2$.  A dual argument works for the second part of \eqref{Zigzag}.
\end{proof}

\subsection{Banach Bundles}

Let us recall the definition of a Banach bundle.

\begin{dfn}
A \emph{bundle} is an open continuous surjection $\rho:B\twoheadrightarrow X$ between topological spaces $B$ and $X$.  We call $\rho$ a \emph{Banach bundle} if $X$ is Hausdorff and
\begin{enumerate}
\item Each fibre $B_x=\rho^{-1}\{x\}$ is a (complex) Banach space.
\item For each $\lambda\in\mathbb{C}$, $x\mapsto\lambda x$ is continuous on $B$.
\item Addition is continuous on $B\times_\rho B=\{(a,b):\rho(a)=\rho(b)\}$.
\item The norm is upper semicontinuous on $B$.
\item $\rho(b_\lambda)\rightarrow x$ and $\|b_\lambda\|\rightarrow0$ implies $b_\lambda\rightarrow0_x$.
\end{enumerate}
If the norm is even continuous then $\rho$ is a \emph{continuous Banach bundle}.
\end{dfn}

We will be particularly concerned with continuous sections of certain Banach bundles $\rho:B\twoheadrightarrow X$.  First let us denote the arbitrary/continuous sections of $\rho$ by
\begin{align*}
\mathcal{A}(\rho)&=\{a\in B^\Gamma:a\subseteq\rho^{-1}\}.\\
\mathcal{C}(\rho)&=\{a\in\mathcal{A}(\rho):a\text{ is continuous}\}.
\end{align*}
Note that these are vector spaces with respect to pointwise operations.  Consider the uniform norm on $\mathcal{A}(\rho)$ with values in $[0,\infty]$ defined by
\[\|a\|_\infty=\sup_{x\in X}\|a(x)\|.\]
As usual, this defines a metric $\|a-b\|_\infty$ again with values in $[0,\infty]$.  As each fibre $B_x=\rho^{-1}\{x\}$ is complete, $\mathcal{A}(\rho)$ is uniformly complete and hence so is $\mathcal{C}(\rho)$.

\begin{prp}
$\mathcal{C}(\rho)$ is uniformly closed in $\mathcal{A}(\rho)$, for any Banach bundle $\rho$.
\end{prp}

\begin{proof}
See \cite[Chapter II Corollary 13.13]{DoranFell1988}.
\end{proof}

Among continuous sections, those with relatively compact support
\[\mathrm{supp}(a)=\{x\in X:a(x)\neq0_x\}\]
and those vanishing at $\infty$ play an important role.  As usual, we denote these by
\begin{align*}
\mathcal{C}_\mathsf{c}(\rho)&=\{a\in\mathcal{C}(\rho):\mathrm{cl}(\mathrm{supp}(a))\text{ is compact}\}.\\
\mathcal{C}_0(\rho)&=\{a\in\mathcal{C}(\rho):a^{-1}\{b\in B:\|b\|\geq\delta\}\text{ is compact, for all }\delta>0\}.
\end{align*}

\begin{prp}\label{CcClosure}
If $\rho:B\twoheadrightarrow X$ is a Banach bundle, $\mathcal{C}_0(\rho)$ is uniformly closed.  If $X$ is also locally compact then, for any $a\in\mathcal{C}_0(\rho)$, we have $(a_n)\subseteq\mathcal{C}_\mathsf{c}(\rho)$ with
\[\|a-a_n\|_\infty\rightarrow0\qquad\text{and}\qquad\mathrm{supp}(a_n)\subseteq\mathrm{supp}(a_{n+1})\subseteq\mathrm{supp}(a),\]
for all $n\in\mathbb{N}$.  In particular, $\mathcal{C}_0(\rho)=\mathrm{cl}_\infty(\mathcal{C}_\mathsf{c}(\rho))$.
\end{prp}

\begin{proof}
Take $(a_n)\subseteq\mathcal{C}_0(\rho)$ with uniform limit $a$.  For any $\delta>0$, we have $n\in\mathbb{N}$ with $\|a-a_n\|_\infty<\delta$ and hence
\[a^{-1}\{b\in B:\|b\|\geq2\delta\}\subseteq a_n^{-1}\{b\in B:\|b\|\geq\delta\}.\]
As $a_n\in\mathcal{C}_0(\rho)$, this last set is compact.  This shows that $a\in\mathcal{C}_0(\rho)$ too, which in turn shows that $\mathcal{C}_0(\rho)$ is uniformly closed.

Now if $X$ is locally compact and $a\in\mathcal{C}_0(\rho)$ then, for each $n\in\mathbb{N}$, we have continuous $f_n:X\rightarrow[0,1]$ with relatively compact support which is $1$ on the compact set $a^{-1}\{b\in B:\|b\|\geq\frac{1}{n}\}$.  For each $n\in\mathbb{N}$, the pointwise supremum $\bigvee_{k=1}^nf_n$ then also has relatively compact support and hence so does $a_n$ defined by
\[a_n(x)=(\textstyle\bigvee_{k=1}^nf_n)(x)a(x).\]
Moreover, each $a_n$ is also continuous as scalar multiplication is continuous (as a function from $\mathbb{C}\times B$ to $B$ -- see \cite[Proposition 13.10]{DoranFell1988}, noting that the proof is still valid even if the norm is only upper semicontinuous).  Also $\|a-a_n\|_\infty\rightarrow0$ and $\mathrm{supp}(a_n)\subseteq\mathrm{supp}(a_{n+1})\subseteq\mathrm{supp}(a)$, for all $n\in\mathbb{N}$, as required.
\end{proof}

Let $\Subset$ denote compact containment, i.e.
\[Y\Subset Z\qquad\Leftrightarrow\qquad\mathrm{cl}(Y)\text{ is a compact subset of }Z.\]
When $\rho$ above is a continuous Banach bundle, we can even get $(a_n)\subseteq\mathcal{C}_\mathsf{c}(\rho)$ with
\[\mathrm{supp}(a_n)\Subset\mathrm{supp}(a_{n+1})\Subset\mathrm{supp}(a),\]
for all $n\in\mathbb{N}$, still with uniform limit $a$.  Indeed, then we can choose the each $f_n$ above to have support in the open set $a^{-1}\{b\in B:\|b\|>\frac{1}{n+1}\}$.

Another result about Banach bundles over locally compact base spaces is that they have `enough continuous sections', i.e. there will be continuous sections going through each element of the total space (see \cite[Appendix C]{DoranFell1988} for a proof due to A. Douady and L. dal Soglio-Herault -- although stated only for continuous Banach bundles, it still applies when the norm is only upper semicontinuous, as noted by Hofmann in \cite[Remark C.18]{DoranFell1988}).

\begin{thm}\label{CtsSections}
If $\rho:B\twoheadrightarrow X$ is a Banach bundle such that $X$ is locally compact then, for any $b\in B$, we have $a\in\mathcal{C}_\mathsf{c}(\rho)$ with $b\in\mathrm{ran}(b)$, i.e. $a(\rho(b))=b$.
\end{thm}

\begin{proof}
Take $a\in\mathcal{C}(\rho)$ with $a(\rho(b))=b$.  Again we have $f\in\mathcal{C}_\mathsf{c}(X)(=\mathcal{C}_\mathsf{c}(\tau)$ for the trivial complex line bundle $\tau$ over $X$) such that $f(\rho(b))=1$.  Then defining $fa(x)=f(x)a(x)$, we again see that $fa\in\mathcal{C}_\mathsf{c}(\rho)$ and $fa(\rho(b))=b$.
\end{proof}

\subsection{Fell Bundles}

We are primarily interested in Banach bundles where both the base space and the total space have some compatible *-semigroupoid structure.

\begin{dfn}
A \emph{Fell bundle} is a Banach bundle $\rho:B\twoheadrightarrow\Gamma$ such that, moreover,
\begin{enumerate}
\item $B$ is a topological *-semigroupoid with bilinear products and antilinear *.
\item $\Gamma$ is a locally compact \'etale groupoid and $\rho$ is a *-isocofibration.
\item $\|\cdot\|$ is submultiplicative and, for all $b\in B$,
\[\|b^*b\|=\|b\|^2.\]
\item For all $b\in B$, we have $a\in B$ with
\[\rho(a)=\rho(b^*b)\qquad\text{and}\qquad a^*a=b^*b.\]
\end{enumerate}
\end{dfn}

\begin{rmk}
Fell bundles were originally defined over groups \textendash\, their extension to groupoids was elucidated by Kumjian in \cite{Kumjian1998}.  For us, it is convenient to restrict to base groupoids that are locally compact and \'etale, even though one could certainly consider more general topological groupoids with Haar systems.
\end{rmk}

Every Fell bundle $\rho:B\twoheadrightarrow\Gamma$ is a Banach bundle and hence the base groupoid $\Gamma$ is Hausdorff by definition.  However, the total space $B$ need not Hausdorff.

\begin{xpl}
Let $\rho:B\twoheadrightarrow\mathbb{N}\cup\{\infty\}$ be the corical Fell bundle whose range/base space is the one-point compactification of $\mathbb{N}$ (considered as groupoid with trivial product) with $1$-dimensional fibres on $\mathbb{N}$ and a $2$-dimensional fibre at $\infty$, i.e.
\[B=(\mathbb{N}\times\mathbb{C})\cup(\{\infty\}\times(\mathbb{C}\oplus\mathbb{C})),\]
where $\mathbb{N}\times\mathbb{C}$ has the usual product topology but where, for any sequence $(\lambda_n)\subseteq\mathbb{C}$,
\begin{align*}
(2n,\lambda_n)\rightarrow(\infty,(\alpha,\beta))\qquad&\Leftrightarrow\qquad\lambda_n\rightarrow\alpha.\\
(2n+1,\lambda_n)\rightarrow(\infty,(\alpha,\beta))\qquad&\Leftrightarrow\qquad\lambda_n\rightarrow\beta.
\end{align*}
In particular, $(2n,0)\rightarrow(\infty,(0,\gamma))$ and $(2n+1,0)\rightarrow(\infty,(\gamma,0))$, for all $\gamma\in\mathbb{C}$, so $B$ is certainly not Hausdorff.  Taking $\gamma\neq0$, we also see that the norm not continuous, only upper semicontinuous.  This is no accident, as the total space of any continuous Banach bundle is automatically Hausdorff \textendash\, see \cite[Proposition 16.4]{Gierz1982}.
\end{xpl}

We will be particularly interested in the following kinds of Fell bundles.

\begin{dfn}\label{CategoricalCorical}
A Fell bundle $\rho\!:\!B\!\twoheadrightarrow\!\Gamma$ is \emph{categorical} if $B$ is a topological category.  We call $\rho$ \emph{corical} if the core $B^\times$ is a topological groupoid, open and $\rho[B^\times]=\Gamma$.
\end{dfn}

\begin{prp}\label{CategoricalChars}
For any Fell bundle $\rho:B\twoheadrightarrow\Gamma$, the following are equivalent.
\begin{enumerate}
\item\label{TopCat} $\rho$ is categorical.
\item\label{Homeo} $\rho|_{B^0}$ is a homeomorphism onto $\Gamma^0$.
\item\label{CBundle} $\rho|_{\mathbb{C}B^0}$ is a trivial C*-bundle over $\Gamma^0$, i.e.
\begin{equation}\label{CB0}
\lambda b\mapsto(\lambda,\rho(b)),
\end{equation}
for $\lambda\in\mathbb{C}$ and $b\in B^0$, is an isomorphism from $\mathbb{C}B^0$ onto $\mathbb{C}\times\Gamma^0$.
\end{enumerate}
\end{prp}

\begin{proof}\
\begin{itemize}
\item[\eqref{TopCat}$\Rightarrow$\eqref{Homeo}] By the last defining property of a Banach bundle, the map $\gamma\mapsto0_\gamma$ is continuous on $\Gamma$.  If $B$ is a topological category then the function $s$ defined by $s(\gamma)=\mathsf{s}(0_\gamma)$ is continuous on $\Gamma$.  As $\rho$ is a functor, for any $x\in\Gamma^0$,
\[\qquad\qquad x=\rho(0_x)=\rho(0_x\mathsf{s}(0_x))=\rho(0_x)\rho(\mathsf{s}(0_x))=x\rho(\mathsf{s}(0_x))=\rho(\mathsf{s}(0_x))=\rho(s(x)).\]
Also $b=\mathsf{s}(0_{\rho(b)})=s(\rho(b))$, for any $b\in B^0$, as $\rho$ is an isocofibration.  Thus $s|_{\Gamma^0}=\rho|_{B^0}^{-1}$ and hence $\rho|_{B^0}$ is a homeomorphism onto $\Gamma^0$.

\item[\eqref{Homeo}$\Rightarrow$\eqref{TopCat}] If $\rho$ restricted to $B^0$ is a homeomorphism onto $\Gamma^0$ then its inverse is a continuous map from $\Gamma^0$ onto $B^0$.  For any $b\in B$, note $(\rho(b),\mathsf{s}(\rho(b)))\in\Gamma^2$ so $(b,\rho|_{B^0}^{-1}(\mathsf{s}(\rho(b))))\in B^2$, as $\rho$ is an isocofibration, i.e. $\mathsf{s}(b)=\rho|_{B^0}^{-1}(\mathsf{s}(\rho(b)))$.  Likewise, $\mathsf{r}(b)=\rho|_{B^0}^{-1}(\mathsf{r}(\rho(b)))$, showing that $B$ has source and range maps that are defined and continuous everywhere, i.e. $B$ is a topological category.

\item[\eqref{Homeo}$\Rightarrow$\eqref{CBundle}] We immediately see that \eqref{CB0} is an algebraic isomorphism from $\mathbb{C}B^0$ to $\mathbb{C}\times\Gamma^0$.  To see that \eqref{CB0} is also a homeomorphism, take nets $(\lambda_n)\subseteq\mathbb{C}$ and $(b_n)\subseteq B^0$.  If $\lambda_nb_n\rightarrow\lambda b$, for some $\lambda\in\mathbb{C}$ and $b\in B^0$, then
\[\rho(b_n)=\rho(\lambda_nb_n)\rightarrow\rho(\lambda b)=\rho(b).\]
If $\rho|_{B^0}$ is a homeomorphism onto $\Gamma^0$ then this implies $b_n\rightarrow b$ and hence $(\lambda_n-\lambda)b_n\rightarrow\lambda b-\lambda b=0_{\rho(b)}$.  As the norm is upper semicontinuous, $|\lambda_n-\lambda|=\|(\lambda_n-\lambda)b_n\|\rightarrow0$, i.e. $\lambda_n\rightarrow\lambda$ so $(\lambda_n,\rho(b_n))\rightarrow(\lambda,\rho(b))$, showing that \eqref{CB0} is continuous.  On the other hand, if $\lambda_n\rightarrow\lambda$ and $\rho(b_n)\rightarrow\gamma\in\Gamma^0$ then $b_n\rightarrow\rho|_{B^0}^{-1}(\gamma)$ and hence $\lambda_nb_n\rightarrow\lambda\rho|_{B^0}^{-1}(\gamma)$, showing that the inverse map is also continuous, i.e. \eqref{CB0} is a homeomorphism.

\item[\eqref{CBundle}$\Rightarrow$\eqref{Homeo}] If \eqref{CB0} is an isomorphism then, in particular, $b\mapsto(1,\rho(b))$ and hence $b\mapsto\rho(b)$ is a homeomorphism on $B^0$. \qedhere
\end{itemize}
\end{proof}

\begin{cor}
Every corical Fell bundle $\rho:B\twoheadrightarrow\Gamma$ is categorical.
\end{cor}

\begin{proof}
For any $b\in B^0$, we have a continuous section $a$ of $\rho$ with $a(\rho(b))=b$.  As $B^\times$ is open, so is $O=a^{-1}[B^\times]\cap\Gamma^0\ni\rho(b)$.  As $B^\times$ is a topological groupoid, $x\mapsto\mathsf{s}(a(x))=\mathsf{r}(a(x))$ is a continuous inverse of $\rho|_{B^0\cap\rho^{-1}[O]}$.  This shows that $\rho|_{B^0}$ is a homeomorphism onto $\rho[B^0]=\Gamma^0$, as $\rho[B^\times]=\Gamma$, i.e. $\rho$ is categorical.
\end{proof}

Note $\rho[B^\times]=\Gamma$ alone is not enough to ensure that $\rho$ is corical or even categorical.  For example, consider the subbundle $\rho:B\rightarrow[0,1]$ of the trivial Fell bundle of $2\times2$ matrices over the unit interval $[0,1]$ where the fibre at $1$ is the $1$-dimensional subspace of matrices with $0$ entries in all but the top left corner, i.e.
\[B=\{(x,m)\in[0,1]\times M_2:x=1\Rightarrow m\in\mathbb{C}e_{11}\}.\]
Then $B$ still has a unit in the fibre at $1$, namely $(1,e_{11})$, so $\rho[B^\times]=\rho[B^0]=\Gamma$.  But $(1,e_{11})\neq\lim_{x\rightarrow1}(x,e_{11}+e_{22})$ so $\rho$ restricted to $B^0$ is not a homeomorphism.

However, $\rho[B^\times]=\Gamma$ is enough to ensure that $\rho$ is \emph{saturated}, i.e.
\[\tag{Saturated}B_{\alpha\beta}=B_\alpha B_\beta,\]
for all $(\alpha,\beta)\in\Gamma^2$, where $B_\gamma=\rho^{-1}\{\gamma\}$.  To see this note that, for any $(\alpha,\beta)\in\Gamma^2$ and $b\in B_{\alpha\beta}$, if $\rho[B^\times]=\Gamma$ then we have $a\in B_\alpha^\times$ and hence $b=aa^{-1}b\in B_\alpha B_\beta$.  Thus coricality can be viewed as a strong saturation condition.

\section{Arbitrary Sections}

Next we consider sections of Fell bundles and the various algebras and modules they give rise to, like the reduced C*-algebra $\mathcal{C}_\mathsf{r}(\rho)$.  The standard approach is to take $\mathcal{C}_\mathsf{r}(\rho)$ to be the abstract completion of the compactly supported continuous sections $\mathcal{C}_\mathsf{c}(\rho)$ with the left-regular norm.  One then defines a linear map $j$ taking elements in the completion to concrete continuous sections again (see \cite[Proposition 3.3.3]{Sims2017} or \cite[Theorem 3.4.1]{Putnam2019}).  We take a slightly different approach where we first define larger Banach spaces of concrete sections of $\rho$.  Within these spaces, we then take closures instead of completions and thus do away with the linear map $j$.

To avoid repeating our basic hypotheses we make the following assumption. 
\begin{center}
\textbf{Throughout the rest of this section $\rho:B\twoheadrightarrow\Gamma$ is a Fell bundle.}
\end{center}
At first we focus on the norm-algebraic structure of $\rho$, ignoring the topology.  Recall $\mathcal{A}(\rho)$ denotes arbitrary sections of $\rho$.
For any $a\in\mathcal{A}(\rho)$, we define $a^*\in\mathcal{A}(\rho)$ by
\[a^*(\gamma)=a(\gamma^{-1})^*.\]
For any $Y\subseteq\Gamma$, we also define the `restriction' $a_Y\in\mathcal{A}(\rho)$ by
\[a_Y(\gamma)=\begin{cases}a(\gamma)&\text{if }\gamma\in Y\\0_\gamma&\text{if }\gamma\notin Y.\end{cases}\]
So $a_Y$ is the same as the usual restriction $a|_Y$ but with $0$ values outside $Y$.  In particular, we let $\Phi$ denote the restriction of any $a\in\mathcal{A}(\rho)$ to the diagonal, i.e.
\[\Phi(a)=a_{\Gamma^0}.\]

Let us denote finite subsets by $\subset$, i.e.
\[F\subset\Lambda\qquad\Leftrightarrow\qquad F\subseteq\Lambda\text{ and }|F|<\infty\]
(which could also be viewed as compact containment w.r.t. the discrete topology).  Note $\{F:F\subset\Lambda\}$ is directed by inclusion $\subseteq$.  If $V$ is a normed space, $(v_\lambda)_{\lambda\in\Lambda}\subseteq V$ and the net of finite partial sums $(\sum_{\lambda\in F}v_\lambda)_{F\subset\Lambda}$ is norm-convergent then we denote the limit by $\sum_{\lambda\in\Lambda}v_\lambda$.  Whenever possible, we define $ab\in\mathcal{A}(\rho)$ from $a,b\in\mathcal{A}(\rho)$ by
\[ab(\gamma)=\sum_{\gamma=\alpha\beta}a(\alpha)b(\beta),\]
i.e. whenever the finite partial sums $ab_F(\gamma)=\sum_{\beta\in F}a(\gamma\beta^{-1})b(\beta)$, for $F\subset\Gamma\gamma$, converge in each fibre $B_\gamma=\rho^{-1}\{\gamma\}$.  Note $ab_Y$ denotes the product of $a$ with $b_Y$, while $(ab)_Y$ denotes the restriction of $ab$ to $Y$ (if these products are defined).

\begin{prp}
Take any $F\subset\Gamma$ and $a,b,c\in\mathcal{A}(\rho)$.  If $ab$ is defined then
\begin{equation}\label{abcF}
(ab)c_F=a(bc_F).
\end{equation}
Likewise, if $bc$ is defined then $(a_Fb)c=a_F(bc)$.
\end{prp}

\begin{proof}
As $F$ is finite and $ab$ is defined, so is $(ab)c_F$.  Specifically, for any $\gamma\in\Gamma$,
\begin{align*}
((ab)c_F)(\gamma)&=\sum_{\beta\in F\cap\Gamma\gamma}ab(\gamma\beta^{-1})c(\beta)\\
&=\sum_{\beta\in F\cap\Gamma\gamma}\lim_{G\subset\gamma\Gamma}\sum_{\alpha\in G}a(\alpha)b(\alpha^{-1}\gamma\beta^{-1})c(\beta)\\
&=\lim_{G\subset\gamma\Gamma}\sum_{\beta\in F\cap\Gamma\gamma}\sum_{\alpha\in G}a(\alpha)b(\alpha^{-1}\gamma\beta^{-1})c(\beta)\\
&=\lim_{G\subset\gamma\Gamma}\sum_{\alpha\in G}a(\alpha)\sum_{\beta\in F\cap\Gamma\gamma}b(\alpha^{-1}\gamma\beta^{-1})c(\beta)\\
&=\lim_{G\subset\gamma\Gamma}\sum_{\alpha\in G}a(\alpha)bc_F(\alpha^{-1}\gamma)\\
&=(a(bc_F))(\gamma)
\end{align*}
This shows that $a(bc_F)$ is also defined and equal to $(ab)c_F$ everywhere on $\Gamma$.  The second statement follows by a dual argument.
\end{proof}

Let us denote the finitely supported sections by
\[\mathcal{F}(\rho)=\{f\in\mathcal{A}(\rho):|\mathrm{supp}(f)|<\infty\}.\]
The Cauchy-Schwarz inequality (see \cite[Lemma 2.5]{RaeburnWilliams1998}) yields the following.

\begin{lem}
For any $f,g\in\mathcal{F}(\rho)$ and $\gamma\in\Gamma$.
\begin{equation}\label{gammaCS}
fg(\gamma)^*fg(\gamma)\leq\|fg(\gamma)\|\sqrt{\|ff^*(\mathsf{r}(\gamma))\|g^*g(\mathsf{s}(\gamma))},
\end{equation}
where $\leq$ denotes the usual ordering in the C*-algebra $B_{\mathsf{s}(\gamma)}=\rho^{-1}\{\mathsf{s}(\gamma)\}$.
\end{lem}

\begin{proof}
Cauchy-Schwarz says that, for any $j$ and $k$ in a Hilbert module $H$,
\[\langle j,k\rangle^*\langle j,k\rangle\leq\|\langle j,j\rangle\|\langle k,k\rangle.\]
As $a^2\leq b^2$ implies $a\leq b$, for any positive $a$ and $b$ in a C*-algebra (see \cite[II.3.1.10]{Blackadar2017}), it follows that, whenever $\langle j,k\rangle$ is positive,
\begin{equation}\label{CSsqrt}
\langle j,k\rangle\leq\sqrt{\|\langle j,j\rangle\|\langle k,k\rangle}.
\end{equation}
In particular, we can consider the Hilbert $B_{\mathsf{s}(\gamma)}$-module
\[H=\{h\in\mathcal{A}(\rho):\mathrm{supp}(h)\subseteq\mathrm{supp}(g)\cap\Gamma\gamma\}\]
where $\langle j,k\rangle=j^*k(\mathsf{s}(\gamma))$.  By \eqref{CSsqrt} applied in $H$,
\begin{align*}
fg(\gamma)^*fg(\gamma)&=fg(\gamma)^*\sum_{\alpha\beta=\gamma}f(\alpha)g(\beta).\\
&=\sum_{\alpha\beta=\gamma}fg(\gamma)^*f(\alpha)g(\beta).\\
&\leq\sqrt{\Big\|\sum_{\alpha\in\gamma\Gamma}fg(\gamma)^*f(\alpha)f(\alpha)^*fg(\gamma)\Big\|\sum_{\beta\in\Gamma\gamma}g(\beta)^*g(\beta)}.\\
&\leq\|fg(\gamma)\|\sqrt{\|ff^*(\mathsf{r}(\gamma))\|g^*g(\mathsf{s}(\gamma))}.\qedhere
\end{align*}
\end{proof}

Recall $\|\cdot\|_\infty$ is the uniform norm.  The $[0,\infty]$-valued $2$-norm on $\mathcal{A}(\rho)$ is given by
\[\|a\|_2=\sup_{F\subset\Gamma}\sqrt{\|\Phi(a^*a_F)\|}_\infty.\]

\begin{prp}
For any $a,b\in\mathcal{A}(\rho)$ such that $a^*b$ is defined,
\begin{equation}\label{infty22}
\|a^*b\|_\infty\leq\|a\|_2\|b\|_2.
\end{equation}
\end{prp}

\begin{proof}
For any $\gamma\in\Gamma$ and $F\subset\Gamma\gamma$, \eqref{gammaCS} applied to $a^*b_F(\gamma)=(a_{F\gamma^{-1}})^*b_F(\gamma)$ yields
\[\|a^*b_F(\gamma)\|^2\leq\|ab_F(\gamma)\|\sqrt{\|a^*a_{F\gamma^{-1}}(\mathsf{r}(\gamma))\|\|b^*b_F(\mathsf{s}(\gamma))\|}\leq\|ab_F(\gamma)\|\|a\|_2\|b\|_2.\]
Thus $\|a^*b(\gamma)\|=\lim_{F\subset\Gamma\gamma}\|a^*b_F(\gamma)\|\leq\|a\|_2\|b\|_2$, for all $\gamma\in\Gamma$, proving \eqref{infty22}.
\end{proof}

We define two more $[0,\infty]$-valued norms on $\mathcal{A}(\rho)$ as follows (taking $0/0=0$).
\begin{align*}
\|a\|_\mathsf{B}&=\sup_{f\in\mathcal{F}(\rho)}\frac{\|af\|_2}{\|f\|_2}.\\
\|a\|_\mathsf{b}&=\sup_{f,g\in\mathcal{F}(\rho)}\frac{\|\Phi(g^*af)\|_\infty}{\|g\|_2\|f\|_2}.
\end{align*}

\begin{prp}\label{infty2bprp}
For all $a\in\mathcal{A}(\rho)$,
\begin{equation}\label{infty2b}
\|a\|_\infty\leq\|a\|_2\leq\|a\|_\mathsf{B}=\|a\|_\mathsf{b}=\|a^*\|_\mathsf{b}.
\end{equation}
If $\mathrm{supp}(a)$ is a slice of $\Gamma$ then all these norms coincide.
\end{prp}

\begin{proof}
For all $\gamma\in\Gamma$, the C*-condition on Fell bundles yields
\[\|a(\gamma)\|^2=\|a(\gamma)^*a(\gamma)\|=\|\Phi(a^*a_\gamma)\|_\infty\leq\|a\|_2.\]
Taking suprema, we get $\|a\|_\infty\leq\|a\|_2$.  Also, for any $F\subset\Gamma$,
\[\|\Phi(a^*a_F)\|_\infty\leq\|a^*a_F\|_\infty\leq\|a^*a_F\|_2\leq\|a^*\|_\mathsf{B}\|a_F\|_2\leq\|a^*\|_\mathsf{B}\|a\|_2.\]
Taking suprema again, we get $\|a\|_2^2\leq\|a^*\|_\mathsf{B}\|a\|_2$ and hence $\|a\|_2\leq\|a^*\|_\mathsf{B}$.

For any $f\in\mathcal{F}(\rho)$,
\[\|af\|_2^2=\sup_{F\subset\Gamma}\|\Phi((af)^*_Faf)\|_\infty\leq\sup_{F\subset\Gamma}\|(af)_F\|_2\|a\|_\mathsf{b}\|f\|_2=\|af\|_2\|a\|_\mathsf{b}\|f\|_2.\]
Thus $\|af\|_2\leq\|a\|_\mathsf{b}\|f\|_2$ and hence $\|a\|_\mathsf{B}\leq\|a\|_\mathsf{b}$.  Conversely, for any $f,g\in\mathcal{F}(\rho)$,
\[\|\Phi(g^*af)\|_\infty\leq\|g^*af\|_\infty\leq\|g\|_2\|af\|_2\leq\|g\|_2\|a\|_\mathsf{B}\|f\|_2,\]
by \eqref{infty22}, and hence $\|a\|_\mathsf{b}\leq\|a\|_\mathsf{B}$.  Moreover, for any $b\in B$, $\|b^*\|=\|b\|$ and hence
\[\|a^*\|_\mathsf{b}=\sup_{f,g\in\mathcal{F}(\rho)}\frac{\|\Phi(g^*a^*f)\|_\infty}{\|g\|_2\|f\|_2}=\sup_{g,f\in\mathcal{F}(\rho)}\frac{\|\Phi(f^*ag)^*\|_\infty}{\|f\|_2\|g\|_2}=\|a\|_\mathsf{b}.\]
This proves \eqref{infty2b}.  Lastly note that if $\mathrm{supp}(a)$ is a slice then $\|af\|_2\leq\|a\|_\infty\|f\|_2$, for all $f\in\mathcal{F}(\rho)$, and hence $\|a\|_\mathsf{B}\leq\|a\|_\infty$, so all the norms coincide.
\end{proof}

%

From now on we usually refer to $\|\cdot\|_\mathsf{b}$ even when using the definition of $\|\cdot\|_\mathsf{B}$.

\begin{prp}
For any $a,b\in\mathcal{A}(\rho)$ such that $ab$ is defined,
\begin{equation}\label{||ab||_b}
\|ab\|_\mathsf{b}\leq\|a\|_\mathsf{b}\|b\|_\mathsf{b}.
\end{equation}
\end{prp}

\begin{proof}
For any $f,g\in\mathcal{F}(\rho)$ and $x\in\Gamma^0$,
\[\|gabf(x)\|=\|\Phi(ga(bf)_{\Gamma x})\|_\infty=\lim_{F\subset\Gamma x}\|\Phi(ga(bf)_F)\|_\infty.\]
However, for any $F\subset\Gamma x$,
\[\|\Phi(ga(bf)_F)\|_\infty\leq\|g\|_2\|a\|_\mathsf{b}\|(bf)_F\|_2\leq\|g\|_2\|a\|_\mathsf{b}\|bf\|_2\leq\|g\|_2\|a\|_\mathsf{b}\|b\|_\mathsf{b}\|f\|_2.\]
It follows that
\[\|ab\|_\mathsf{b}=\sup_{f,g\in\mathcal{F}(\rho)}\frac{\|\Phi(gabf)\|_\infty}{\|g\|_2\|f\|_2}=\sup_{\substack{f,g\in\mathcal{F}(\rho)\\ x\in\Gamma^0}}\frac{\|gabf(x)\|}{\|g\|_2\|f\|_2}\leq\|a\|_\mathsf{b}\|b\|_\mathsf{b}.\qedhere\]
\end{proof}

For $k=\infty$, $2$ and $\mathsf{b}$, we denote the subspaces of $\mathcal{A}(\rho)$ with finite $k$-norm by
\[\mathcal{A}_k(\rho)=\{a\in\mathcal{A}(\rho):\|a\|_k<\infty\}.\]
By a \emph{contraction} we mean a linear map with operator norm at most $1$.

\begin{thm}\label{kBanach}
$\mathcal{A}_k(\rho)$ is a Banach space w.r.t. $\|\cdot\|_k$ on which $\Phi$ is a contraction, for $k=\infty$, $2$ and $\mathsf{b}$.
\end{thm}

\begin{proof}
As each fibre of $\rho$ is a Banach space, so is $\mathcal{A}_\infty(\rho)$.  Now say $(a_n)\subseteq\mathcal{A}_2(\rho)$ is $2$-Cauchy.  By \eqref{infty2b}, $(a_n)$ is $\infty$-Cauchy and hence has an $\infty$-limit $a\in\mathcal{A}_\infty(\rho)$, i.e. $\|a-a_n\|_\infty\rightarrow0$.  For any $F\subset\Gamma$, it follows that
\[\|(a-a_n)_F\|_2=\lim_{m\rightarrow\infty}\|(a_m-a_n)_F\|_2\leq\lim_{m\rightarrow\infty}\|a_m-a_n\|_2.\]
Taking the supremum, it follows that $\|a-a_n\|_2\leq\lim_{m\rightarrow\infty}\|a_m-a_n\|_2$ and hence $\lim_{n\rightarrow\infty}\|a-a_n\|_2=0$, as $(a_n)$ is $2$-Cauchy.  In particular, $\|a-a_n\|_2<\infty$, for some $n$, so $\|a\|_2\leq\|a-a_n\|_2+\|a_n\|_2<\infty$ (subadditivity of the $2$-norm follows from Cauchy-Schwarz).  Thus $a\in\mathcal{A}_2(\rho)$ is a $2$-limit of $(a_n)$, showing that $\mathcal{A}_2(\rho)$ is $2$-complete and hence a Banach space w.r.t. $\|\cdot\|_2$.

If $(a_n)\subseteq\mathcal{A}_\mathsf{b}(\rho)$ is $\mathsf{b}$-Cauchy and hence $2$-Cauchy then we have $a\in\mathcal{A}_2(\rho)$ with $\|a-a_n\|_2\rightarrow0$.  For any $f\in\mathcal{F}(\rho)$, it follows that $\|(a-a_n)f\|_2\rightarrow0$ and hence
\[\|(a-a_n)f\|_2=\lim_{m\rightarrow\infty}\|(a_m-a_n)f\|_2\leq\lim_{m\rightarrow\infty}\|a_m-a_n\|_\mathsf{b}\|f\|_2.\]
Thus $\|a-a_n\|_\mathsf{b}\leq\lim_{m\rightarrow\infty}\|a_m-a_n\|_\mathsf{b}$ and hence $\lim_{n\rightarrow\infty}\|a-a_n\|_\mathsf{b}=0$, as $(a_n)$ is $\mathsf{b}$-Cauchy.  By subadditivity again it follows that $a\in\mathcal{A}_\mathsf{b}(\rho)$ is a $\mathsf{b}$-limit of $(a_n)$, showing that $\mathcal{A}_\mathsf{b}(\rho)$ is $\mathsf{b}$-complete and hence a Banach space w.r.t. $\|\cdot\|_\mathsf{b}$.

To see that $\Phi$ is a contraction on $\mathcal{A}_k(\rho)$, for $k=\infty$, $2$ and $\mathsf{b}$, just note that
\[\|\Phi(a)\|_\mathsf{b}=\|\Phi(a)\|_2=\|\Phi(a)\|_\infty\leq\|a\|_\infty\leq\|a\|_2\leq\|a\|_\mathsf{b}.\qedhere\]
\end{proof}

\begin{xpl}\label{NN}
If $\Gamma=\mathbb{N}\times\mathbb{N}$ is the discrete groupoid arising from full equivalence relation on $\mathbb{N}$, and $\rho:\mathbb{C}\times\Gamma\twoheadrightarrow\Gamma$ is the trivial complex line bundle over $\Gamma$ then every element of $\mathcal{A}(\rho)$ is an $\mathbb{N}\times\mathbb{N}$ matrix with entries in $\mathbb{C}$.  In this case, $\mathcal{A}_\infty(\rho)$ consists of the matrices whose entries form a bounded subset of $\mathbb{C}$, $\mathcal{A}_2(\rho)$ consists of matrices whose columns form a bounded subset of the Hilbert space $\ell^2$, while $\mathcal{A}_\mathsf{b}(\rho)$ consists of matrices which correspond to bounded operators on $\ell^2$.
\end{xpl}

\begin{lem}\label{a*aDefined}
If $a\in\mathcal{A}(\rho)$ then $a^*a$ is defined precisely when, for each $x\in\Gamma^0$,
\begin{equation}\label{a*a}
\lim_{F\subset\Gamma x}\|a_{\Gamma x\setminus F}\|_2=0,
\end{equation}
in which case $\|a\|_2=\sqrt{\|a^*a\|_\infty}=\sqrt{\|\Phi(a^*a)\|_\infty}$ and $\|a\|_\mathsf{b}=\sqrt{\|a^*a\|_\mathsf{b}}$.
\end{lem}

\begin{proof}
If $a^*a$ is defined then, for each $x\in\Gamma^0$,
\[\lim_{F\subset\Gamma x}\|a_{\Gamma x\setminus F}\|_2=\lim_{F\subset\Gamma x}\|(a^*a-a^*a_F)(x)\|=0,\]
i.e. \eqref{a*a} holds.  Conversely, say \eqref{a*a} holds, for all $x\in X$, and take $\gamma\in\Gamma$.  Whenever $F\subseteq G\subset\Gamma\gamma$, \eqref{infty22} yields
\[\|(a^*a_G-a^*a_F)(\gamma)\|=\|a^*a_{G\setminus F}(\gamma)\|\leq\|a_{(G\setminus F)\gamma^{-1}}\|_2\|a_{G\setminus F}\|_2\leq\|a_{\Gamma\gamma^{-1}}\|_2\|a_{\Gamma\gamma\setminus F}\|_2.\]
As $\lim_{F\subset\Gamma x}\|a_{\Gamma x\setminus F}\|_2=0$, for $x=\mathsf{s}(\gamma)$ and $\mathsf{r}(\gamma)$, this shows that $(a^*a_F(\gamma))_{F\subset\Gamma\gamma}$ is Cauchy and hence converges, as $B_\gamma=\rho^{-1}\{\gamma\}$ is a Banach space.  As $\gamma$ was arbitrary, this shows that $a^*a$ is defined.  Then $\|a^*a(\gamma)\|\leq\|a_{\Gamma\gamma^{-1}}\|_2\|a_{\Gamma\gamma}\|_2$, by \eqref{infty22} again, so taking suprema yields
\[\sqrt{\|a^*a\|_\infty}\leq\|a\|_2\leq\sqrt{\|\Phi(a^*a)\|_\infty}\leq\sqrt{\|a^*a\|_\infty}.\]
Finally, for any $f\in\mathcal{F}(\rho)$,
\[\|af\|_2=\sqrt{\|\Phi(f^*a^*af)\|_\infty}\leq\sqrt{\|a^*a\|_\mathsf{b}}\|f\|_2,\]
from which it follows that $\|a\|_\mathsf{b}\leq\sqrt{\|a^*a\|_\mathsf{b}}\leq\|a\|_\mathsf{b}$, by \eqref{infty2b} and \eqref{||ab||_b}.
\end{proof}

\begin{prp}\label{a*b}
If $a\in\mathcal{A}_2(\rho)$ and $b^*b$ is defined then so is $a^*b$.
\end{prp}

\begin{proof}
For all $\gamma\in\Gamma$ and $F\subseteq G\subset\Gamma\gamma$, \eqref{infty22} yields
\[\|a^*b_G(\gamma)-a^*b_F(\gamma)\|=\|a^*b_{G\setminus F}(\gamma)\|\leq\|a_{(G\setminus F)\gamma^{-1}}\|_2\|b_{G\setminus F}\|_2\leq\|a\|_2\|b_{\Gamma\gamma\setminus F}\|_2.\]
It follows that $(a^*b_F(\gamma))_{F\subset\Gamma\gamma}$ is Cauchy, as $\lim_{F\subset\Gamma\gamma}\|b_{\Gamma\gamma\setminus F}\|_2=0$, by \eqref{a*a}.  As $B_\gamma$ is a Banach space, it follows that $a^*b(\gamma)$ is defined, for all $\gamma\in\Gamma$.
\end{proof}

We define
\begin{align*}
\mathcal{H}(\rho)&=\{a\in\mathcal{A}_2(\rho):a^*a\text{ is defined}\}.\\
\mathcal{D}(\rho)&=\{a\in\mathcal{A}_2(\rho):\mathrm{supp}(a)\subseteq\Gamma^0\}.
\end{align*}
Actually, for $\mathcal{D}(\rho)$, we could replace $2$ with $\infty$ or $\mathsf{b}$, by \autoref{infty2bprp}.

\begin{thm}\label{Hilbert}
$\mathcal{H}(\rho)$ is a right Hilbert module over $\mathcal{D}(\rho)$ with inner product
\[\langle a,b\rangle=\Phi(a^*b).\]
\end{thm}

\begin{proof}
By \autoref{a*b}, the given inner product is well-defined with values in $\mathcal{D}(\rho)$.  Moreover, for any $a,b\in\mathcal{H}(\rho)$, $(a+b)^*(a+b)=a^*a+a^*b+b^*a+b^*b$ is defined, again by \autoref{a*b}, and hence $a+b\in\mathcal{H}(\rho)$.  For any $d\in\mathcal{D}(\rho)$, we see that $bd$ and $(bd)^*bd=d^*b^*bd$ are also defined and $\|bd\|_2\leq\|b\|_2\|d\|_\infty$, which means that $bd\in\mathcal{H}(\rho)$.  Moreover, as $\mathrm{supp}(d)\subseteq\Gamma^0$, we also see that
\[\langle a,bd\rangle=\Phi(a^*bd)=\Phi(a^*b)d=\langle a,b\rangle d.\]
The other properties of a right inner product $\mathcal{D}(\rho)$-module are immediate.  It only remains to show that $\mathcal{H}(\rho)$ is complete with respect to the inner product norm.

By \autoref{a*aDefined}, the inner product norm is the same as the $2$-norm, i.e.
\[\|a\|_2=\sqrt{\|\langle a,a\rangle\|}.\]
Thus it suffices to show that $\mathcal{H}(\rho)$ is closed in the Banach space $\mathcal{A}_2(\rho)$.  To see this, take $a\in\mathcal{A}_2(\rho)$ in the closure of $\mathcal{H}(\rho)$, so we have $(a_n)\subseteq\mathcal{H}(\rho)$ with $\|a-a_n\|_2\rightarrow0$.  For any $x\in\Gamma^0$, $F\subset\Gamma x$ and $n\in\mathbb{N}$,
\[\|a_{\Gamma x\setminus F}\|_2\leq\|(a-a_n)_{\Gamma x\setminus F}\|_2+\|(a_n)_{\Gamma x\setminus F}\|_2\leq\|a-a_n\|_2+\|(a_n)_{\Gamma x\setminus F}\|_2.\]
Thus $\lim_{F\subset\Gamma x}\|a_{\Gamma x\setminus F}\|_2\leq\|a-a_n\|_2\rightarrow0$, for all $x\in\Gamma^0$, so $a^*a$ is defined, by \autoref{a*aDefined}.  This means $a\in\mathcal{H}(\rho)$, showing that $\mathcal{H}(\rho)$ is closed, as required.
\end{proof}

\begin{prp}\label{abinH}
If $a\in\mathcal{A}_\mathsf{b}(\rho)$ and $b\in\mathcal{H}(\rho)$ then $ab\in\mathcal{A}_2(\rho)$ and
\begin{equation}\label{b2}
\|ab\|_2\leq\|a\|_\mathsf{b}\|b\|_2.
\end{equation}
\end{prp}

\begin{proof}
By \autoref{a*b}, $ab$ is defined, as $\|a^*\|_2\leq\|a^*\|_\mathsf{b}=\|a\|_\mathsf{b}<\infty$.  Moreover, for all $x\in\Gamma^0$ and $F\subseteq G\subset\Gamma x$,
\[\|ab_G-ab_F\|_2=\|ab_{G\setminus F}\|_2\leq\|a\|_\mathsf{b}\|b_{G\setminus F}\|_2\leq\|a\|_\mathsf{b}\|b_{\Gamma x\setminus F}\|_2.\]
This shows $(ab_F)_{F\subset\Gamma x}$ is $2$-Cauchy, as $\lim_{F\subset\Gamma x}\|b_{\Gamma x\setminus F}\|_2=0$.  By \autoref{kBanach}, $(ab_F)_{F\subset\Gamma x}$ has a limit in $\mathcal{A}_2(\rho)$, which must also be the pointwise limit $ab_{\Gamma x}$, i.e.
\begin{equation}\label{2limitF}
\lim_{F\subset\Gamma x}\|ab_{\Gamma x}-ab_F\|_2=0.
\end{equation}
Also $\|ab_{\Gamma x}\|_2=\lim_{F\subset\Gamma x}\|ab_F\|_2\leq\sup_{F\subset\Gamma x}\|a\|_\mathsf{b}\|b_F\|_2\leq\|a\|_\mathsf{b}\|b\|_2$ so
\[\|ab\|_2=\sup_{x\in\Gamma^0}\|ab_{\Gamma x}\|_2\leq\|a\|_\mathsf{b}\|b\|_2.\]
In particular, $\|ab\|_2<\infty$ so $ab\in\mathcal{A}_2(\rho)$.
\end{proof}

\begin{prp}\label{AssociativeProduct}
If $a^*,c\in\mathcal{H}(\rho)$ and $b\in\mathcal{A}_\mathsf{b}(\rho)$ then $(ab)c=a(bc)$.
\end{prp}

\begin{proof}
Note $bc,b^*a^*\in\mathcal{A}_2(\rho)$, by \autoref{abinH}, so $(ab)c$ and $a(bc)$ are defined, by \autoref{a*b}.  Moreover, for any $x\in\Gamma^0$ and $F\subset\Gamma x$, \eqref{infty22} and \eqref{b2} yield
\begin{align*}
\|(ab)c_{\Gamma x}-(ab)c_F\|_\infty&=\|(ab)c_{\Gamma x\setminus F}\|_\infty\leq\|b^*a^*\|_2\|c_{\Gamma x\setminus F}\|_2\leq\|b^*\|_\mathsf{b}\|a^*\|_2\|c_{\Gamma x\setminus F}\|_2\\
\|a(bc_{\Gamma x})-a(bc_F)\|_\infty&=\|a(bc_{\Gamma x\setminus F})\|_\infty\leq\|a^*\|_2\|bc_{\Gamma x\setminus F}\|_2\leq\|a^*\|_2\|b\|_\mathsf{b}\|c_{\Gamma x\setminus F}\|_2.
\end{align*}
By \eqref{abcF}, $(ab)c_F=a(bc_F)$ so $\|(ab)c_{\Gamma x}-a(bc_{\Gamma x})\|_\infty\leq2\|a^*\|_2\|b\|_\mathsf{b}\|c_{\Gamma x\setminus F}\|_2$.  Then \eqref{a*a} yields $((ab)c)_{\Gamma x}=(ab)c_{\Gamma x}=a(bc_{\Gamma x})=(a(bc))_{\Gamma x}$.  As $x$ was arbitrary, it follows that $(ab)c=a(bc)$ everywhere on $\Gamma$.
\end{proof}

Let
\begin{align*}
\mathcal{B}(\rho)&=\mathcal{A}_\mathsf{b}(\rho)\cap\mathcal{H}(\rho)\cap\mathcal{H}(\rho)^*\\
&=\{a\in\mathcal{A}_\mathsf{b}(\rho):a^*a\text{ and }aa^*\text{ are defined}\}.
\end{align*}

\begin{thm}
$\mathcal{B}(\rho)$ is a C*-algebra w.r.t. $\|\cdot\|_\mathsf{b}$.
\end{thm}

\begin{proof}
By \eqref{||ab||_b} and \autoref{abinH}, $\mathcal{B}(\rho)$ is closed under products.  By \autoref{AssociativeProduct}, products are also associative on $\mathcal{B}(\rho)$.  We also immediately see that $(ab)^*=b^*a^*$ and $a(b+c)=ab+ac$, for all $a,b,c\in\mathcal{B}(\rho)$.  The $\mathsf{b}$-norm is submultiplicative, by \eqref{||ab||_b}, and satisfies the C*-norm condition, by \autoref{a*aDefined}.  As $\mathcal{A}_\mathsf{b}(\rho)$ and $\mathcal{H}(\rho)$ are Banach spaces, and the $\mathsf{b}$-norm dominates the $2$-norm, it follows that $\mathcal{B}(\rho)$ is also a Banach space and hence a C*-algebra.
\end{proof}

\begin{prp}\label{LHsubH}
If $a\in\mathcal{B}(\rho)$ and $b\in\mathcal{H}(\rho)$ then $ab\in\mathcal{H}(\rho)$.
\end{prp}

\begin{proof}
For all $x\in\Gamma^0$, \eqref{2limitF} yields $\lim_{F\subset\Gamma x}\|ab_{\Gamma x}-ab_F\|=0$.  As $a\in\mathcal{B}(\rho)\subseteq\mathcal{H}(\rho)$, $ab_F\in\mathcal{H}(\rho)$, for all $F\subset\Gamma x$, and hence $ab_{\Gamma x}=(ab)_{\Gamma x}\in\mathcal{H}(\rho)$, by \autoref{Hilbert}.  As this holds for all $x\in\Gamma^0$, it follows that $(ab)^*(ab)$ is defined, by \eqref{a*a}.
\end{proof}

\begin{prp}\label{AssociativeProduct2}
If $a\in\mathcal{A}_\mathsf{b}(\rho)$, $b\in\mathcal{B}(\rho)$ and $c\in\mathcal{H}(\rho)$ then $(ab)c=a(bc)$.
\end{prp}

\begin{proof}
This is proved exactly as in \autoref{AssociativeProduct}, once we note that $(ab)^*=b^*a^*$ is a well-defined element of $\mathcal{A}_\mathsf{b}(\rho)\subseteq\mathcal{A}_2(\rho)$, by \eqref{||ab||_b}.
\end{proof}

Say $M$ is a bimodule over an algebra $A$.  We say $M$ is a \emph{Banach bimodule} over $A$ if both $M$ and $A$ are also Banach spaces with $\|ab\|\leq\|a\|\|b\|$, whenever $a$ and/or $b$ are in $A$.  We say $M$ is a \emph{Banach *-bimodule} over $A$ if we also have involutions on both $M$ and $A$ such that $\|a^*\|=\|a\|$ and $(ab)^*=b^*a^*$, whenever $a$ and/or $b$ are in $A$.  In fact, we are interested in certain Banach *-bimodules over subspaces, i.e. where $A\subseteq M$ with the involution and Banach space structure on $A$ induced by $M$.

\begin{cor}
$\mathcal{A}_\mathsf{b}(\rho)$ is a Banach *-bimodule over $\mathcal{B}(\rho)$.
\end{cor}

\begin{proof}
By \autoref{kBanach}, $\mathcal{A}_\mathsf{b}(\rho)$ is a Banach space.  By \eqref{||ab||_b} and \autoref{a*b}, whenever $a\in\mathcal{A}_\mathsf{b}(\rho)$ and $b\in\mathcal{B}(\rho)$ or vice versa, $ab$ is a well-defined element of $\mathcal{A}_\mathsf{b}(\rho)$ with $(ab)^*=b^*a^*$ and $\|ab\|_\mathsf{b}\leq\|a\|_\mathsf{b}\|b\|_\mathsf{b}$.  Also $a(bc)=(ab)c$ when at least two of these elements are in $\mathcal{B}(\rho)$ and the other is in $\mathcal{A}_\mathsf{b}(\rho)$, thanks to \autoref{AssociativeProduct} and \autoref{AssociativeProduct2}.  Distributivity of products over sums is immediate.
\end{proof}

For any Hilbert module $H$, we denote the adjointable operators on $H$ by $\mathsf{B}(H)$.  By \cite[Proposition 2.21]{RaeburnWilliams1998}, $\mathsf{B}(H)$ is a C*-algebra w.r.t. the operator norm.

\begin{thm}
Each $a\in\mathcal{A}_\mathsf{b}(\rho)$ defines a bounded operator $a_\mathsf{B}:\mathcal{H}(\rho)\rightarrow\mathcal{A}_2(\rho)$ by
\[a_\mathsf{B}(b)=ab.\]
Moreover, $a\mapsto a_\mathsf{B}$ is an isomorphism from $\mathcal{B}(\rho)$ onto a C*-subalgebra of $\mathsf{B}(\mathcal{H}(\rho))$.
\end{thm}

\begin{proof}
If $a\in\mathcal{A}_\mathsf{b}(\rho)$ then, by \autoref{abinH}, $a_\mathsf{B}$ is indeed a well-defined $2$-bounded operator from $\mathcal{H}(\rho)$ to $\mathcal{A}_2(\rho)$ with $\|a_\mathsf{B}\|\leq\|a\|_\mathsf{b}$.  If $a\in\mathcal{B}(\rho)$ then $\mathrm{ran}(a_\mathsf{B})\subseteq\mathcal{H}(\rho)$, by \autoref{LHsubH}.  For any $b,c\in\mathcal{H}(\rho)$, \autoref{AssociativeProduct} then yields
\[\langle ab,c\rangle=(ab)^*c=(b^*a^*)c=b^*(a^*c)=\langle b,a^*c\rangle.\]
So if $a\in\mathcal{B}(\rho)$ then $a_\mathsf{B}$ is adjointable, specifically $(a_\mathsf{B})^*=(a^*)_\mathsf{B}$, i.e. $a_\mathsf{B}\in\mathsf{B}(\mathcal{H}(\rho))$.

Next note that $\|a\|_\mathsf{b}=\|a\|_\mathsf{B}\leq\|a_\mathsf{B}\|$, as $\mathcal{F}(\rho)\subseteq\mathcal{H}(\rho)$, so in fact $\|a_\mathsf{B}\|=\|a\|_\mathsf{b}$, i.e. $a\mapsto a_\mathsf{B}$ is an isometry.  Certainly $(a+b)_\mathsf{B}=a_\mathsf{B}+b_\mathsf{B}$, for any $a,b\in\mathcal{A}_\mathsf{b}(\rho)$, and also $(ab)_\mathsf{B}=a_\mathsf{B}b_\mathsf{B}$, as long as $a\in\mathcal{B}(\rho)$ or $b\in\mathcal{B}(\rho)$, by \autoref{AssociativeProduct} and \autoref{AssociativeProduct2} respectively.  In particular, $a\mapsto a_\mathsf{B}$ is an isomorphism from $\mathcal{B}(\rho)$ onto a C*-subalgebra of $\mathsf{B}(\mathcal{H}(\rho))$.
\end{proof}

By using the above representation, we can show that the $\mathsf{b}$-norm of any $c\in\mathcal{B}(\rho)$ can be calculated just from $f\in\mathcal{F}(\rho)$ taking values generated by the range of $c$.

\begin{lem}
If $C\subseteq B$ be a *-subsemigroupoid of $B$ such that $C_\gamma=C\cap B_\gamma$ is a subspace of $B_\gamma=\rho^{-1}\{\gamma\}$, for all $\gamma\in\Gamma$, then, for any $c\in\mathcal{B}(\rho)$ with $\mathrm{ran}(c)\subseteq C$,
\begin{equation}\label{bC}
\|c\|_\mathsf{b}=\sup\{\|cf\|_2:f\in\mathcal{F}(\rho),\|f\|_2=1\text{ and }\mathrm{ran}(f)\subseteq C\}.
\end{equation}
\end{lem}

\begin{proof}
Let $A=C^*(c)$ be the C*-subalgebra of $\mathcal{B}(\rho)$ generated by $c$.  Further let
\[H=\mathrm{cl}_2\{h\in\mathcal{H}(\rho):\mathrm{ran}(h)\subseteq C\}.\]
Note that $H$ is a Hilbert $D$-module, where $D=\mathcal{D}(\rho)\cap H$, which is also invariant under multiplication by $c$ and $c^*$.  Thus $a\mapsto a_\mathsf{B}|_H$ is a representation of $A$ as a C*-subalgebra of $\mathsf{B}(H)$.  To see that this representation is faithful, take any $a\in A\setminus\{0\}$ so we have $\alpha\in\Gamma$ with $a(\alpha)\neq0$.  Defining $h\in H$ by $h(\alpha^{-1})=a(\alpha)^*$ and $h(\gamma)=0$, for all $\gamma\neq\alpha^{-1}$, we see that $ah(\mathsf{r}(\alpha))=a(\alpha)h(\alpha^{-1})=a(\alpha)a(\alpha)^*\neq0$ so, in particular, $a_\mathsf{B}|_H\neq0$, as required.  As faithful representations are isometric,
\[\|c\|_\mathsf{b}=\|c_\mathsf{B}|_H\|=\sup\{\|cf\|_2:f\in\mathcal{F}(\rho),\|f\|_2=1\text{ and }\mathrm{ran}(f)\subseteq C\}.\qedhere\]
\end{proof}

\section{Continuous Sections}

First we reiterate our standing assumption.
\begin{center}
\textbf{Throughout the rest of this section $\rho:B\twoheadrightarrow\Gamma$ is a Fell bundle.}
\end{center}
Now we use the topology of $\rho$ and consider the compactly supported continuous sections $\mathcal{C}_\mathsf{c}(\rho)$.  Among these, we denote those supported on a compact slice of $\Gamma$ by
\[\mathcal{S}_\mathsf{c}(\rho)=\{a\in\mathcal{C}_\mathsf{c}(\rho):\mathrm{cl}(\mathrm{supp}(a))\text{ is a slice}\}.\]
We further denote the `reduction' of $\mathcal{C}(\rho)$ to their $\mathsf{b}$-closures by
\[\mathcal{C}_\mathsf{r}(\rho)=\mathrm{cl}_\mathsf{b}(\mathcal{C}_\mathsf{c}(\rho))\qquad\text{and}\qquad\mathcal{S}_\mathsf{r}(\rho)=\mathrm{cl}_\mathsf{b}(\mathcal{S}_\mathsf{c}(\rho)).\]

\begin{prp}\label{SrChar}
$\mathcal{S}_\mathsf{r}(\rho)$ consists of the slice-supported elements of $\mathcal{C}_\mathsf{r}(\rho)$, i.e.
\[\mathcal{S}_\mathsf{r}(\rho)=\{a\in\mathcal{C}_\mathsf{r}(\rho):\mathrm{supp}(a)\text{ is a slice}\}.\]
\end{prp}

\begin{proof}
Take $(a_n)\subseteq\mathcal{A}(\rho)$ with $\infty$-limit $a$.  If $\mathrm{supp}(a)$ is not a slice, then we have $\beta,\gamma\in\mathrm{supp}(a)$ with a common source or range unit.  Then we have $n\in\mathbb{N}$ with $\|a-a_n\|_\infty<\|a(\beta)\|\wedge\|a(\gamma)\|$ and hence $\beta,\gamma\in\mathrm{supp}(a_n)$, showing that $\mathrm{supp}(a_n)$ is not a slice.  So $\infty$-limits and hence $\mathsf{b}$-limits of slice-supported sections are again slice-supported.  In particular, $\mathrm{supp}(a)$ is a slice, for all $a\in\mathcal{S}_\mathsf{r}(\rho)$.

Conversely, take $a\in\mathcal{C}_\mathsf{r}(\rho)\subseteq\mathcal{C}_0(\rho)$ such that $\mathrm{supp}(a)$ is a slice.  For each $n\in\mathbb{N}$, it follows that $K_n=a^{-1}\{b\in B:\|b\|\geq1/n\}$ is a compact slice and hence $K_n\subseteq O_n$, for some open slice $O_n$, by \cite[Proposition 6.3]{BiceStarling2018}.  For each $n\in\mathbb{N}$, we have continuous $f_n:\Gamma\rightarrow[0,1]$ with $f_n[K_n]=\{1\}$ and $\mathrm{supp}(f_n)\Subset O_n$ and then we can define $a_n\in\mathcal{S}_\mathsf{c}(\rho)$ by $a_n(\gamma)=f_n(\gamma)a(\gamma)$.  As $\mathrm{supp}(a-a_n)\subseteq\mathrm{supp}(a)$ is a slice, \autoref{infty2bprp} yields $\|a-a_n\|_\mathsf{b}=\|a-a_n\|_\mathsf{\infty}\rightarrow0$, showing that $a\in\mathcal{S}_\mathsf{r}(\rho)$.
\end{proof}

For any semigroup $(S,+)$, denote the subsemigroup generated by any $T\subseteq S$ by
\[T_\Sigma=\{{\textstyle\sum}_{k=1}^nt_k:t_1,\ldots,t_n\in T\}.\]

\begin{prp}
Every compactly supported continuous section of $\rho$ is a finite sum of compact-slice-supported ones and is thus a member of $\mathcal{B}(\rho)$, i.e.
\begin{equation}\label{CcSc}
\mathcal{C}_\mathsf{c}(\rho)=\mathcal{S}_\mathsf{c}(\rho)_\Sigma\subseteq\mathcal{B}(\rho).
\end{equation}
\end{prp}

\begin{proof}
Let $(\lambda_i)_{i\in I}$ be a continuous partition of unity on $\Gamma$ such that $\mathrm{cl}(\mathrm{supp}(\lambda_i))$ is a compact slice, for each $i\in I$.  Given any $a\in\mathcal{C}_\mathsf{c}(\rho)$, we have $F\subset I$ such that $\sum_{i\in F}\lambda_i(\gamma)=1$, for all $\gamma\in\mathrm{cl}(\mathrm{supp}(a))$.  For each $i\in F$, we can define $a_i\in\mathcal{S}_\mathsf{c}(\rho)$ by $a_i(\gamma)=\lambda_i(\gamma)a(\gamma)$ and then $a=\sum_{i\in F}a_i$, proving the first equality.  Next recall that $\|s\|_\mathsf{b}=\|s\|_\infty<\infty$, for any $s\in\mathcal{S}_\mathsf{c}(\rho)$, by \eqref{infty2b}.  It follows that $\mathcal{S}_\mathsf{c}(\rho)\subseteq\mathcal{B}(\rho)$ and hence $\mathcal{C}_\mathsf{c}(\rho)\subseteq\mathcal{B}(\rho)$.
\end{proof}

We denote the $\mathsf{b}$-bounded elements of $\mathcal{C}(\rho)$ by
\[\mathcal{C}_\mathsf{b}(\rho)=\mathcal{C}(\rho)\cap\mathcal{A}_\mathsf{b}(\rho).\]

\begin{thm}
$\mathcal{C}_\mathsf{b}(\rho)$ is a Banach *-bimodule over the C*-algebra $\mathcal{C}_\mathsf{r}(\rho)\subseteq\mathcal{C}_\mathsf{b}(\rho)$.
\end{thm}

\begin{proof}
First we argue as in \cite[Proposition 2.5]{Bice2020Rings} to show that any product of $a\in\mathcal{C}(\rho)$ and $s\in\mathcal{S}_\mathsf{c}(\rho)$ is again continuous.  To see this, take an open slice $O$ containing the compact slice $K=\mathrm{cl}(\mathrm{supp}(a))$, which is possible by \cite[Proposition 6.3]{BiceStarling2018}.  Note that $\gamma\mapsto\gamma_O=\mathsf{r}|_O^{-1}(\mathsf{r}(\gamma))$ is a continuous map from $O\Gamma=\mathsf{r}^{-1}[\mathsf{r}[O]]$ to $O$.  Thus $sa$ is also continuous on $O\Gamma$, as $sa(\gamma)=s(\gamma_O)a(\gamma_O^{-1}\gamma)$, for all $\gamma\in O\Gamma$.  On the other hand, note that $K\Gamma=\mathsf{r}^{-1}[\mathsf{r}[K]]$ is closed subset of $O\Gamma$.  As $sa$ takes zero values on the open set $\Gamma\setminus K\Gamma$, it is also continuous there, i.e. $sa\in\mathcal{C}(\rho)$.  Dually, $as\in\mathcal{C}(\rho)$, which proves the claim.  By \eqref{CcSc}, this extends to $\mathcal{C}_\mathsf{c}(\rho)$, i.e.
\begin{equation}\label{CcC}
\mathcal{C}_\mathsf{c}(\rho)\mathcal{C}(\rho)\cup\mathcal{C}(\rho)\mathcal{C}_\mathsf{c}(\rho)\subseteq\mathcal{C}(\rho).
\end{equation}

As products of compact subsets of $\Gamma$ are again compact, $\mathcal{C}_\mathsf{c}(\rho)\mathcal{C}_\mathsf{c}(\rho)\subseteq\mathcal{C}_\mathsf{c}(\rho)$.  Thus $\mathcal{C}_\mathsf{c}(\rho)$ is a *-subalgebra of $\mathcal{B}(\rho)$ and hence its $\mathsf{b}$-closure $\mathcal{C}_\mathsf{r}(\rho)$ is a C*-subalgebra of $\mathcal{B}(\rho)$.  For any $a\in\mathcal{C}_\mathsf{b}(\rho)$, it also follows from \eqref{CcC} that $a_\mathsf{B}$ restricts to a $\mathsf{b}$-bounded map from $\mathcal{C}_\mathsf{c}(\rho)$ to $\mathcal{C}_\mathsf{b}(\rho)$.  But $\mathcal{C}_\mathsf{b}(\rho)$ is a closed subspace of $\mathcal{A}_\mathsf{b}(\rho)$, as any $\mathsf{b}$-limit is an $\infty$-limit and $\infty$-limits of continuous functions are again continuous.  It follows that $\mathcal{C}_\mathsf{r}(\rho)\subseteq\mathcal{C}_\mathsf{b}(\rho)$ and the restriction of $a_\mathsf{B}$ to $\mathcal{C}_\mathsf{r}(\rho)$ still has range in $\mathcal{C}_\mathsf{b}(\rho)$.  In other words, for any $b\in\mathcal{C}_\mathsf{r}(\rho)$, $ab\in\mathcal{C}_\mathsf{b}(\rho)$ and, likewise, $ba\in\mathcal{C}_\mathsf{b}(\rho)$.  Thus $\mathcal{C}_\mathsf{b}(\rho)$ is a Banach *-bimodule over $\mathcal{C}_\mathsf{r}(\rho)$.
\end{proof}

\begin{rmk}
As the $\mathsf{b}$-norm dominates the uniform norm, $\mathcal{C}_\mathsf{r}(\rho)\subseteq\mathcal{C}_0(\rho)$, although this inclusion can certainly be proper.  For example, if $\rho$ is the trivial complex line bundle over $\mathbb{N}\times\mathbb{N}$ as in \autoref{NN} and $a$ is the section defined by $a(m,1)=1/\sqrt{m}$ and $a(m,n)=0$, for all $n\neq1$, then $a\in\mathcal{C}_0(\rho)\setminus\mathcal{A}_2(\rho)\subseteq\mathcal{C}_0(\rho)\setminus\mathcal{C}_\mathsf{r}(\rho)$.
\end{rmk}

We can also consider the \emph{left locally reduced algebra} of $\rho$ given by
\[\mathcal{L}_\mathsf{r}(\rho)=\{a\in\mathcal{A}_\mathsf{b}(\rho):\forall K\Subset\Gamma^0\ \exists b\in\mathcal{C}_\mathsf{r}(\rho)\ (a|_{\Gamma K}=b|_{\Gamma K})\}.\]

\begin{prp}\label{Lr}
$\mathcal{L}_\mathsf{r}(\rho)$ is a $\mathsf{b}$-closed algebra containing $\mathcal{C}_\mathsf{r}(\rho)$ as a left ideal, i.e.
\begin{equation}\label{Lr}
\mathcal{L}_\mathsf{r}(\rho)\mathcal{L}_\mathsf{r}(\rho)\subseteq\mathcal{L}_\mathsf{r}(\rho)=\mathrm{cl}_\mathsf{b}(\mathcal{L}_\mathsf{r}(\rho))\quad\text{and}\quad\mathcal{L}_\mathsf{r}(\rho)\mathcal{C}_\mathsf{r}(\rho)\subseteq\mathcal{C}_\mathsf{r}(\rho).
\end{equation}
Moreover, $\mathcal{C}_\mathsf{b}(\rho)$ is a right $\mathcal{L}_\mathsf{r}(\rho)$-module, i.e. $\mathcal{C}_\mathsf{b}(\rho)\mathcal{L}_\mathsf{r}(\rho)\subseteq\mathcal{C}_\mathsf{b}(\rho)$.
\end{prp}

\begin{proof}
As elements of $\mathcal{L}_\mathsf{r}(\rho)$ are locally like $\mathcal{C}_\mathsf{r}(\rho)\subseteq\mathcal{C}_\mathsf{b}(\rho)\cap\mathcal{H}(\rho)$,
\[\mathcal{L}_\mathsf{r}(\rho)\subseteq\mathcal{C}_\mathsf{b}(\rho)\cap\mathcal{H}(\rho).\]
So $ab$ is defined whenever $a\in\mathcal{C}_\mathsf{b}(\rho)$ and $b\in\mathcal{L}_\mathsf{r}(\rho)$, by \autoref{a*b}.  Moreover, again because elements of $\mathcal{L}_\mathsf{r}(\rho)$ are locally like $\mathcal{C}_\mathsf{r}(\rho)$ and $\mathcal{C}_\mathsf{b}(\rho)\mathcal{C}_\mathsf{r}(\rho)\subseteq\mathcal{C}_\mathsf{b}(\rho)$,
\[\mathcal{C}_\mathsf{b}(\rho)\mathcal{L}_\mathsf{r}(\rho)\subseteq\mathcal{C}_\mathsf{b}(\rho).\]

We next claim that $\mathcal{L}_\mathsf{r}(\rho)\mathcal{C}_\mathsf{c}(\rho)\subseteq\mathcal{C}_\mathsf{r}(\rho)$.  To see this, take $a\in\mathcal{L}_\mathsf{r}(\rho)$ and $c\in\mathcal{C}_\mathsf{c}(\rho)$.  Then $K=\mathsf{r}[\mathrm{cl}(\mathrm{supp}(c))]$ is compact and hence, for some $b\in\mathcal{C}_\mathsf{r}(\rho)$,
\[ac=a|_{\Gamma K}c=b|_{\Gamma K}c=bc\in\mathcal{C}_\mathsf{r}(\rho)\mathcal{C}_\mathsf{c}(\rho)\subseteq\mathcal{C}_\mathsf{r}(\rho).\]
This proves the claim and then \eqref{||ab||_b} yields
\[\mathcal{L}_\mathsf{r}(\rho)\mathcal{C}_\mathsf{r}(\rho)=\mathcal{L}_\mathsf{r}(\rho)\mathrm{cl}_\mathsf{b}(\mathcal{C}_\mathsf{c}(\rho))\subseteq\mathrm{cl}_\mathsf{b}(\mathcal{L}_\mathsf{r}(\rho)\mathcal{C}_\mathsf{c}(\rho))=\mathrm{cl}_\mathsf{b}(\mathcal{C}_\mathsf{r}(\rho))=\mathcal{C}_\mathsf{r}(\rho).\]

For any $a,b\in\mathcal{L}_\mathsf{r}(\rho)$ and compact $K\subseteq\Gamma^0$, we have $c\in\mathcal{C}_\mathsf{r}(\rho)$ with $b|_{\Gamma K}=c|_{\Gamma K}$ and hence $ab|_{\Gamma K}=ac|_{\Gamma K}$.  As $ac\in\mathcal{L}_\mathsf{r}(\rho)\mathcal{C}_\mathsf{r}(\rho)\subseteq\mathcal{C}_\mathsf{r}(\rho)$, this means $ab\in\mathcal{L}_\mathsf{r}(\rho)$, i.e.
\[\mathcal{L}_\mathsf{r}(\rho)\mathcal{L}_\mathsf{r}(\rho)\subseteq\mathcal{L}_\mathsf{r}(\rho).\]

To see that $\mathcal{L}_\mathsf{r}(\rho)$ is $\mathsf{b}$-closed, take $(a_n)\subseteq\mathcal{L}_\mathsf{r}(\rho)$ with $a_n\rightarrow a$ in $\mathcal{C}_\mathsf{b}(\rho)$.  For any compact $K\subseteq\Gamma^0$, we have continuous $f:\Gamma^0\rightarrow[0,1]$ such that $J=\mathrm{cl}(\mathrm{supp}(f))$ is compact and $f[K]=\{1\}$.  For each $n\in\mathbb{N}$, we have $b_n\in\mathcal{C}_\mathsf{r}(\rho)$ with $a_n|_{\Gamma J}=b_n|_{\Gamma J}$ and hence $a_nf=b_nf\in\mathcal{C}_\mathsf{r}(\rho)$, where $bf(\gamma)=f(\mathsf{s}(\gamma))b(\gamma)$.  Thus $a_nf\rightarrow af\in\mathcal{C}_\mathsf{r}(\rho)$ and $a|_{\Gamma K}=af|_{\Gamma K}$, showing that $a\in\mathcal{L}_\mathsf{r}(\rho)$.
\end{proof}

Recall that the left-multiplier algebra of a C*-algebra $A$ is its completion with respect to the left-strict uniformity generated by seminorms $\|\cdot\|_a$, for $a\in A$, where
\[\|b\|_a=\|ba\|.\]
The right-multiplier algebra of $A$ is likewise defined to be its right-strict completion.  The multiplier algebra of $A$ is its (left and right) strict completion.

We wish to show that $\mathcal{L}_\mathsf{r}(\rho)$ is the completion of $\mathcal{C}_\mathsf{r}(\rho)$ even when we restrict to seminorms coming from diagonal or diagonally central C*-subalgebras
\begin{align*}
\mathcal{D}_\mathsf{r}(\rho)&=\mathcal{D}(\rho)\cap\mathcal{C}_\mathsf{r}(\rho).\\
\mathcal{Z}_\mathsf{r}(\rho)&=\{z\in\mathcal{D}_\mathsf{r}(\rho):z[\Gamma^0]\subseteq\mathbb{C}B^0\}.
\end{align*}
When $\rho$ is categorical, \eqref{CB0} implies $\mathcal{Z}_\mathsf{r}(\rho)\approx\mathcal{C}_0(\Gamma^0)(=\mathcal{C}_0(\tau)$, where $\tau$ is the trivial line bundle over $\Gamma^0)$.  We also denote compactly supported elements of $\mathcal{Z}_\mathsf{r}(\rho)$ by
\[\mathcal{Z}_\mathsf{c}(\rho)=\{z\in\mathcal{Z}_\mathsf{r}(\rho):\mathrm{cl}(\mathrm{supp}(\rho))\text{ is compact}\}.\]

\begin{thm}\label{LrLM}
If $\rho$ is categorical, $\mathcal{L}_\mathsf{r}(\rho)$ is the left multiplier algebra of $\mathcal{C}_\mathsf{r}(\rho)$.
\end{thm}

\begin{proof}
As $\rho$ is categorical, for every $K\Subset\Gamma^0$, we have $z^K\in\mathcal{Z}_\mathsf{c}(\rho)^1_+$ such that $z^K[K]\subseteq B^0$.  For any $b\in\mathcal{L}_\mathsf{r}(\rho)$, note $bz^K\in\mathcal{L}_\mathsf{r}(\rho)\mathcal{C}_\mathsf{r}(\rho)\subseteq\mathcal{C}_\mathsf{r}(\rho)$.  Also $(z^K)_{K\Subset\Gamma^0}$ is an approximate unit for $\mathcal{C}_\mathsf{r}(\rho)$ so, for any $a\in\mathcal{C}_\mathsf{r}(\rho)$, $\|z^Ka-a\|_\mathsf{b}\rightarrow0$ and hence $\|bz^Ka-ba\|_\mathsf{b}\rightarrow0$, showing that $(bz^K)$ approaches $b$ in the left-strict topology arising from $\mathcal{C}_\mathsf{r}(\rho)$.  This shows $\mathcal{L}_\mathsf{r}(\rho)$ is contained in the left-strict closure of $\mathcal{C}_\mathsf{r}(\rho)$.

Conversely, say $(b_\lambda)\subseteq\mathcal{C}_\mathsf{r}(\rho)$ is left-strict Cauchy.  In particular, for each $K\Subset\Gamma^0$, $(b_\lambda z^K)$ is $\mathsf{b}$-Cauchy.  Thus the restrictions $(b_\lambda)_{K\Gamma}=(b_\lambda z^K)_{K\Gamma}$ also form a $\mathsf{b}$-Cauchy sequence which thus has a $\mathsf{b}$-limit.  It follows that $(b_\lambda)$ has a pointwise limit $b$ such that $\|b_\lambda z^K-bz^K\|_\mathsf{b}\rightarrow0$, for all $K\Subset\Gamma^0$.

We claim that $b\in\mathcal{C}_\mathsf{b}(\rho)$.  If not then $\sup_{K\Subset\Gamma^0}\|bz^K\|=\|b\|_\mathsf{b}=\infty$ so, for each $n\in\mathbb{N}$, we have $K_n\Subset\Gamma^0$ with $\|bz_{K_n}\|>4^n$.  Letting $z=\sum2^{-n}z^{K_n}\in\mathcal{Z}_\mathsf{r}(\rho)$, note
\[\limsup_\lambda\|b_\lambda z\|_\mathsf{b}\geq\lim_\lambda2^{-n}\|b_\lambda z^{K_n}\|_\mathsf{b}=2^{-n}\|bz^{K_n}\|_\mathsf{b}\geq2^n,\]
for all $n\in\mathbb{N}$, and hence $\limsup_\lambda\|b_\lambda z\|_\mathsf{b}=\infty$.  In particular, $(b_\lambda z)$ is not $\mathsf{b}$-Cauchy so $(b_\lambda)$ is not left-strict Cauchy, a contradiction, so $b\in\mathcal{C}_\mathsf{b}(\rho)$ as claimed.  Thus $b\in\mathcal{L}_\mathsf{r}(\rho)$, as $\|b_\lambda z^K-bz^K\|_\mathsf{b}\rightarrow0$ and hence $bz^K\in\mathcal{C}_\mathsf{r}(\rho)$, for all $K\Subset\Gamma^0$.

To show $b$ is the left-strict limit of $(b_\lambda)$, it now suffices to show $\|b_\lambda c-bc\|_\mathsf{b}\rightarrow0$, for all $c\in\mathcal{C}_\mathsf{c}(\rho)$.  But if $c\in\mathcal{C}_\mathsf{c}(\rho)$ then $K=\mathrm{supp}(c)\Subset\Gamma^0$ so $z^Kc=c$ and hence $\|b_\lambda c-bc\|_\mathsf{b}=\|b_\lambda z^Kc-bz^Kc\|_\mathsf{b}\rightarrow0$, as required.  As $(b_\lambda)$ was an arbitrary left-strict Cauchy sequence, this shows that $\mathcal{L}_\mathsf{r}(\rho)$ is the left-multiplier algebra of $\mathcal{C}_\mathsf{r}(\rho)$.
\end{proof}

Now \eqref{Lr} and \autoref{LrLM} immediately yield the corresponding results for the `bilocally reduced C*-algebra' defined by
\[\mathcal{B}_\mathsf{r}(\rho)=\mathcal{L}_\mathsf{r}(\rho)\cap\mathcal{L}_\mathsf{r}(\rho)^*\subseteq\mathcal{B}(\rho)\cap\mathcal{C}(\rho).\]

\begin{cor}\label{BilocallyReduced}
$\mathcal{C}_\mathsf{b}(\rho)$ is a *-bimodule over $\mathcal{B}_\mathsf{r}(\rho)$, which is a C*-algebra containing $\mathcal{C}_\mathsf{r}(\rho)$ as an ideal.  If $\rho$ is categorical then $\mathcal{B}_\mathsf{r}(\rho)$ is the multiplier algebra of $\mathcal{C}_\mathsf{r}(\rho)$.
\end{cor}

For example, if $\rho:\mathbb{C}\times\Gamma\rightarrow\Gamma$ is the trivial line bundle over $\Gamma=\mathbb{N}\times\mathbb{N}$ as in \autoref{NN}, \autoref{BilocallyReduced} corresponds to the well-known fact that the multiplier algebra of the compact operators on $\ell^2$ consists of all bounded operators on $\ell^2$.

We can also consider the reduced Hilbert module given by
\[\mathcal{H}_\mathsf{r}(\rho)=\mathrm{cl}_2(\mathcal{C}_\mathsf{c}(\rho)).\]
For $a\in\mathcal{A}_\mathsf{b}(\rho)$, the restriction of $a_\mathsf{B}$ is then an operator on $\mathcal{H}_\mathsf{r}(\rho)$ of the same norm.  In particular, this makes it clear that our reduced C*-algebra $\mathcal{C}_\mathsf{r}(\rho)$ is the same as the usual one constructed from the left-regular representation (see \cite{Kumjian1998}).

\begin{prp}
$\mathcal{H}_\mathsf{r}(\rho)$ is a $\mathcal{D}_\mathsf{r}(\rho)$-submodule of $\mathcal{H}(\rho)\cap\mathcal{C}(\rho)$ which is $a_\mathsf{B}$-invariant, for all $a\in\mathcal{L}_\mathsf{r}(\rho)$.  Moreover $\|a\|_\mathsf{b}=\|a_\mathsf{B}|_{\mathcal{H}_\mathsf{r}(\rho)}\|$, for all $a\in\mathcal{A}_\mathsf{b}(\rho)$.
\end{prp}

\begin{proof}
For any $a\in\mathcal{C}_\mathsf{c}(\rho)$ and $d\in\mathcal{D}_\mathsf{r}(\rho)\subseteq\mathcal{C}(\rho)$, we know that $ad\in\mathcal{C}(\rho)$, by \eqref{CcC}.  Also $\mathrm{cl}(\mathrm{supp}(ad))\subseteq\mathrm{cl}(\mathrm{supp}(a))$ is compact and hence $ad\in\mathcal{C}_\mathsf{c}(\rho)$.  Thus $\mathcal{C}_\mathsf{c}(\rho)$ is a $\mathcal{D}_\mathsf{r}(\rho)$-submodule and hence its $2$-closure $\mathcal{H}_\mathsf{r}(\rho)$ is also a $\mathcal{D}_\mathsf{r}(\rho)$-submodule of $\mathrm{cl}_2(\mathcal{H}(\rho))=\mathcal{H}(\rho)$ and $\mathrm{cl}_2(\mathcal{C}(\rho))\subseteq\mathrm{cl}_\infty(\mathcal{C}(\rho))=\mathcal{C}(\rho)$.  Similarly, for any $a\in\mathcal{L}_\mathsf{r}(\rho)$, \eqref{Lr} implies $a_\mathsf{B}[\mathcal{C}_\mathsf{r}(\rho)]\subseteq\mathcal{C}_\mathsf{r}(\rho)$ and hence $a_\mathsf{B}[\mathcal{H}_\mathsf{r}(\rho)]\subseteq\mathcal{H}_\mathsf{r}(\rho)$, as $\mathcal{H}_\mathsf{r}(\rho)=\mathrm{cl}_2(\mathcal{C}_\mathsf{r}(\rho))$, i.e. $\mathcal{H}_\mathsf{r}(\rho)$ is $a_\mathsf{B}$-invariant.

For any $a\in\mathcal{A}_\mathsf{b}(\rho)$, certainly $\|a_\mathsf{B}|_{\mathcal{H}_\mathsf{r}(\rho)}\|\leq\|a_\mathsf{B}\|=\|a\|_\mathsf{b}$.  For the converse note that, for any $b\in B$, we have $c\in\mathcal{S}_\mathsf{c}(\rho)$ with $b\in\mathrm{ran}(c)$ (in the proof of \autoref{CtsSections} just note that we can choose $\lambda\in\mathcal{S}_\mathsf{c}(\Gamma)$ because $\Gamma$ is \'etale).  It follows that, for any $f\in\mathcal{F}(\rho)$, we have $c\in\mathcal{C}_\mathsf{c}(\rho)\subseteq\mathcal{H}_\mathsf{r}(\rho)$ with $f=c_G$, where $G=\Gamma\mathrm{supp}(f)$, and hence $\|af\|_\mathsf{b}=\|ac_G\|_\mathsf{b}\leq\|ac\|_\mathsf{b}$.  This shows that $\|a\|_\mathsf{b}\leq\|a_\mathsf{B}|_{\mathcal{H}_\mathsf{r}(\rho)}\|$.
\end{proof}

\section{Morphisms}

Morphisms of section algebras/modules of Fell bundles can come from morphisms of the total spaces or base groupoids, or indeed a combination of the two.

First we consider groupoid morphisms.

\begin{dfn}
A functor $\phi:\Gamma\rightarrow\Gamma'$ between groupoids $\Gamma$ and $\Gamma'$ is \emph{star-bijective} if, whenever $\phi(x)=\mathsf{s}(\gamma')$, there is precisely one $\gamma\in\Gamma$ with $\mathsf{s}(\gamma)=x$ and $\phi(\gamma)=\gamma'$.
\end{dfn}

More symbolically, this can be written as
\[\tag{Star-Bijective}\phi(x)=\mathsf{s}(\gamma')\qquad\Rightarrow\qquad|\mathsf{s}^{-1}\{x\}\cap\phi^{-1}\{\gamma'\}|=1.\]
See \cite[\S2]{Bice2020Rep} for several equivalent characterisations of star-bijectivity.

Recall that a continuous map $\phi:X\rightarrow Y$ between Hausdorff topological spaces $X$ and $Y$ is \emph{proper} if preimages of compact subsets of $Y$ are compact in $X$, i.e.
\[\tag{Proper}K\Subset Y\qquad\Rightarrow\qquad\phi^{-1}[K]\Subset X.\]
In other words, proper means continuous w.r.t. the de Groot duals of $X$ and $Y$.

\begin{dfn}
An \emph{\'etale morphism} is a proper continuous star-bijective functor $\phi:\Gamma\rightarrow\Gamma'$ from an \'etale groupoid $\Gamma$ to another \'etale groupoid $\Gamma'$.
\end{dfn}

Given a Fell bundle $\rho:B\twoheadrightarrow\Gamma'$ and an \'etale morphism $\phi:\Gamma\rightarrow\Gamma'$, let
\[\phi^\rho B=\{(\gamma,b)\in\Gamma\times B:\phi(\gamma)=\rho(b)\}.\]
We consider $\phi^\rho B$ as a topological subspace of $\Gamma\times B$ with operations
\begin{align*}
\lambda(\gamma,b)&=(\gamma,\lambda b).\\
(\gamma,a)+(\gamma,a)&=(\gamma,a+b).\\
(f,a)(\gamma,b)&=(f\gamma,ab).\\
(\gamma,b)^*&=(\gamma^{-1},b^*).\\
\|(\gamma,b)\|&=\|b\|.\\
\rho_\phi(\gamma,b)&=\gamma.
\end{align*}
Then $\rho_\phi:\phi^\rho B\twoheadrightarrow\Gamma$ is also a Fell bundle known as the \emph{pullback bundle} of $\rho$ induced by $\phi$.  The pullback-section map $\phi^\rho:\mathcal{A}(\rho)\rightarrow\mathcal{A}(\rho_\phi)$ is then given by
\[\phi^\rho(a)(\gamma)=(\gamma,a(\phi(\gamma))).\]
Let $\phi^\rho_\mathsf{r}$ denote the restriction of $\phi^\rho$ to $\mathcal{C}_\mathsf{r}(\rho)$.

\begin{prp}\label{phirho}
$\phi^\rho_\mathsf{r}$ is a C*-algebra homomorphism from $\mathcal{C}_\mathsf{r}(\rho)$ to $\mathcal{C}_\mathsf{r}(\rho_\phi)$.
\end{prp}

\begin{proof}
Certainly $\phi^\rho$ is a *-preserving vector space homomorphism on $\mathcal{A}(\rho)$.  To see that it also preserves products whenever they are defined, take $a,b\in\mathcal{A}(\rho)$ such that $ab$ is defined.  As $\phi$ is star-bijective, whenever $\phi(\gamma)=\alpha'\beta'$, there exist unique $\alpha,\beta\in\Gamma$ with $\phi(\alpha)=\alpha'$, $\phi(\beta)=\beta'$ and $\gamma=\alpha\beta$ (for existence, use star-surjectivity to get $\beta\in\Gamma$ with $\phi(\beta)=\beta'$ and $\mathsf{s}(\beta)=\mathsf{s}(\gamma)$, and then let $\alpha=\gamma\beta^{-1}$ \textendash\, uniqueness then follows from star-injectivity).  It follows that, for any $\gamma\in\Gamma$,
\[ab(\phi(\gamma))=\sum_{\phi(\gamma)=\alpha'\beta'}a(\alpha')b(\beta')=\sum_{\gamma=\alpha\beta}a(\phi(\alpha))b(\phi(\beta)).\]
As fibres of $\rho_\phi$ are linearly isometric to the corresponding fibres of $\rho$, infinite sums that are valid in one are also valid in the other and hence
\begin{align*}
\phi^\rho(ab)(\gamma)&=(\gamma,ab(\phi(\gamma)))\\
&=(\gamma,\sum_{\gamma=\alpha\beta}a(\phi(\alpha))b(\phi(\beta)))\\
&=\sum_{\gamma=\alpha\beta}(\alpha,a(\phi(\alpha)))(\beta,b(\phi(\beta)))\\
&=\sum_{\gamma=\alpha\beta}\phi^\rho(a)(\alpha)\phi^\rho(b)(\beta)\\
&=(\phi^\rho(a)\phi^\rho(b))(\gamma)
\end{align*}
This shows that indeed $\phi^\rho(ab)=\phi^\rho(a)\phi^\rho(b)$.  Also $\phi^\rho$ respects restrictions in that
\[\phi^\rho(a_{Y'})=\phi^\rho(a)_{\phi^{-1}[Y']},\]
for all $Y'\subseteq\Gamma'$.  In particular, as $\phi$ is star-bijective, $\Gamma^0=\phi^{-1}[\Gamma'^0]$ and hence
\[\phi^\rho(\Phi'(a))=\phi^\rho(a_{\Gamma'^0})=\phi^\rho(a)_{\Gamma^0}=\Phi(\phi^\rho(a)).\]

We immediately see that $\phi^\rho$ is $\infty$-contractive on $\mathcal{A}(\rho)$, as
\[\|\phi^\rho(a)\|_\infty=\|a_{\phi[\Gamma]}\|_\infty\leq\|a\|_\infty.\]
As $\phi$ respects all the various operations, it likewise follows that
\[\|\phi^\rho(a)\|_2=\|a_{\phi[\Gamma]}\|_2\leq\|a\|_2,\]
i.e. $\phi^\rho$ is $\infty$-contractive on $\mathcal{A}(\rho)$.  As $\phi$ is star-bijective, $\phi[\Gamma]=\Gamma\mathsf{s}[\phi[\Gamma]]$ and hence
\[\|\phi^\rho(a)\|_\mathsf{b}=\|a_{\phi[\Gamma]}\|_\mathsf{b}\leq\|a\|_\mathsf{b},\]
i.e. $\phi^\rho$ is also $\mathsf{b}$-contractive on $\mathcal{A}(\rho)$.  As $\phi$ is continuous, $\phi^\rho[\mathcal{C}(\rho)]\subseteq\mathcal{C}(\rho_\phi)$.  As $\phi$ is also proper, $\phi^\rho[\mathcal{C}_\mathsf{c}(\rho)]\subseteq\mathcal{C}_\mathsf{c}(\rho_\phi)$ and hence $\phi^\rho[\mathcal{C}_\mathsf{r}(\rho)]\subseteq\mathcal{C}_\mathsf{r}(\rho_\phi)$, by $\mathsf{b}$-contractivity.  As $\phi^\rho$ preserves all the C*-algebra operations, its restriction $\phi^\rho_\mathsf{r}$ to $\mathcal{C}_\mathsf{r}(\rho)$ is thus a C*-algebra homomorphism to $\mathcal{C}_\mathsf{r}(\rho_\phi)$.
\end{proof}

Next we consider morphisms of the total spaces.

\begin{dfn}
Let $\rho:B\rightarrow\Gamma$ and $\rho':B'\rightarrow\Gamma'$ be Fell bundles where $\Gamma$ is an open subgroupoid of $\Gamma'$.  A \emph{bundle morphism} from $\rho$ to $\rho'$ is a continuous fibre-wise linear *-homomorphism $\beta:B\rightarrow B'$, i.e. when $(a,b)\in B^2$, $\lambda\in\mathbb{C}$ and $\rho(b)=\rho(c)$,
\begin{align*}
\rho'(\beta(a))&=\rho(a).\\
\beta(a^*)&=\beta(a)^*.\\
\beta(ab)&=\beta(a)\beta(b)\\
\beta(\lambda a)&=\lambda\beta(a).\\
\beta(b+c)&=\beta(b)+\beta(c).
\end{align*}
\end{dfn}

\begin{prp}
Every bundle morphism is contractive.
\end{prp}

\begin{proof}
Say $\beta:B\rightarrow B'$ is a bundle morphism from $\rho:B\rightarrow\Gamma$ to $\rho':B'\rightarrow\Gamma'$.  In particular, for any $x\in\Gamma^0$, the restriction of $\beta$ to $B_x$ is a C*-algebra homomorphism to $B'_x$ and is thus contractive.  For any $b\in B$ with $\mathsf{s}(\rho(b))=x$, it follows that
\[\|\beta(b)\|^2=\|\beta(b)^*\beta(b)\|=\|\beta(b^*b)\|\leq\|b^*b\|=\|b\|^2.\]
Thus $\|\beta(b)\|\leq\|b\|$, showing that $\beta$ is contractive.
\end{proof}

If $\beta:B\rightarrow B'$ is a bundle morphism from $\rho:B\rightarrow\Gamma$ to $\rho':B'\rightarrow\Gamma'$, where $\Gamma$ is an open subgroupoid of $\Gamma'$, then we define $\beta_\rho^{\rho'}:\mathcal{A}(\rho)\rightarrow\mathcal{A}(\rho')$ by
\begin{equation}\label{BetaDef}
\beta_\rho^{\rho'}(a)(\gamma')=\begin{cases}\beta(a(\gamma'))&\text{if }\gamma'\in\Gamma\\0_{\gamma'}&\text{if }\gamma'\in\Gamma'\setminus\Gamma.\end{cases}
\end{equation}

\begin{thm}\label{betarho}
$\beta_\rho^{\rho'}$ restricted to $\mathcal{C}_\mathsf{r}(\rho)$ is a C*-algebra homomorphism to $\mathcal{C}_\mathsf{r}(\rho')$.
\end{thm}

\begin{proof}
We immediately see that $\beta_\rho^{\rho'}$ is a *-preserving vector space homomorphism on $\mathcal{A}(\rho)$.  Next note that, as $\beta$ is contractive, $\beta$ preserves infinite sums, i.e. if $\gamma\in\Gamma$, $(b_\lambda)_{\lambda\in\Lambda}\subseteq B_\gamma$ and $\sum_{\lambda\in\Lambda}b_\lambda$ is defined then $\beta(\sum_{\lambda\in\Lambda}b_\lambda)=\sum_{\lambda\in\Lambda}\beta(b_\lambda)$.  Thus if $a,b\in\mathcal{A}(\rho)$ and $ab$ is defined then
\begin{align*}
\beta_\rho^{\rho'}\!(ab)(\gamma)&=\beta\Big(\sum_{\zeta\in\Gamma\gamma}a(\gamma\zeta^{-1})b(\zeta)\Big)\\
&=\sum_{\zeta\in\Gamma\gamma}\beta(a(\gamma\zeta^{-1}))\beta(b(\zeta))\\
&=\sum_{\zeta\in\Gamma'\gamma}\beta_\rho^{\rho'}\!(a)(\gamma\zeta^{-1})\beta_\rho^{\rho'}\!(b)(\zeta)\\
&=(\beta_\rho^{\rho'}\!(a)\beta_\rho^{\rho'}\!(b))(\gamma).
\end{align*}
In other words, $\beta_\rho^{\rho'}$ also preserves products whenever they are defined.  Also note that $\beta_\rho^{\rho'}$ respects restrictions, i.e. $\beta_\rho^{\rho'}\!(a_Y)=\beta_\rho^{\rho'}\!(a)_Y$, for any $Y\subseteq\Gamma$.  In particular, if $\Phi$ and $\Phi'$ denote the unit space restriction maps on $\mathcal{A}(\rho)$ and $\mathcal{A}(\rho')$ then
\[\beta_\rho^{\rho'}\!(\Phi(a))=\Phi'(\beta_\rho^{\rho'}\!(a)).\]

As $\beta$ is contractive, it follows immediately that $\beta_\rho^{\rho'}$ is $\infty$-contractive.  As $\beta$ preserves the relevant operations, $\beta_\rho^{\rho'}$ is then $2$-contractive, i.e. for all $a\in\mathcal{A}(\rho)$,
\begin{align*}
\|\beta_\rho^{\rho'}\!(a)\|_2^2&=\sup_{F\subset\Gamma}\|\Phi'(\beta_\rho^{\rho'}\!(a)^*\beta_\rho^{\rho'}\!(a)_F)\|_\infty=\sup_{F\subset\Gamma}\|\beta_\rho^{\rho'}(\Phi(a^*a_F))\|_\infty\\
&\leq\sup_{F\subset\Gamma}\|\Phi(a^*a_F)\|_\infty=\|a\|_2^2.
\end{align*}

We further claim that
\begin{equation}\label{2ContractiveInverse}
\beta_\rho^{\rho'}[\mathcal{H}(\rho)^1_2]=\beta_\rho^{\rho'}[\mathcal{H}(\rho)]^1_2,
\end{equation}
where $H^1_2=\{h\in H:\|h\|_2\leq1\}$.  Indeed, $\beta_\rho^{\rho'}[\mathcal{H}(\rho)^1_2]\subseteq\beta_\rho^{\rho'}[\mathcal{H}(\rho)]^1_2$ is immediate from $2$-contractivity.  Conversely, take $h'\in\beta_\rho^{\rho'}[\mathcal{H}(\rho)]^1_2$, so we have $h\in\mathcal{H}(\rho)$ with $\beta_\rho^{\rho'}(h)=h'$.  For each $x\in\Gamma^0$, $h^*h(x)$ is a positive element of the C*-algebra $B_x$.  For each $n\in\mathbb{N}$, let $f_n$ be the function on $\mathbb{R}_+$ that maps $[0,1/n]$ linearly to $[0,1]$, has constant value $1$ on $[1/n,1]$ and takes any $r\geq1$ to $1/\sqrt{r}$.  Applying the continuous functional calculus, we then define a function $d_n\in\mathcal{D}(\rho)$ of norm at most one by
\[d_n(x)=f_n(h^*h(x)).\]
As $|f_m(r)-f_n(r)|\leq1$, for all $r\in\mathbb{R}_+$, and $f_m(r)-f_n(r)=0$, for all $r\geq1/(m\wedge n)$,
\begin{align*}
\|h(d_m-d_n)\|_2^2&=\|\Phi((d_m-d_n)h^*h(d_m-d_n))\|_\infty\\
&=\sup_{x\in\Gamma^0}\|(f_m-f_n)^2(h^*h(x))h^*h(x)\|\\
&\leq1/(m\wedge n).
\end{align*}
Thus $hd_n$ is $2$-Cauchy and has a $2$-limit $g\in\mathcal{H}(\rho)$.  As $f_n(r)^2r\leq1$, for all $r\in\mathbb{R}_+$,
\[\|hd_n\|_2^2=\sup_{x\in\Gamma^0}\|f_n(h^*h(x))^2h^*h(x)\|\leq1,\]
for all $n\in\mathbb{N}$, and hence $\|g\|_2\leq1$.  Also, for all $x\in\Gamma^0$, $\beta$ is a C*-algebra homomorphism from $B_x$ to $B'_x$ and so respects the continuous functional calculus.  As $(1-f_n)(r)^2r\leq1/n$, for all $r\leq1$, and $\|\beta(h^*h(x))\|\leq\|\beta_\rho^{\rho'}\!(h)\|_2^2=\|h'\|_2^2=1$,
\begin{align*}
\|h'-\beta_\rho^{\rho'}\!(hd_n)\|_2^2&=\|\beta_\rho^{\rho'}\!(h-hd_n)\|_2^2\\
&=\sup_{x\in\Gamma^0}\|\beta((h-hd_n)^*(h-hd_n)(x))\|\\
&=\sup_{x\in\Gamma^0}\|(1-f_n)(\beta(h^*h(x)))^2\beta(h^*h(x))\|\\
&\leq1/n\rightarrow0.
\end{align*}
Thus $h'=\beta_\rho^{\rho'}\!(g)$, as they are both the $2$-limit of $(\beta_\rho^{\rho'}\!(hd_n))$ (because $g$ is the $2$-limit of $(hd_n)$ and $\beta_\rho^{\rho'}$ is $2$-contractive).  This completes the proof of claim \eqref{2ContractiveInverse}.

Now take $a\in\mathcal{C}_\mathsf{c}(\rho)$.  As $\mathrm{supp}(\beta_\rho^{\rho'}(a))\subseteq\mathrm{supp}(a)$, it follows that $\beta_\rho^{\rho'}(a)\in\mathcal{C}_\mathsf{c}(\rho')$.  For any $\varepsilon>0$, \eqref{bC} yields $f'\in\beta_\rho^{\rho'}[\mathcal{F}(\rho)]$ such that $\|f'\|_2=1$ and
\[\|\beta_\rho^{\rho'}(a)\|_\mathsf{b}\leq\|\beta_\rho^{\rho'}(a)f'\|_2+\varepsilon\]
By what we just proved, we have $f\in\mathcal{H}(\rho)$ (in fact $f\in\mathcal{F}(\rho)$) such that $\|f\|_2=1$ and $f'=\beta_\rho^{\rho'}(f)$.  As $\beta_\rho^{\rho'}$ is $2$-contractive, it follows that
\[\|\beta_\rho^{\rho'}(a)\|_\mathsf{b}\leq\|\beta_\rho^{\rho'}(af)\|_2+\varepsilon\leq\|af\|_2+\varepsilon\leq\|a\|_\mathsf{b}+\varepsilon.\]
This shows that $\beta_\rho^{\rho'}$ is $\mathsf{b}$-contractive on $\mathcal{C}_\mathsf{c}(\rho)$.  Thus $\beta_\rho^{\rho'}$ has a unique $\mathsf{b}$-contractive extension to $\mathcal{C}_\mathsf{r}(\rho)$, which must coincide with the unique $2$-contractive extension, namely $\beta_\rho^{\rho'}$.  This shows that $\beta_\rho^{\rho'}$ is a $\mathsf{b}$-contraction from $\mathcal{C}_\mathsf{r}(\rho)$ to $\mathcal{C}_\mathsf{r}(\rho')$.  As it preserves all the C*-algebra operations, it is also a C*-algebra homomorphism.
\end{proof}

We now consider a category of \'etale-bundle morphism pairs.

\begin{dfn}
Let $\mathbf{Fell}$ denote quadruples $(\rho',\beta,\phi,\rho)$ where
\begin{enumerate}
\item $\rho:B\twoheadrightarrow\Gamma$ and $\rho':B'\twoheadrightarrow\Gamma'$ are Fell bundles.
\item $\phi$ is an \'etale morphism from an open subgroupoid of $\Gamma'$ to $\Gamma$.
\item $\beta$ is a bundle morphism from $\rho_\phi$ to $\rho'$.
\end{enumerate}
\end{dfn}

When $(\rho',\beta,\phi,\rho)\in\mathbf{Fell}$, we say the pair $(\beta,\phi)$ is a \emph{Fell morphism} from $\rho$ to $\rho'$.

\begin{prp}
$\mathbf{Fell}$ forms a category under the product
\[(\rho'',\beta',\phi',\rho')(\rho',\beta,\phi,\rho)=(\rho'',\beta'\bullet\beta,\phi\circ\phi',\rho)\]
where $\beta'\bullet\beta=\beta'\bullet_{\phi'}\beta$ is the bundle morphism from $\rho_{\phi\circ\phi'}$ to $\rho''$ defined by
\[\beta'\bullet\beta(\gamma'',b)=\beta'(\gamma'',\beta(\phi'(\gamma''),b)).\]
\end{prp}

\begin{proof}
For associativity, take $(\rho''',\beta'',\phi'',\rho''),(\rho'',\beta',\phi',\rho'),(\rho',\beta,\phi,\rho)\in\mathbf{Fell}$ and note that, for all $\gamma'''\in\mathrm{dom}(\phi\circ\phi'\circ\phi'')=\phi''^{-1}[\phi'^{-1}[\phi^{-1}[\Gamma]]]$ and $b\in B$,
\begin{align*}
(\beta''\bullet(\beta'\bullet\beta))(\gamma''',b)&=\beta''(\gamma''',(\beta'\bullet\beta)(\phi''(\gamma'''),b))\\
&=\beta''(\gamma''',\beta'(\phi''(\gamma'''),\beta(\phi'(\phi''(\gamma''')),b)))\\
&=(\beta''\bullet\beta')(\gamma''',\beta(\phi'\circ\phi''(\gamma'''),b))\\
&=((\beta''\bullet\beta')\bullet\beta)(\gamma''',b).
\end{align*}
This shows that $\beta''\bullet(\beta'\bullet\beta)=(\beta''\bullet\beta')\bullet\beta$, as required.

We immediately see that every $(\rho',\beta,\phi,\rho)\in\mathbf{Fell}$ has a source unit $(\rho,\mathsf{p}_B,\mathrm{id}_\Gamma,\rho)$ and range unit $(\rho',\mathsf{p}_{B'},\mathrm{id}_{\Gamma'},\rho')$ satisfying \eqref{CategoryDef}, where $\mathrm{id}_\Gamma$ denotes the identity on $\Gamma$ and $\mathsf{p}_B:\Gamma\times B\rightarrow B$ denotes the right projection (where $\rho:B\twoheadrightarrow\Gamma$).  Thus $\mathbf{Fell}$ forms a category under the given product.
\end{proof}

Let $\mathbf{C^*}$ denote triples $(A',\pi,A)$ where $A$ and $A'$ are C*-algebras and $\pi:A\rightarrow A'$ is a C*-algebra homomorphism, which forms a category in the usual way where
\[(A'',\pi',A')(A',\pi,A)=(A',\pi'\circ\pi,A).\]

\begin{thm}\label{Afunctor}
We have a functor $\mathsf{A}:\mathbf{Fell}\rightarrow\mathbf{C^*}$ given by
\[\mathsf{A}(\rho',\beta,\phi,\rho)=(\mathcal{C}_\mathsf{r}(\rho'),\beta_{\rho_\phi}^{\rho'}\!\circ\phi^\rho_\mathsf{r},\mathcal{C}_\mathsf{r}(\rho))\]
\end{thm}

\begin{proof}
By \autoref{phirho}, $\phi^\rho|_{\mathcal{C}_\mathsf{r}(\rho)}$ is a C*-algebra homomorphism to $\mathcal{C}_\mathsf{r}(\rho_\phi)$.  By \autoref{betarho}, $\beta_{\rho_\phi}^{\rho'}|_{\mathcal{C}_\mathsf{r}(\rho_\phi)}$ is a C*-algebra homomorphism to $\mathcal{C}_\mathsf{r}(\rho')$.   Thus their composition $\beta_{\rho_\phi}^{\rho'}\!\circ\phi^\rho$ is also C*-algebra homomorphism so $\mathsf{A}$ indeed maps $\mathbf{Fell}$ to $\mathbf{C^*}$.

Take $(\rho'',\beta',\phi',\rho'),(\rho',\beta,\phi,\rho)\in\mathbf{Fell}$, where $\rho:B\twoheadrightarrow\Gamma$.  Then $\mathsf{A}$ takes their product $(\rho'',\beta'\bullet\beta,\phi\circ\phi',\rho)$ to $(\mathcal{C}_\mathsf{r}(\rho''),(\beta'\bullet\beta)_{\rho_{\phi\circ\phi'}}^{\rho'}\!\circ(\phi\circ\phi')^\rho,\mathcal{C}_\mathsf{r}(\rho))$.  For any $a\in\mathcal{C}_\mathsf{r}(\rho)$ and $\gamma''\in\mathrm{dom}(\phi\circ\phi')=\phi'^{-1}[\phi^{-1}[\Gamma]]$, we see that
\begin{align*}
((\beta'\bullet\beta)_{\rho_{\phi\circ\phi'}}^{\rho'}\!\circ(\phi\circ\phi')^\rho)(a)(\gamma'')&=(\beta'\bullet\beta)_{\rho_{\phi\circ\phi'}}^{\rho'}\!((\phi\circ\phi')^\rho(a))(\gamma'')\\
&=(\beta'\bullet\beta)(\gamma'',a(\phi\circ\phi'(\gamma'')))\\
&=\beta'(\gamma'',\beta(\phi'(\gamma''),a(\phi(\phi'(\gamma'')))))\\
&=\beta'(\gamma'',\beta(\phi^\rho(a)(\phi'(\gamma''))))\\
&=\beta'(\gamma'',\beta_{\rho_\phi}^{\rho'}\!\circ\phi^\rho(a)(\phi'(\gamma'')))\\
&=\beta'(\phi'^{\rho'}\circ\beta_{\rho_\phi}^{\rho'}\!\circ\phi^\rho(a)(\gamma''))\\
&(\beta'^{\rho''}_{\rho'_{\phi'}}\!\circ\phi'^{\rho'}\circ\beta_{\rho_\phi}^{\rho'}\!\circ\phi^\rho(a))(\gamma'').
\end{align*}
For any $\gamma''\in\Gamma''\setminus\mathrm{dom}(\phi\circ\phi')$, we also see that
\[((\beta'\bullet\beta)_{\rho_{\phi\circ\phi'}}^{\rho'}\!\circ(\phi\circ\phi')^\rho)(a)(\gamma'')=0_{\gamma''}=(\beta'^{\rho''}_{\rho'_{\phi'}}\!\circ\phi'^{\rho'}\circ\beta_{\rho_\phi}^{\rho'}\!\circ\phi^\rho(a))(\gamma'').\]
This shows that $(\beta'\bullet\beta)_{\rho_{\phi\circ\phi'}}^{\rho'}\!\circ(\phi\circ\phi')^\rho=\beta'^{\rho''}_{\rho'_{\phi'}}\!\circ\phi'^{\rho'}\circ\beta_{\rho_\phi}^{\rho'}\!\circ\phi^\rho$ and hence
\begin{align*}
\mathsf{A}((\rho'',\beta',\phi',\rho')(\rho',\beta,\phi,\rho))&=\mathsf{A}(\rho'',\beta'\bullet\beta,\phi\circ\phi',\rho)\\
&=(\mathcal{C}_\mathsf{r}(\rho''),(\beta'\bullet\beta)_{\rho_{\phi\circ\phi'}}^{\rho'}\!\circ(\phi\circ\phi')^\rho,\mathcal{C}_\mathsf{r}(\rho))\\
&=(\mathcal{C}_\mathsf{r}(\rho''),\beta'^{\rho''}_{\rho'_{\phi'}}\!\circ\phi'^{\rho'}\circ\beta_{\rho_\phi}^{\rho'}\!\circ\phi^\rho,\mathcal{C}_\mathsf{r}(\rho))\\
&=(\mathcal{C}_\mathsf{r}(\rho''),\beta'^{\rho''}_{\rho'_{\phi'}}\!\circ\phi'^{\rho'},\mathcal{C}_\mathsf{r}(\rho'))(\mathcal{C}_\mathsf{r}(\rho'),\beta_{\rho_\phi}^{\rho'}\!\circ\phi^\rho,\mathcal{C}_\mathsf{r}(\rho))\\
&=\mathsf{A}(\rho'',\beta',\phi',\rho')\mathsf{A}(\rho',\beta,\phi,\rho).
\end{align*}
Also $\mathsf{A}$ takes units in $\mathbf{Fell}$ to units in $\mathbf{C^*}$ and hence $\mathsf{A}$ is a functor.
\end{proof}

We call $\mathsf{A}$ the \emph{abstraction} or \emph{algebraisation} functor.  While $\mathsf{A}$ is faithful, it is certainly not full, as the C*-algebra homomorphisms resulting from Fell morphisms automatically preserve other structures arising from the Fell bundle, structures that are not determined by the C*-algebra alone.  For example, they preserve the unit space restriction maps, as we saw in the proofs of \autoref{phirho} and \autoref{betarho}, as well as the semigroups of slice-supported sections.  If the bundle morphism part of the Fell morphism preserves units then the resulting C*-algebra homomorphism will also preserve the diagonally central C*-subalgebra.  This suggests that, if we want an abstract counterpart to $\mathbf{Fell}$, we should take these additional structures as an intrinsic part of the algebraic structures forming the abstract category.  This leads to the notion of a `structured C*-algebra', as we now move on to discuss.

\part{Structured C*-Algebras}\label{StructuredCAlgebras}

\section{Preliminaries}

Our goal is to exhibit a duality between certain Fell bundles and C*-algebras.  However, C*-algebras on their own are not sufficient \textendash\, we need to work with C*-algebras carrying some additional structure including an expectation.

\begin{dfn}
An \emph{expectation} $\Phi$ on a normed space $A$ is an idempotent contraction, i.e. a linear map on $A$ with $\|\Phi(a)\|\leq\|a\|$ and $\Phi(\Phi(a))=\Phi(a)$, for $a\in A$.
\end{dfn}

By a result Tomiyama (e.g. see \cite[II.6.10]{Blackadar2017}), expectations on a C*-algebra $A$ are precisely the positive idempotent operators $\Phi$ that are $\Phi[A]$-equivariant, i.e.
\[\Phi(\Phi(a)b)=\Phi(a)\Phi(b)=\Phi(a\Phi(b)),\]
for all $a,b\in A$.  In particular, $\Phi[A]$ is necessarily a C*-subalgebra of $A$.  We will be interested in structures where we also have a distinguished central C*-subalgebra of $\Phi[A]$ as well as a larger distinguished *-subsemigroup $S$.

Let us denote the centre and *-squares of any subset $T$ of a *-semigroup $S$ by
\begin{align*}
\mathsf{Z}(T)&=\{z\in T:\forall t\in T\ (tz=zt)\}.\\
T_+&=\{t^*t:t\in T\}.
\end{align*}

\begin{dfn}
We call $(A,S,Z,\Phi)$ a \emph{structured C*-algebra} when
\begin{enumerate}
\item\label{SCA} $A$ is C*-algebra on which we have an expectation $\Phi$.
\item\label{SCS} $S=\mathbb{C}S$ is a closed *-subsemigroup of $A$.
\item\label{SCZ} $Z\subseteq S$ is a C*-subalgebra of $\mathsf{Z}(\Phi[A])$.
\item\label{SCP} $\Phi[S]$ is an ideal of $\Phi[A]$ with
\[S_+\subseteq\Phi[S]\subseteq S.\]
\end{enumerate}
\end{dfn}

Before giving some examples, we make one quick observation.

\begin{prp}
If $(A,S,Z,\Phi)$ is a structured C*-algebra then $\Phi[S]$ is closed and hence a C*-algebra whose positive part consists of the *-squares of $S$, i.e.
\[\Phi[S]_+=S_+=S\cap A_+.\]
\end{prp}

\begin{proof}
Take $a\in\mathrm{cl}(\Phi[S])$, so we have $(a_n)\subseteq\Phi[S]\subseteq S$ with limit $a\in\mathrm{cl}(S)=S$.  Then $a_n=\Phi(a_n)\rightarrow\Phi(a)$ and hence $a=\Phi(a)\in\Phi[S]$.  Thus $\Phi[S]$ is a closed ideal of $\Phi[A]$ and, in particular, a C*-subalgebra of $\Phi[A]$.  Moreover, $\Phi[S]\subseteq S$ implies $\Phi[S]_+\subseteq S_+\subseteq S\cap A_+$.  On the other hand, if $a\in S\cap A_+$ then $a^2\in S_+\subseteq\Phi[S]$ and hence $a\in\Phi[S]_+$, as positive elements in C*-algebras have positive square roots.
\end{proof}

\subsection{Examples}\label{Examples}

\begin{xpl}\label{FellXpl}
The primary motivating examples of structured C*-algebras of course come from Fell bundles $\rho:B\twoheadrightarrow\Gamma$ where we can take
\[A=\mathcal{B}_\mathsf{r}(\rho)\text{ or }\mathcal{C}_\mathsf{r}(\rho),\qquad S=\mathcal{S}_\mathsf{r}(\rho),\qquad Z=\mathcal{Z}_\mathsf{r}(\rho)\qquad\text{and}\qquad\Phi(a)=a_{\Gamma^0}.\]
With some extra conditions, we can even characterise the structured C*-algebras that arise like this from corical Fell bundles.  Indeed, our primary goal is to investigate these conditions and use them to construct appropriate Fell bundles on which to represent structured C*-algebras.

Incidentally, for the most part (up to and including \autoref{TheWeylBundle}), we could work in a slightly more general situation where $A$ itself is not a C*-algebra but only a Banach *-bimodule over some smaller C*-algebra $B\subseteq A$ containing $S$.  The motivating example for such a `structured Banach *-bimodule' being $A=\mathcal{C}_\mathsf{b}(\rho)$, where $B=\mathcal{C}_\mathsf{r}(\rho)$ or $\mathcal{B}_\mathsf{r}(\rho)$, again for some Fell bundle $\rho$.  However, for convenience we will restrict our attention to structured C*-algebras.
\end{xpl}

\begin{xpl}\label{CartanPairs}
Every Cartan pair $(A,C)$ defines a structured C*-algebra $(A,S,Z,\Phi)$ in a canonical way.  Indeed, by definition there is already an expectation $\Phi$ onto the Cartan subalgebra $C$.  Cartan subalgebras are also required to be commutative so we can simply take $Z=C$.  We then take $S$ to be the normalisers given by
\[S=\{a\in A:aCa^*+a^*Ca\subseteq C\}.\]
One then immediately verifies that $S=\mathbb{C}S$ is a closed *-subsemigroup of $A$ containing $C$.  As Cartan subalgebras are also required to contain an approximate unit of the larger algebra, it follows that $S_+\subseteq C$.  Thus the defining properties of a structured C*-algebra are satisfied.

Consequently, the theory we develop will also apply to the Cartan pairs considered by Kumjian and Renault.  Indeed, the normalisers already play a fundamental role in their work, so it is not surprising that they also turn up here.  The main reason we allow more freedom with our choice of $S$ is that we want our duality to apply to Fell bundles over even non-effective groupoids, where there would be too many normalisers.  By allowing $Z$ to be a proper C*-subalgebra of the diagonal $\Phi[A]$, we are also able to deal with more general Fell bundles rather than just the line bundles/twisted groupoids originally considered by Kumjian and Renault.
\end{xpl}

\begin{xpl}
From any C*-algebra $A$, we get a structured C*-algebra by setting
\[A=S=\mathsf{M}(A)\qquad Z=\mathsf{Z}(\mathsf{M}(A))\qquad\text{and}\qquad\Phi=\mathrm{id}_{\mathsf{M}(A)},\]
where $\mathsf{M}(A)$ is the multiplier algebra of $A$ and $\mathrm{id}_{\mathsf{M}(A)}$ is the identity map on $\mathrm{id}_{\mathsf{M}(A)}$.  Then our construction will yield a C*-bundle $\mathsf{p}$, i.e. a Fell bundle over a space/trivial groupoid, namely the Stone-\v{C}ech compactification $\beta\check{A}$ of the primitive ideal space $\check{A}$ (or at least its `Hausdorffization' \textendash\, see \cite{Hofmann2011}).  Moreover, $\mathsf{M}(A)$ gets represented as its continuous sections $\mathcal{C}(\mathsf{p})$, so here our construction reduces to the Dauns-Hofmann representation.  In particular, any continuous complex valued function on the spectrum $\check{A}$ can be extended to $\beta\check{A}$ and identified with a central multiplier $z\in\mathsf{Z}(\mathsf{M}(\mathsf{p}))$, an oft-quoted corollary of the Dauns-Hofmann representation.
\end{xpl}

\subsection{Morphisms}\label{StructuredMorphisms}
Structured C*-algebras form a category in a natural way.

\begin{dfn}
A \emph{structure-preserving morphism} from one structured C*-algebra $(A,S,Z,\Phi)$ to another $(A',S',Z',\Phi')$ is a C*-homomorphism $\pi:A\rightarrow A'$ with
\[\pi[S]\subseteq S',\quad\pi[Z]\subseteq Z'\quad\text{and}\quad\Phi'(\pi(a))=\pi(\Phi(a)),\text{ for all }a\in A.\]
\end{dfn}

\begin{rmk}\label{RangeKernelPhi}
As $\Phi$ is an idempotent linear map, preserving $\Phi$ is equivalent to preserving its range and kernel, i.e. $\Phi'(\pi(a))=\pi(\Phi(a))$, for all $a\in A$, iff
\[\pi[\Phi^{-1}\{0\}]\subseteq\Phi'^{-1}\{0\}\qquad\text{and}\qquad\pi[\Phi[A]]\subseteq\Phi'[A'].\]
In fact, later we will restrict our attention to structured C*-algebras where $\Phi[A]=C^*(S_+)$, in which case $\pi[\Phi[A]]\subseteq\Phi'[A']$ already follows from $\pi[S]\subseteq S'$.
\end{rmk}

Let $\mathbf{SC^*}$ denote the triples $(\langle A'\rangle,\pi,\langle A\rangle)$ where $\langle A\rangle=(A,S,Z,\Phi)$ and $\langle A\rangle=(A,S,Z,\Phi)$ are structured C*-algebras and $\pi:A\rightarrow A'$ is a structure-preserving morphism.  We immediately see that $\mathbf{SC^*}$ forms a category under the product coming from morphism composition, i.e. where
\[(\langle A''\rangle,\pi',\langle A'\rangle)(\langle A'\rangle,\pi,\langle A\rangle)=(\langle A''\rangle,\pi'\circ\pi,\langle A\rangle).\]
Also let $\mathbf{1Fell}$ denote the subcategory of $\mathbf{Fell}$ consisting of quadruples $(\rho',\beta,\phi,\rho)$ where $\beta$ is unital and hence a functor, i.e. $\beta[(\phi^\rho B)^0]\subseteq B'^0$.

We can also re(de)fine the functor $\mathsf{A}:\mathbf{Fell}\rightarrow\mathbf{C^*}$ from \autoref{Afunctor} to obtain a functor $\mathsf{Ab}:\mathbf{1Fell}\rightarrow\mathbf{SC^*}$.  First, for any Fell bundle $\rho:B\twoheadrightarrow\Gamma$, let
\[\langle\rho\rangle_\mathsf{r}=(\mathcal{C}_\mathsf{r}(\rho),\mathcal{S}_\mathsf{r}(\rho),\mathcal{Z}_\mathsf{r}(\rho),\rho^0_\mathsf{r}),\]
where $\rho^0_\mathsf{r}$ denotes the canonical expectation on $\mathcal{C}_\mathsf{r}(\rho)$ given by $\rho^0_\mathsf{r}(a)=a_{\Gamma^0}$.

\begin{prp}\label{Ab}
We have a functor $\mathsf{Ab}:\mathbf{1Fell}\rightarrow\mathbf{SC^*}$ given by
\[\mathsf{Ab}(\rho',\beta,\phi,\rho)=(\langle\rho'\rangle_\mathsf{r},\beta_{\rho_\phi}^{\rho'}\!\circ\phi^\rho_\mathsf{r},\langle\rho\rangle_\mathsf{r})\]
\end{prp}

\begin{proof}
As noted in \autoref{FellXpl}, $\langle\rho\rangle_\mathsf{r}=(\mathcal{C}_\mathsf{r}(\rho),\mathcal{S}_\mathsf{r}(\rho),\mathcal{Z}_\mathsf{r}(\rho),\rho^0_\mathsf{r})$ is a structured C*-algebra, for any Fell bundle $\rho:B\twoheadrightarrow\Gamma$.  Indeed, conditions \eqref{SCA}-\eqref{SCZ} are immediate.  For \eqref{SCP}, note that $\rho^0_\mathsf{r}[\mathcal{S}_\mathsf{r}(\rho)]=\rho^0_\mathsf{r}[\mathcal{C}_\mathsf{r}(\rho)]=\{a\in\mathcal{C}_\mathsf{r}(\rho):\mathrm{supp}(a)\subseteq\Gamma^0\}$, so $\rho^0_\mathsf{r}[\mathcal{S}_\mathsf{r}(\rho)]$ is certainly an ideal of $\rho^0_\mathsf{r}[\mathcal{C}_\mathsf{r}(\rho)]$, and
\[\mathcal{S}_\mathsf{r}(\rho)_+=\{a\in\mathcal{C}_\mathsf{r}(\rho):a[\mathrm{supp}(a)]\subseteq B_+\}\subseteq\rho^0_\mathsf{r}[\mathcal{S}_\mathsf{r}(\rho)]\subseteq\mathcal{S}_\mathsf{r}(\rho).\]

Next note that $\beta_{\rho_\phi}^{\rho'}\!\circ\phi^\rho_\mathsf{r}$ is structure-preserving, for any $(\rho',\beta,\phi,\rho)\in\mathbf{1Fell}$.  Indeed, we already saw in the proofs of \autoref{phirho} and \autoref{betarho} that $\phi^\rho_\mathsf{r}$ and $\beta_{\rho_\phi}^{\rho'}$ preserve the expectation $\rho_\mathsf{r}^0$.  For any $a\in\mathcal{S}_\mathsf{r}(\rho)$, note that
\[\mathrm{supp}(\beta_{\rho_\phi}^{\rho'}\!\circ\phi^\rho_\mathsf{r}(a))\subseteq\mathrm{supp}(\phi^\rho_\mathsf{r}(a))=\phi^{-1}[\mathrm{supp}(a)],\]
which is a slice because $\phi$ is star-injective (see \cite[Proposition 2.4]{Bice2020Rep}).  This shows that $\beta_{\rho_\phi}^{\rho'}\!\circ\phi^\rho_\mathsf{r}[\mathcal{S}_\mathsf{r}(\rho)]\subseteq\mathcal{S}_\mathsf{r}(\rho')$.  Also $\beta$ is unital so, for any $z\in\mathcal{Z}_\mathsf{r}(\rho)$,
\[\beta_{\rho_\phi}^{\rho'}\!\circ\phi^\rho_\mathsf{r}(z)[\mathrm{supp}(\beta_{\rho_\phi}^{\rho'}\!\circ\phi^\rho_\mathsf{r}(z))]\subseteq\beta_{\rho_\phi}^{\rho'}\!\circ\phi^\rho_\mathsf{r}(z)[\Gamma'^0]\subseteq\beta[(\phi^\rho B)^0]\subseteq B'^0.\]
This shows that $\beta_{\rho_\phi}^{\rho'}\!\circ\phi^\rho_\mathsf{r}[\mathcal{Z}_\mathsf{r}(\rho)]\subseteq\mathcal{Z}_\mathsf{r}(\rho')$.  Together with \autoref{Afunctor}, this shows that $\mathsf{Ab}$ is indeed a functor from $\mathbf{1Fell}$ to $\mathbf{SC^*}$.
\end{proof}

Again we call $\mathsf{Ab}$ the \emph{abstraction} or \emph{algebraisation} functor.  We are particularly interested in the restriction of $\mathsf{Ab}$ to $\mathbf{CFell}$, the full subcategory of $\mathbf{1Fell}$ consisting of those quadruples $(\rho',\beta,\phi,\rho)$ where $\rho$ and $\rho'$ are corical.  Indeed, our goal is to construct corical Fell bundles on which to represent structured C*-algebras, at least those satisfying certain extra properties, thus obtaining an adjoint of $\mathsf{Ab}|_{\mathbf{CFell}}$.

\subsection{Properties}\label{StructuredProperties}

To avoid repeating our basic hypotheses we assume the following.
\begin{center}
\textbf{From now on $\langle A\rangle=(A,S,Z,\Phi)$ is a structured C*-algebra.}
\end{center}

Here we examine several interrelated properties of $\langle A\rangle$ which will be important for our later work.  Many proofs make use of the continuous functional calculus/Gelfand representation, something we did not have access too in our algebraic predecessors \cite{Bice2020Rep} and \cite{Bice2020Rings}.  In particular, we will often need a sequence of polynomials $(f_n)$ that is bounded on bounded subsets of $\mathbb{R}$, approaches $1$ uniformly on compact subsets of $\mathbb{R}\setminus\{0\}$ and such that $f_n(0)=0$, for all $n$ (such a sequence exists thanks to Stone-Weierstrass).  For any $a\in A$, Gelfand then yields
\[f_n(aa^*)a=af_n(a^*a)\rightarrow a.\]

To start with, let us call $D\subseteq A$ \emph{diagonal} if, for all $a,b,d\in A$,
\[\tag{Diagonal}\label{Diagonal}ad,d,db\in D\qquad\Rightarrow\qquad adb\in D.\]
Diagonality turned out to be needed in \cite{Bice2020Rep} for parts of the semigroup theory relating to filters and cosets.  Here we have the following general result.

\begin{prp}
Every C*-subalgebra of $A$ is diagonal.
\end{prp}

\begin{proof}
If $D$ is a C*-subalgebra of $A$ and $ad,d,db\in D$ then $af_n(dd^*)\in D$, for all $n\in\mathbb{N}$, and hence $adb=\lim af_n(dd^*)db\in DD\subseteq D$.
\end{proof}

\begin{cor}
$\Phi[A]$ and $\Phi[S]$ are diagonal.
\end{cor}

\begin{proof}
This is immediate from the above result.  Alternatively, for $\Phi[A]$ we can just note that $ad,d,db\in\Phi[A]$ implies $\Phi(adb)=ad\Phi(b)=a\Phi(db)=adb$.
\end{proof}

One important property of expectations is Kadison's inequality
\[\tag{Kadison}\label{Kadison}\Phi(a^*)\Phi(a)\leq\Phi(a^*a)\]
(see \cite[II.6.9.14]{Blackadar2017}).  Later in \autoref{TheWeylRepresentation}, we will need the following $n$-fold extension.

We denote the unit ball of any $B\subseteq A$ by
\[B^1=\{b\in B:\|b\|\leq1\}=B\cap A^1.\]

\begin{prp}
If $a\in A$, $u_1,\ldots,u_n\in A^1$ and $\Phi(u_ju_k^*)=0$ when $j\neq k$ then
\[\tag{$n$-Kadison}\label{nKadison}\sum_{k=1}^n\Phi(a^*u_k^*)\Phi(u_ka)\leq\Phi(a^*a).\]
\end{prp}

\begin{proof}
As $\Phi$ is positive, $\Phi(c^*c)\geq0$ for $c=a-\sum_{k=1}^nu_k^*\Phi(u_ka_k)$, i.e.
\begin{align*}
0&\leq\Phi((a^*-\sum_{k=1}^n\Phi(a^*u_k^*)u_k)(a-\sum_{k=1}^nu_k^*\Phi(u_ka)))\\
&=\Phi(a^*a)-2\sum_{k=1}^n\Phi(a^*u_k^*)\Phi(u_ka)+\sum_{j,k\leq n}\Phi(a^*u_j^*)\Phi(u_ju_k^*)\Phi(u_ka)\\
&=\Phi(a^*a)-2\sum_{k=1}^n\Phi(a^*u_k^*)\Phi(u_ka)+\sum_{k\leq n}\Phi(a^*u_k^*)\Phi(u_ku_k^*)\Phi(u_ka)\\
&\leq\Phi(a^*a)-2\sum_{k=1}^n\Phi(a^*u_k^*)\Phi(u_ka)+\sum_{k\leq n}\Phi(a^*u_k^*)\Phi(u_ka)\\
&=\Phi(a^*a)-\sum_{k=1}^n\Phi(a^*u_k^*)\Phi(u_ka).
\end{align*}
Moving the sum to the left side now yields \eqref{nKadison}.
\end{proof}

Let us call $\Phi$ \emph{normal} or \emph{shiftable} if, for all $a\in A$ and $s\in S$,
\begin{align*}
\tag{Normal}\Phi(s^*as)&=s^*\Phi(a)s.\\
\tag{Shiftable}\Phi(sa)s&=s\Phi(as).
\end{align*}
Normality is probably more natural for C*-algebraists, while shiftability has the advantage that it requires no involution (and thus makes sense for more general structures like the Steinberg rings considered in \cite{Bice2020Rings}).  For structured C*-algebras, however, these turn out to be equivalent.

\begin{prp}\label{NormalShiftable}
$\Phi$ is normal if and only if $\Phi$ is shiftable.
\end{prp}

\begin{proof}
If $\Phi$ is normal then, for any $a\in A$ and $s\in S$, $ss^*,s^*s\in S_+\subseteq\Phi[S]$ so
\[\Phi(ss^*sa)s=ss^*\Phi(sa)s=s\Phi(s^*sas)=ss^*s\Phi(as).\]
It follows that $\Phi(sa)s=\lim\Phi(f_n(ss^*)sa)s=\lim f_n(ss^*)s\Phi(as)=s\Phi(as)$, showing that $\Phi$ is shiftable.  Conversely, if $\Phi$ is shiftable then, for any $a\in A$ and $s\in S$,
\[\Phi(s^*ss^*as)=s^*s\Phi(s^*as)=s^*\Phi(ss^*a)s=s^*ss^*\Phi(a)s.\]
Thus $\Phi(s^*as)=\lim\Phi(f_n(s^*s)s^*as)=\lim f_n(s^*s)s^*\Phi(a)s=s^*\Phi(a)s$, showing that $\Phi$ is normal.
\end{proof}

If $\Phi$ is normal then we immediately see that $sZs^*\subseteq s\Phi[S]s^*\subseteq\Phi[S]$, for all $s\in S$.  In fact this remains valid when we replace $s^*$ with any $a\in A$ with $ab\in\Phi[S]$.  

\begin{prp}
If $\Phi$ is normal then, for all $a\in A$ and $s\in S$,
\begin{equation}\label{ZPhiNormal}
as\in\Phi[S]\qquad\Rightarrow\qquad aZs\subseteq\Phi[S].
\end{equation}
\end{prp}

\begin{proof}
Take $z\in Z$, $a\in A$ and $s\in S$ with $as\in\Phi[S]$.  As $Z\subseteq\mathsf{Z}(\Phi[S])$ and $S_+\subseteq\Phi[S]$, $zss^*=ss^*z$ so $azss^*s=ass^*zs\in\Phi[S]s^*\Phi[S]s\subseteq\Phi[S]$ and hence $azsf_n(s^*s)\in\Phi[S]$, for all $n\in\mathbb{N}$.  Thus $azs=\lim azsf_n(s^*s)\in\Phi[S]$.
\end{proof}

Let us call $N\subseteq S$ \emph{binormal} if, for all $a,b\in S$,
\[\tag{Binormal}ab\in N\qquad\Rightarrow\qquad aNb\subseteq N.\]
When $Z=\Phi[S]$, as with the Cartan pairs considered by Kumjian and Renault (see \autoref{CartanPairs}), \eqref{ZPhiNormal} above says that $Z=\Phi[S]$ is binormal if $\Phi$ is normal.  In general, this is no longer true and, in order to use the theory in \cite{Bice2020Rep} and \cite{Bice2020Rings}, we will need to separately assume that $Z$ or its positive unit ball is binormal.

\begin{prp}\label{Zbinormal}
$Z$ is binormal if and only if $Z^1_+$ is binormal.
\end{prp}

\begin{proof}
Say $Z$ is binormal and take $a,b\in S$ with $ab\in Z^1_+$.  For any $z\in Z^1_+$, this means that $azb\in Z$.  As $a^*a,z\in A_+$ commute, $a^*az\in A_+$ and hence
\[abazb=b^*a^*azb\in Z\cap b^*A_+b\subseteq Z_+.\]
Thus $f_n(ab)azb\in Z_+$ too, for all $n$, so to prove $azb\in Z_+$ it suffices to show that $f_n(ab)azb\rightarrow azb$.  To see this note $bb^*,z\in A_+$ commute so $zbb^*z\leq\|z\|^2bb^*\leq bb^*$ and hence
\begin{align*}
\|(1-f_n(ab))azb\|^2&=\|(1-f_n(ab))azbb^*za^*(1-f_n(ab))\|\\
&\leq\|(1-f_n(ab))abb^*a^*(1-f_n(ab))\|\\
&=\|(1-f_n(ab))ab\|^2\\
&\rightarrow0.
\end{align*}
Also $\|azb\|^2=\|azbb^*za^*\|\leq\|abb^*a^*\|=\|ab\|^2\leq1$, again because $zb^*bz\leq bb^*$, and hence $azb\in Z^1_+$.  This shows that $Z^1_+$ is binormal.

Conversely, if $Z^1_+$ is binormal then so is $Z_+$, as $ab\in Z_+$ implies $\lambda ab\in Z^1_+$, where $\lambda=(\|ab\|\vee1)^{-1}$, so $\lambda aZ^1_+b\subseteq Z^1_+$ and hence $aZ_+b=\mathbb{R}_+\lambda aZ^1_+b\subseteq\mathbb{R}_+Z^1_+=Z_+$.  Now take $a,b\in S$ with $ab\in Z$.  Then $(ab)^*ab\in Z_+$ so $(ab)^*aZ_+b\subseteq Z_+$.  It follows that $(ab)^*aZb\subseteq Z$ and hence $ab(ab)^*aZb\subseteq ZZ\subseteq Z$.  For any $z\in Z$, we claim that $f_n(ab(ab)^*)azb\rightarrow azb$.  To see this, note that
\begin{align*}
\|(1-f_n(ab(ab)^*))azb\|^2&=\|(1-f_n(ab(ab)^*))azbb^*z^*a^*(1-f_n(ab(ab)^*))\|\\
&=\|(1-f_n(ab(ab)^*))abb^*zz^*a^*(1-f_n(ab(ab)^*))\|\\
&\leq\|(1-f_n(ab(ab)^*))ab\|\|b^*zz^*a^*(1-f_n(ab(ab)^*))\|\\
&\rightarrow0.
\end{align*}
Thus $azb=\lim f_n(ab(ab)^*)azb\in Z$, showing that $Z$ is binormal.
\end{proof}

Another important property for some parts of the theory in \cite{Bice2020Rings} is bistability.  Specifically, let us call $B\subseteq S$ \emph{bistable} if, for all $a,b\in S$,
\[\tag{Bistable}ab\in B\qquad\Rightarrow\qquad\Phi(a)b,a\Phi(b)\in B.\]

\begin{prp}\label{B1Bistable}
If $B$ is bistable then so is $B^1$ $($and conversely when $B=\mathbb{R}_+B)$.
\end{prp}

\begin{proof}
If $B$ is bistable and $ab\in B^1$ then $\Phi(a)b\in B$ and, by \eqref{Kadison},
\[\|\Phi(a)b\|^2=\|b^*\Phi(a^*)\Phi(a)b\|\leq\|b^*\Phi(a^*a)b\|=\|b^*a^*ab\|=\|ab\|^2\leq1.\]
Thus $\Phi(a)b\in B^1$ and, likewise, $a\Phi(b)\in B^1$, showing that $B^1$ is bistable.

Conversely, if $B=\mathbb{R}_+B\subseteq S$ then, for any $a,b\in S$ with $ab\in B$, we have $\lambda ab\in B^1$, for some $\lambda>0$.  If $B^1$ is bistable then $\lambda\Phi(a)b,\lambda a\Phi(b)\in B^1$ and hence $\Phi(a)b,a\Phi(b)\in\lambda^{-1}B^1\subseteq B$, showing that $B$ is bistable.
\end{proof}

We call $\Phi$ bistable when $\Phi[S]$ is bistable.  Note that $\Phi(a)b\in\Phi[S]$ is equivalent to $\Phi(a)b=\Phi(a\Phi(b))=\Phi(a)\Phi(b)$.  Thus $\Phi$ is bistable precisely when, for all $a,b\in S$,
\[\tag{$\Phi$-Bistable}ab\in\Phi[S]\qquad\Rightarrow\qquad\Phi(a)b=\Phi(a)\Phi(b)=a\Phi(b).\]

\begin{prp}\label{Z+Bistable}
If $Z_+$ is bistable then so is $Z$ $($and conversely if $\Phi$ is bistable$)$.
\end{prp}

\begin{proof}
For each $n\in\mathbb{N}$, define $g_n$ on $\mathbb{R}_+$ by $g_n(x)=nx\wedge x^{-1}$.  For all $z\in Z$ and $a\in\mathrm{cl}(Az)$, it follows that $az^*zg_n(z^*z)\rightarrow a$.

Assume $Z_+$ is bistable and take $a,b\in S$ with $ab\in Z$.  Then $abb^*a^*\in Z_+$ so the bistability of $Z_+$ yields $\Phi(a)bb^*a^*\in Z_+$ and hence, for all $n\in\mathbb{N}$,
\[\Phi(a)bb^*a^*abg_n(b^*a^*ab)\in ZZZ\subseteq Z.\]
Also $\Phi(a)b=\lim\Phi(af_n(a^*a))b=\lim\Phi(a)f_n(a^*a)b\in\mathrm{cl}(Aab)$ and hence
\[\Phi(a)b=\lim\Phi(a)bb^*a^*abg_n(b^*a^*ab)\in Z.\]
This shows that $Z$ is bistable.

If $\Phi$ is bistable then, for all $a\in S$, $\Phi(a^*)\Phi(a)=\Phi(a^*)a$, as $a^*a\in\Phi[S]$.  If $Z$ is also bistable then, for any $a,b\in S$ with $ab\in Z_+$, it follows that $\Phi(a)b\in Z$ so
\[\Phi(a)b=\lim f_n(\Phi(a)bb^*\Phi(a^*))\Phi(a)b=\lim f_n(\Phi(a)bb^*\Phi(a^*))ab\in\mathrm{cl}(Z_+Z_+)\subseteq Z_+.\]
This shows that $Z_+$ is also bistable.
\end{proof}

\begin{lem}\label{a*a=aa*}
If $Z$ is bistable, $a\in S$ and $a^*a,aa^*\in Z$ then
\[\Phi(a^*)\Phi(a)=\Phi(a)\Phi(a^*)\]
\end{lem}

\begin{proof}
As $Z$ is bistable, $\Phi(a^*)\Phi(a),\Phi(a)\Phi(a^*)\in Z$.  Thus it suffices to show that
\[d\in\Phi[S]\quad\text{and}\quad d^*d,dd^*\in Z\qquad\Rightarrow\qquad d^*d=dd^*.\]
To see this note $d\in\Phi[S]$ and $d^*d,dd^*\in Z$ implies $dd^*$ and $d^*d$ commute with $d$ so $d^*dd^*d=d^*ddd^*=dd^*dd^*$.  Likewise, $(d^*d)^n=(dd^*)^n$, for all $n\geq 2$, so $f_n(d^*d)d^*d=f_n(dd^*)dd^*$, for all $n\in\mathbb{N}$.  Thus
\[d^*d=\lim f_n(d^*d)d^*d=\lim f_n(dd^*)dd^*=dd^*.\qedhere\]
\end{proof}

Let us call $P\subseteq S$ \emph{productive} if, for all $a\in S$,
\[\tag{Productive}\Phi(a)\in\mathrm{cl}(aP)\cap\mathrm{cl}(Pa).\]
This is analogous to the condition defining `algebraic quasi-Cartan pairs' in \cite{ArmstrongCastroClarkCourtneyLinMcCormickRamaggeSimsSteinberg2021}.

\begin{prp}\label{Productivity=>Bistability}
Every closed productive subsemigroup of $S$ is bistable.
\end{prp}

\begin{proof}
Say $P$ is a closed productive subsemigroup of $S$.  If $ab\in P$ then, taking $(p_n)\subseteq P$ with $p_na\rightarrow\Phi(a)$, we see that $\Phi(a)b=\lim p_nab\in\mathrm{cl}(PP)\subseteq P$.  Taking $(p_n)\subseteq P$ with $bp_n\rightarrow\Phi(b)$, the same argument shows that $a\Phi(b)\in P$.
\end{proof}

\begin{cor}\label{ProductiveProp}
If $Z$ is productive then both $Z$ and $\Phi$ are bistable.
\end{cor}

\begin{proof}
If $Z$ is productive then $Z$ is bistable, as above.  But then the larger C*-algebra $\Phi[S]$ is also productive and hence bistable, again by the above result.
\end{proof}

Again productivity of $Z$ is equivalent to productivity of its positive unit ball.

\begin{prp}
$Z$ is productive if and only if $Z^1_+$ is productive.
\end{prp}

\begin{proof}
If $Z^1_+$ is productive then certainly the larger set $Z$ is productive.  Conversely, say $Z$ is productive and take $a\in S$.  Then we have $(z_n)\subseteq Z$ with $az_n\rightarrow\Phi(a)$ so
\begin{align*}
\Phi(a)z_n=\Phi(az_n)&\rightarrow\Phi(\Phi(a))=\Phi(a).\\
(a-\Phi(a))z_n&\rightarrow\Phi(a)-\Phi(a)=0.
\end{align*}
This implies that $z_nz_n^*\Phi(a)\Phi(a)^*=\Phi(a)z_nz_n^*\Phi(a)^*\rightarrow\Phi(a)\Phi(a)^*$.  Letting
\[y_n=f(z_n^*z_n)=f(z_nz_n^*)\in Z^1_+,\]
where $f_n$ is the function on $\mathbb{R}_+$ defined by $f(x)=x\wedge1$, it follows that
\begin{align*}
\|\Phi(a)-\Phi(a)y_n\|^2&=\|\Phi(a)(1-y_n)^2\Phi(a)^*\|\leq\|\Phi(a)(1-y_n)\Phi(a)^*\|\\
&=\|(1-y_n)\Phi(a)\Phi(a)^*\|\leq\|(1-z_nz_n^*)\Phi(a)\Phi(a)^*\|\rightarrow0,
\end{align*}
i.e. $\Phi(a)y_n\rightarrow\Phi(a)$.  Also
\begin{align*}
\|(a-\Phi(a))y_n\|^2&=\|(a-\Phi(a))y_n^2(a-\Phi(a))^*\|\leq\|(a-\Phi(a))y_n(a-\Phi(a))^*\|\\
&\leq\|(a-\Phi(a))z_nz_n^*(a-\Phi(a))^*\|=\|(a-\Phi(a))z_n\|^2\rightarrow0,
\end{align*}
i.e. $(a-\Phi(a))y_n\rightarrow0$.  Combined with the above, this yields $ay_n\rightarrow\Phi(a)$.  A dual argument yields $(y_n)\subseteq Z^1_+$ with $y_na\rightarrow\Phi(a)$, showing that $Z^1_+$ is productive.
\end{proof}

Again let us call $\Phi$ productive when $\Phi[S]$ is productive, i.e. for all $a\in S$,
\[\tag{$\Phi$-Productive}\Phi(a)\in\mathrm{cl}(a\Phi[S])\cap\mathrm{cl}(\Phi[S]a)\]
When $\Phi$ is productive, the kernel of $\Phi$ on $S$ is determined by its range.

\begin{prp}
If $\Phi$ is productive then $S\cap\Phi^{-1}\{0\}=S_\Phi^0$ where
\begin{equation}\label{PhiKernelS}
S_\Phi^0=\{a\in S:\Phi[S]\cap\mathrm{cl}(a\Phi[S])\cap\mathrm{cl}(\Phi[S]a)=\{0\}\}.
\end{equation}
\end{prp}

\begin{proof}
If $a\in S\cap\Phi^{-1}\{0\}$, $(d_n)\subseteq\Phi[S]$ and $ad_n\rightarrow e\in\Phi[S]$ then
\[e=\Phi(e)=\lim(\Phi(ad_n))=\lim(\Phi(a)d_n)=0.\]
This shows that $S\cap\Phi^{-1}\{0\}\subseteq\{a\in S:\Phi[S]\cap\mathrm{cl}(a\Phi[S])=\{0\}\}\subseteq S_\Phi^0$.

Conversely, if $\Phi$ is productive then $\Phi(a)\in\mathrm{cl}(a\Phi[S])\cap\mathrm{cl}(\Phi[S]a)\cap\Phi[S]$, for all $a\in S$, and hence $S_\Phi^0\subseteq S\cap\Phi^{-1}\{0\}$.
\end{proof}

Finally we check that all the properties we have been examining are indeed valid for the sections of Fell bundles that we are primarily interested in.

\begin{prp}\label{FellWellStructured}
Assume $\langle A\rangle=\langle\rho\rangle_\mathsf{r}$, for some Fell bundle $\rho:B\twoheadrightarrow\Gamma$, i.e.
\[A=\mathcal{C}_\mathsf{r}(\rho),\qquad S=\mathcal{S}_\mathsf{r}(\rho),\qquad Z=\mathcal{Z}_\mathsf{r}(\rho)\qquad\text{and}\qquad\Phi(a)=a_{\Gamma^0}.\]
\begin{enumerate}
\item\label{PhiNormalProductive} $\Phi$ is normal and productive.
\item\label{ZBinormalBistable} $Z$ is binormal and bistable.
\item\label{Categorical=>ZProductive} $Z$ is productive if $\rho$ is categorical.
\end{enumerate}
\end{prp}

\begin{proof}\
\begin{itemize}
\item[\eqref{PhiNormalProductive}] To see that $\Phi$ is normal, take $a\in A$ and $s\in S$.  As $\mathrm{supp}(s)$ is a slice,
\begin{align*}
\mathrm{supp}(s\Phi(a)s^*)&\subseteq\mathrm{supp}(s)\Gamma^0\mathrm{supp}(s)^{-1}\subseteq\mathsf{r}[\mathrm{supp}(s).\\
\mathrm{supp}(\Phi(sas^*))&\subseteq\mathrm{supp}(s)\Gamma\mathrm{supp}(s)^{-1}\cap\Gamma^0\subseteq\mathsf{r}[\mathrm{supp}(s)].
\end{align*}
So we only need to check that $s^*\Phi(a)s$ and $\Phi(s^*as)$ agree on $\mathsf{r}(\mathrm{supp}(s))$.  To see this just note that, for any $\gamma\in\mathrm{supp}(s)$,
\[(s\Phi(a)s^*)(\mathsf{r}(\gamma))=s(\gamma)a(\mathsf{s}(\gamma))s(\gamma)^*=sas^*(\mathsf{r}(\gamma))=\Phi(sas^*)(\mathsf{r}(\gamma)).\]

To see that $\Phi$ is productive, again take $s\in S$.  For all $n\in\mathbb{N}$, note that
\[\mathrm{supp}(f_n(\Phi(s)\Phi(s)^*))=\mathrm{supp}(s)\cap\Gamma^0.\]
As $\mathrm{supp}(s)$ is a slice, $\mathrm{supp}(s)\cap\Gamma^0\cap\mathsf{r}(\mathrm{supp}(s)\setminus\Gamma^0)=\emptyset$ and hence
\[f_n(\Phi(s)\Phi(s)^*)s=f_n(\Phi(s)\Phi(s)^*)\Phi(s)\rightarrow\Phi(s).\]
Likewise, $sf_n(\Phi(s^*)\Phi(s))\rightarrow\Phi(s)$, showing that $\Phi$ is productive.

\item[\eqref{ZBinormalBistable}] To see that $Z$ is binormal, take $z\in Z$ and $a,b\in S$ with $ab\in Z$.  For any $\gamma\in\Gamma$, we have $\alpha,\beta\in\Gamma$ with $\gamma=\alpha\beta$ and, setting $x=\mathsf{s}(\alpha)=\mathsf{r}(\beta)$,
\[azb(\gamma)=a(\alpha)z(x)b(\beta)\in\mathbb{C}a(\alpha)b(\beta)=\mathbb{C}ab(\gamma).\]
If $\gamma\notin\Gamma^0$ then this yields $azb(\gamma)=ab(\gamma)=0_\gamma$, as $\mathrm{supp}(ab)\subseteq\Gamma^0$, while if $\gamma\in\Gamma^0$ then this yields $azb\in\mathbb{C}ab(\gamma)\subseteq\mathbb{C}B^0$, as $ab[\Gamma^0]\subseteq\mathbb{C}B^0$.  Thus $azb\in\mathcal{Z}_\mathsf{r}(\rho)=Z$, showing that $Z$ is binormal.

To see that $Z$ is bistable just note that, for any $a,b\in S$ with $ab\in Z$,
\[\Phi(a)b=a\Phi(b)=(ab)_{\mathrm{supp}(a)\cap\Gamma^0}=(ab)_{\mathrm{supp}(b)\cap\Gamma^0}.\]
Thus $\mathrm{supp}(\Phi(a)b)\subseteq\Gamma^0$ and $(\Phi(a)b)[\Gamma^0]\subseteq ab[\Gamma^0]\subseteq\mathbb{C}B^0$ so $\Phi(a)\Phi(b)\in Z$.

\item[\eqref{Categorical=>ZProductive}] Now say $\rho$ is categorical.  To show $Z$ is productive it suffices to show that, for all $s\in\mathcal{S}_\mathsf{c}(\rho)$, we have $z\in Z$ with $sz=\Phi(s)=zs$.  But if $s\in\mathcal{S}_\mathrm{c}(\rho)$ then
\[\qquad\qquad J=\mathrm{cl}(\mathrm{supp}(s))\cap\Gamma^0\quad\text{and}\quad K=\mathsf{r}[\mathrm{cl}(\mathrm{supp}(s))\setminus\Gamma^0]\cup\mathsf{s}[\mathrm{cl}(\mathrm{supp}(s))\setminus\Gamma^0]\]
are disjoint compact subsets of the unit space $\Gamma^0$.  As $\rho$ is categorical, we then have $z\in Z^1_+$ with $z(x)=1_x$, for all $x\in J$, and $z(x)=0_x$, for all $x\in K$, and hence $sz=\Phi(s)=zs$, as required. \qedhere
\end{itemize}
\end{proof}

\section{Domination}\label{Domination}

First let us reiterate our standing assumption.
\begin{center}
\textbf{$\langle A\rangle=(A,S,Z,\Phi)$ is a structured C*-algebra.}
\end{center}

\begin{dfn}[from \cite{Bice2020Rings} and \cite{Bice2020Rep}]
For any $a,s,b\in S$ we define
\begin{align*}
a<_sb\qquad&\Leftrightarrow\qquad a=asb=bsa,\quad as,sa\in\Phi[S]\quad\text{and}\quad bs,sb\in Z.\\
a<b\qquad&\Leftrightarrow\qquad\exists s\in S\ (a<_sb).
\end{align*}
When $a<b$ we say $b$ \emph{dominates} $a$.
\end{dfn}

This is essentially an algebraic version of compact containment $\Subset$.  Indeed, if $\langle A\rangle=\langle\rho\rangle_\mathsf{r}$, for some corical Fell bundle $\rho:B\twoheadrightarrow\Gamma$, and $a,b\in S=\mathcal{S}_\mathsf{r}(\rho)$ then
\begin{equation}\label{<Sub}
a<b\qquad\Leftrightarrow\qquad\mathrm{supp}(a)\Subset b^{-1}[B^\times],
\end{equation}
by essentially the same argument as in the proof of \cite[Proposition 4.4]{BiceClark2020}.  In \eqref{USub} below, we will also show that domination corresponds to compact containment on the associated open slices of the ultrafilter groupoid.

Before that, however, we need to examine some basic algebraic properties of $<$.  The first thing to note is that domination is transitive.

\begin{prp}
For any $a,b,c,s\in S$,
\[\tag{Transitivity}\label{Transitivity}a<b<_sc\qquad\Rightarrow\qquad a<_sc.\]
\end{prp}

\begin{proof}
If $a<_tb<_sc$ then $as=atbs\in\Phi[S]\Phi[S]\subseteq\Phi[S]\supseteq\Phi[S]\Phi[S]\ni sbta=sa$ and $csa=csbta=bta=a=atb=atbsc=asc$ and hence $a<_sc$.
\end{proof}

As $\Phi[S]^*=\Phi[S]$ and $Z^*=Z$, domination also respects the involution, i.e.
\[a<_sb\qquad\Rightarrow\qquad a^*<_{s^*}b^*.\]
The continuous functional calculus also yields the following similar result.

\begin{prp}
For any $a,b,s\in S$,
\begin{equation}\label{a*<_bs}
a<_sb\qquad\Rightarrow\qquad a^*<_bs.
\end{equation}
\end{prp}

\begin{proof}
If $a<_sb$ then, in particular, $a=bsa$ and hence $aa^*=bsaa^*=aa^*bs$, as $aa^*\in S_+\subseteq\Phi[S]$ and $bs\in Z\subseteq\mathsf{Z}(\Phi[S])$.  It follows that
\[a^*bs=\lim a^*f_n(aa^*)bs=\lim a^*f_n(aa^*)=a^*.\]
Likewise, we also see that $sba^*=a^*$.  Finally just note that
\[a^*b=a^*s^*b^*b\in\Phi[S]\Phi[S]\subseteq\Phi[S]\]
and, likewise, $ba^*\in\Phi[S]$.
\end{proof}

We can also show that $Z$ could have been replaced with $Z^1_+$ in the definition above without affecting the domination relation $<$.

\begin{prp}\label{ZDomination}
If $a<b$ then we have $s\in S$ with
\begin{equation}\label{Z1+}
a<_sb\qquad\text{and}\qquad sb,bs\in Z^1_+
\end{equation}
\end{prp}

\begin{proof}
Say $a<_sb$.  We claim that we can modify $s$ to ensure that $bs,sb\in Z_+$.  For this it suffices to show we can replace $s$ with $ss^*b^*$, because $bss^*b^*\in A_+\cap ZZ\subseteq Z_+$ and $ss^*b^*b=sbb^*s^*\in Z_+$ \textendash\, as $sb,bs\in Z$ commute with $bb^*,b^*b\in S_+\subseteq\Phi[S]$,
\[ss^*b^*bb^*b=sbb^*s^*b^*b=sbb^*bb^*s^*,\]
and, likewise, $ss^*b^*b(b^*b)^k=sb(b^*b)^kb^*s^*$ so
\[ss^*b^*b=\lim_nss^*b^*bf_n(b^*b)=\lim_nsbf_n(b^*b)b^*s^*=sbb^*s^*.\]

To show that we even take $bs,sb\in Z_+^1$, it suffices to show that we can replace $s$ with $sg(bs)$, where $g(x)=x\wedge x^{-1}$ on $\mathbb{R}_+$.  This is because $bsg(bs)\in Z^1_+$ (as $xg(x)=x^2\wedge1\leq1$) and $sg(bs)b=sbg(sb)\in Z^1_+$ \textendash\, taking a sequence of polynomials $(g_n)$ approaching $g$ uniformly on compact subsets of $\mathbb{R}_+$, we then get
\[g(bs)b=\lim g_n(bs)b=\lim bg_n(sb)=bg(sb).\]
To see that we can indeed replace $s$ with $sg(bs)$, note that these polynomials $(g_n)$ can be chosen so that, for all $n$, $g_n(0)=0$ and $g_n(1)=1$ so $g(bs)a=\lim g_n(bs)a=a$ and hence $asg(bs)b=g(bs)asb=a=bsg(bs)a$.  Also $asg(bs)\in\Phi[S]Z\subseteq\Phi[S]$ and $sg(bs)a=g(sb)sa\in Z\Phi[S]\subseteq\Phi[S]$, as required.
\end{proof}

When $Z$ is binormal, we can also split the middle witness of the domination relation into two separate witnesses on the left and right.

\begin{prp}
If $Z$ is binormal then, for all $a,b,s,t\in S$,
\begin{equation}\label{st<}
a=asb=bta,\quad as,ta\in\Phi[S]\quad\text{and}\quad sb,bt\in Z\qquad\Rightarrow\qquad a<b.
\end{equation}
\end{prp}

\begin{proof}
We claim that the given conditions imply that
\[a<_{sbt}b.\]
To see this, first note that $asbtb=btasb=a$, as $as\in\Phi[S]$ commutes with $bt\in Z$, and, likewise, $bsbta=btasb=a$.  Also $asbt\in\Phi[S]Z\subseteq\Phi[S]$ and $sbta\in Z\Phi[S]\subseteq\Phi[S]$.  Lastly, as $Z$ is binormal, $sbtb\in sZb\subseteq Z$ and $bsbt\in bZt\subseteq Z$, as required.
\end{proof}

We will also have occasion to consider the `strong domination' relation $\ll$ from \cite[II.3.4.3]{Blackadar2017}.  Specifically, for any $y,z\in Z$ we define
\[y\ll z\qquad\Leftrightarrow\qquad y=yz.\]
Note $y\ll z$ if and only if $y<_{zz^*}z$.  Identifying $y$ and $z$ with their Gelfand transforms on the spectrum of $Z$, we see that $y\ll z$ precisely when $\mathrm{supp}(y)\subseteq z^{-1}\{1\}$.

\begin{lem}
If $Z$ is binormal then, for any $y,z\in Z$ and $a,s,b\in S$,
\begin{equation}\label{SubLem}
a<_sb\quad\text{and}\quad y\ll z\qquad\Rightarrow\qquad a-az<_sb-by.
\end{equation}
\end{lem}

\begin{proof}
First we must show that $a-az$ is actually member of $S$.  To see this, note $aa^*(a-az)=a(a^*a-a^*az)\in S\Phi[S]\subseteq S$.  Thus $(aa^*)^n(a-az)\in S$, for all $n\in\mathbb{N}$, and hence $a-az=\lim f_n(aa^*)(a-az)\in S$.  Likewise, $b-by\in S$.

Now $(a-az)s=as-asbzs\in\Phi[S]-\Phi[S]bZs\subseteq\Phi[S]$, as $Z$ is binormal.  Also $s(a-az)=sa-saz\in\Phi[S]-\Phi[S]Z\subseteq\Phi[S]$.  Similarly, binormality yields $(b-by)s=bs-bys\in Z-bZs\subseteq Z$.  Also $s(b-by)=sb-sby\in Z-ZZ\subseteq Z$ and
\[(a-az)s(b-by)=a-ay-az+ay=a-az=a-az-ay+ay=(b-by)s(a-az).\qedhere\]
\end{proof}

For any $T\subseteq S$ and relation $R$ on $S$ such as $<,>,\ll$ or $\gg$, let
\[T^R=\{s\in S:\exists t\in T\ (t\mathrel{R}s)\}.\]
In particular, $Z^\gg$ denotes the strongly dominated elements of $Z$, while $S^>$ denotes the dominated elements of $S$ and $S^>_\Sigma=(S^>)_\Sigma$ denotes their finite sums (e.g. when $\langle A\rangle=\langle\rho\rangle_\mathsf{r}$, for corical $\rho$, $Z^\gg=\mathcal{Z}_\mathsf{c}(\rho)$ and $S^>_\Sigma=\mathcal{C}_\mathsf{c}(\rho)$ -- see \eqref{SumsCompactSupports} below).  

\begin{prp}
We always have
\begin{align}
\label{PhiSZ}\Phi[S^>_\Sigma]&=Z^>=\Phi[S]^>=S^>\cap\Phi[S].\\
\label{clPhiSZ}\Phi[\mathrm{cl}(S^>_\Sigma)]&=\mathrm{cl}(Z^>).\\
\label{clZgg}Z&=\mathrm{cl}(Z^\gg).\\
\label{Zgg}S^>\cap Z&=Z^\gg.
\end{align}
\end{prp}

\begin{proof}\
\begin{itemize}
\item[\eqref{PhiSZ}] Say $a=\sum_{k\leq n}a_k$, for some $a_1,\ldots,a_n\in S^>$.  Then we have $z_1,\ldots,z_k\in Z^1_+$ with $a_kz_k=a_k$ and hence $z_k\Phi(a_k)=\Phi(a_k)z_k=\Phi(a_kz_k)=\Phi(a_k)$, for all $k\leq n$.  Letting $z=\bigvee_{k\leq n}z_k$, it follows that $z\Phi(a_k)=\Phi(a_k)z=\Phi(a_k)$, for all $k\leq n$.  Then $z\Phi(a)=\Phi(a)z=\Phi(a)$ so $\Phi(a)<_zz$.  Thus $\Phi[S^>_\Sigma]\subseteq Z^>$.

Now if $a<_sb\in\Phi[S]$ then $a=asb\in\Phi[S]\Phi[S]\subseteq\Phi[S]$ and hence
\[Z^>\subseteq\Phi[S]^>\subseteq S^>\cap\Phi[S]=\Phi[S^>\cap\Phi[S]]\subseteq\Phi[S^>_\Sigma]\subseteq Z^>.\]

\item[\eqref{clPhiSZ}] As $\Phi$ is contractive, $\Phi[S^>_\Sigma]=Z^>$ and $\mathrm{cl}(Z^>)\subseteq\mathrm{cl}(\Phi[S])=\Phi[S]$,
\[\Phi[\mathrm{cl}(Z^>)]\subseteq\Phi[\mathrm{cl}(S^>_\Sigma)]\subseteq\mathrm{cl}(\Phi[S^>_\Sigma])=\mathrm{cl}(Z^>)=\Phi[\mathrm{cl}(Z^>)].\]

\item[\eqref{clZgg}] As $\gg$ is only defined on $Z$, certainly $\mathrm{cl}(Z^\gg)\subseteq\mathrm{cl}(Z)=Z$.  To get the reverse inclusion, for each $n\in\mathbb{N}$, define a function $g_n$ on $\mathbb{R}_+$ by
\[g_n(x)=0\vee(2^nx-1)\wedge1.\]
Then $g_{n+1}(zz^*)\gg g_n(zz^*)z\rightarrow z$, for any $z\in Z$, showing $Z\subseteq\mathrm{cl}(Z^\gg)$.

\item[\eqref{Zgg}] Certainly $Z^\gg\subseteq Z^>\subseteq S^>$ and $Z^\gg\subseteq Z$, i.e. $Z^\gg\subseteq S^>\cap Z$.  Conversely, if $Z\ni z<_sb$, for some $b\in S$, then $z\ll sb$, showing that $S^>\cap Z\subseteq Z^\gg$.
\qedhere
\end{itemize}
\end{proof}

Let us observe that bistability of $Z$ implies bistability of $\Phi$ on dominated elements $S^>$.  We will extend this to `compatible sums' below in \autoref{BistableCompatibleSums}.

\begin{prp}
If $Z$ is bistable, $a,b\in S$ and $a\in S^>$ or $b\in S^>$ then
\begin{equation}\label{DominatedBistability}
ab\in\Phi[S]\qquad\Rightarrow\qquad\Phi(a)b,a\Phi(b)\in\Phi[S].
\end{equation}
\end{prp}

\begin{proof}
Say $ab\in\Phi[S]$ and $a<_sc$.  Then $\Phi(c)s,c\Phi(s)\in Z$, by bistability, so
\begin{align*}
\Phi(a)b&=\Phi(csa)b=\Phi(c)sab\in Z\Phi[S]\subseteq\Phi[S].\\
a\Phi(b)&=csa\Phi(b)=c\Phi(sab)=c\Phi(s)ab\in Z\Phi[S]\subseteq\Phi[S].
\end{align*}
This proves the $a\in S^>$ case, while the $b\in S^>$ case follows by a dual argument.
\end{proof}

As usual, we let $C^*(B)$ denote the C*-subalgebra generated by any $B\subseteq A$.

\begin{prp}
If $\Phi$ is normal and $Z$ is binormal then
\[\mathrm{cl}(S^>_\Sigma)=C^*(S^>).\]
\end{prp}

\begin{proof}
when $\Phi$ is normal and $Z$ is binormal, $<$ respects products, by \cite[Proposition 3.7]{Bice2020Rep}, and hence $S^>$ is a *-subsemigroup.  Also $S^>=\mathbb{C}S^>$ so $S^>_\Sigma$ is *-algebra generated by $S^>$ and hence $\mathrm{cl}(S^>_\Sigma)=C^*(S^>)$.
\end{proof}

\begin{prp}
Assume $A=\mathrm{cl}(S_\Sigma^>)$ and $\Phi$ is productive $(\!$and hence bistable$)$.  Then $\Phi$ is the unique bistable expectation from $A$ onto $\Phi[S](=\mathrm{span}(S_+)=C^*(S_+))$.
\end{prp}

\begin{proof}
Say we have another bistable expectation $\Psi$ from $A$ onto $\Phi[S]$.  First we show that $\Phi$ and $\Psi$ agree on $S^>$.  To see this, take $a,b,s\in S$ with $a<_sb$.  Then $\Psi(a)<_sb$, as $\Psi$ is bistable, and hence $a-\Psi(a)=(as-\Psi(a)s)b\in\Phi[S]S\subseteq S$.  As in the first part of the proof of \eqref{PhiKernelS}, it follows that $a-\Psi(a)\in S\cap\Psi^{-1}\{0\}\subseteq S_\Psi^0=S_\Phi^0$.  As $\Phi$ is productive, $S_\Phi^0=S\cap\Phi^{-1}\{0\}$ so $a=\Psi(a)+(a-\Psi(a))$ yields a decomposition of $a$ into elements of $\Phi[A]$ and $\Phi^{-1}\{0\}$.  But $a=\Phi(a)+(a-\Phi(a))$ is another such decomposition and, as $\Phi$ is a idempotent linear operator, such decompositions are unique.  In particular, $\Psi(a)=\Phi(a)$, which then extends to all $a\in\mathrm{cl}(S_\Sigma^>)=A$.
\end{proof}

Towards the end we will restrict our attention to `sum-structured C*-algebras' where, in particular, $\Phi$ is productive and $A=\mathrm{cl}(S_\Sigma^>)$.  Sum-structured C*-algebras could thus be viewed as certain triples $(A,S,Z)$ where $C^*(S_+)$ is the range of some productive expectation $\Phi$, which is then unique by the above result.  This would be more in line with the usual approach to Cartan subalgebras/Cartan pairs.  However, as the expectation plays such a key role in our work, we feel it deserves to be stated explicitly as an intrinsic part of the structured C*-algebra.  In particular, this makes it clear that structure preserving morphisms should preserve the expectations themselves, not just their ranges (i.e. the morphisms should also preserve their kernels -- see \autoref{RangeKernelPhi}).

\subsection{Interpolation}\label{Interpolative}

The compact containment relation on open subsets in locally compact spaces satisfies a converse of transitivity known as interpolation.  This is also a key property of the domination relation.

\begin{prp}\label{<Interpolation}
For any $a,b\in S$,
\[\tag{Interpolation}\label{Interpolation}a<b\qquad\Rightarrow\qquad\exists c\in S\ (a<c<b).\]
\end{prp}

\begin{proof}
Let $g$ and $h$ be continuous functions on $\mathbb{R}_+$ such that $g(0)=h(0)=0$, $g(1)=h(1)=1$ and $g(x)h(x)x=g(x)$.  If $a<_sb$ and $sb,bs\in Z^1_+$ then we claim
\begin{equation}\label{asbInt}
a<_sbg(sb)<_{h(sb)s}b.
\end{equation}
For $a<_sbg(sb)$, note that $sbg(sb)\in ZZ\subseteq Z$, $bg(sb)s=g(bs)bs\in ZZ\subseteq Z$ and $asbg(sb)=a=g(bs)bsa=bg(sb)sa$ (note $ag(sb)=a=g(bs)a$ as $g(1)=1$ \textendash\, see the argument above).  To see that $sbg(sb)<_{h(sb)s}b$, note $h(sb)sb\in Z$, $bh(sb)s=h(bs)bs\in Z$,
\[bg(sb)h(sb)s=g(bs)h(bs)bs=g(bs)=h(sb)sbg(sb)\in Z\subseteq\Phi[S]\]
and hence $bg(sb)h(sb)sb=g(bs)b=bg(sb)=bh(sb)sbg(sb)$, proving the claim.
\end{proof}

Similarly, we can always witness domination with dominated elements.

\begin{prp}
For any $a,b\in S$,
\[a<b\qquad\Rightarrow\qquad\exists s\in S^>\ (a<_sb).\]
\end{prp}

\begin{proof}
We claim that $a<_tb$ implies $a<_sb$ and $s<_rt$, where $s=tg(bt)=g(tb)t$ and $r=bh(tb)=h(bt)b$, for the same $g$ and $h$ as above.  To see this note that $asb=ag(tb)tb=a=btg(bt)a=bsa$, $sb=g(tb)tb\in ZZ\subseteq Z\supseteq ZZ\ni btg(bt)=bs$ and $as=atg(bt)\in\Phi[S]Z\subseteq\Phi[S]\supseteq\Phi[S]Z\ni g(tb)ta=sa$, showing that $a<_sb$.  Also $sr=g(tb)tbh(tb)=g(tb)\in Z\subseteq\Phi[S]\supseteq Z\ni g(bt)=h(bt)btg(bt)=rs$, $srt=g(tb)t=s=tg(bt)=trs$ and $rt=h(bt)bt\in ZZ\subseteq Z\supseteq ZZ\ni tbh(tb)=tr$, showing that $s<_rt$, as required.
\end{proof}

Actually, we can improve \eqref{Interpolation} with involutive norm $1$ witnesses.  For this we will need the following lemma.  First define functions $g$ and $(h_n)$ on $\mathbb{R}_+$ by
\begin{equation}\label{ghn}
g(x)=(2x-1)\vee0\qquad\text{and}\qquad h_n(x)=x^{-1/2}\wedge nx.
\end{equation}

\begin{lem}
For all $a,s,b\in S$ with $sb,bs\in Z^1_+$ and $n\geq8\|s\|^3$,
\begin{equation}\label{RepLem}
a<_sb\qquad\Rightarrow\qquad a<_{s_n}s_n^*,
\end{equation}
where $s_n=g(sb)b^*_n=b^*_ng(bs)\in S^1$ and $b_n=h_n(bb^*)b=bh_n(b^*b)\in S^1$.
\end{lem}

\begin{proof}
First note that
\[g(sb)b^*bb^*=b^*bg(sb)b^*=b^*g(bs)bb^*=b^*bb^*g(bs)\]
and, likewise, $g(sb)b^*(bb^*)^n=b^*(bb^*)^ng(bs)$, for all $n\in\mathbb{N}$.  As each $h_n$ is a limit of polynomials with $h(0)=0$, the same applies, i.e.
\[g(sb)b^*_n=g(sb)b^*h_n(bb^*)=b^*h_n(bb^*)g(bs)=b^*_ng(bs).\]
To see that $b_n,s_n\in S^1$, just note that
\[\|b_n\|^2=\|b^*_nb_n\|=\|h_n(b^*b)b^*bh_n(b^*b)\|\leq\sup_{x\in\mathbb{R}_+}x^{-1/2}xx^{-1/2}=1\]
and hence $\|s_n\|=\|g(sb)b^*_n\|\leq\|b^*_n\|\leq1$.

Next note $(sb)^2=b^*s^*sb\leq\|s\|^2b^*b$.  As $sb\in Z$ commutes with $b^*b\in\Phi[S]$, Gelfand allows us to represent both $sb$ and $b^*b$ on some space $Y$.  For $y\in Y$ with $sb(y)\leq1/2$, we see that $2sb(y)-1\leq0$ and hence $g(sb)(y)=0$.  On the other hand, for $y\in Y$ with $sb(y)\geq1/2$, we see that $1/4\leq sb(y)^2\leq\|s\|^2b^*b(y)$ so $1/8\leq\|s\|^3b^*b(y)^{3/2}$ and hence $b^*b(y)^{-1/2}\leq8\|s\|^3b^*b(y)$.  If $n\geq8\|s\|^3$ then $h_n(b^*b)\sqrt{b^*b}(y)=1$.  This observation when $sb(y)\geq1/2$ and the above observation when $sb(y)\leq1/2$ together imply that $g(sb)h_n(b^*b)\sqrt{b^*b}=g(sb)$ and hence
\[s_ns_n^*=g(sb)b^*_nb_ng(sb)=g(sb)h_n(b^*b)b^*bh_n(b^*b)g(sb)=g(sb)^2\in Z.\]
Likewise, $(bs)^2=bss^*b^*\leq\|s\|^2bb^*$ and so, for $n\geq8\|s\|^3$, we again get
\[s_n^*s_n=g(bs)b_nb^*_ng(bs)=g(bs)h_n(bb^*)bb^*h_n(bb^*)g(bs)=g(bs)^2\in Z.\]
As $g(1)=1$ and $asb=a=bsa$, it follows that $ag(sb)=a=g(bs)a$ and hence
\[as_ns_n^*=ag(sb)^2=a=g(bs)^2a=s_n^*s_na.\]
By \eqref{a*<_bs}, $s_na=g(sb)b^*_na=g(sb)h_n(b^*b)b^*a\in Z\Phi[S]\Phi[S]\subseteq\Phi[S]$ and, likewise
\[as_n=ab^*_ng(sb)=ab^*h_n(bb^*)g(bs)\in\Phi[S]\Phi[S]Z\subseteq\Phi[S].\]
This shows that $a<_{s_n}s_n^*$, as required.
\end{proof}

\begin{thm}
For any $a,b\in S$,
\[\tag{$1$-Interpolation}\label{1Interpolation}a<b\qquad\Rightarrow\qquad\exists s\in S^1\ (a<_ss^*<b).\]
\end{thm}

\begin{proof}
If $a<b$ then \eqref{Interpolation} and \eqref{Z1+} yields $c,t\in S$ with $a<_tc<b$ and $tc,ct\in Z^1_+$.  By \eqref{RepLem}, $a<_ss^*$, where $n\geq8\|t\|^3$ and
\[s=g(tc)h_n(c^*c)c^*=c^*h_n(cc^*)g(ct)\in S^1\]
As $c<b$ and $g(tc),g(ct),h_n(c^*c),h_n(cc^*)\in\Phi[S]$, it follows that $s^*<b$ too.
\end{proof}

\begin{rmk}\label{<*Rem}
Let $<^*$ be the stronger version of the domination relation given by
\[a<^*b\qquad\Leftrightarrow\qquad a<_{b^*}b.\]
This is essentially the same as the relation $\precsim$ considered in \cite{Bice2021}.  Our original plan was to do the same work here with $<^*$ rather than $<$, \cite{Bice2021} being the first step in this direction.  However, our collaboration with Clark in \cite{BiceClark2020} revealed that $<$ is a better replacement for $<^*$.  For one thing, $<$ can be defined without any involution and thus makes sense in more general structures, like the Steinberg rings considered in \cite{Bice2020Rings}.  Even in C*-algebras, where an involution is available, $<$ still has better interpolation properties.  Indeed, much like the relation $\ll$ considered in \cite[II.3.4.3]{Blackadar2017}, $<^*$ only satisfies the weaker condition
\[\tag{Weak Interpolation}a<^* b<^* c\qquad\Rightarrow\qquad\exists d,e\ (a<^* d<^* e<^* c).\]
On the plus side, \eqref{1Interpolation} does at least assure us that $<^*$ is close to $<$, sufficiently so that often either $<$ or $<^*$ could be used, e.g. when defining the ultrafilter groupoid as in the next section.
\end{rmk}

We can also improve \eqref{Interpolation} with pairs of lower bounds, showing that $(S,<)$ is a predomain/abstract basis in the sense of \cite{Keimel2016}/\cite[Lemma 5.1.32]{Goubault2013}.  For this, we first need to modify \eqref{RepLem}.  Again define $g$ and $(h_n)$ on $\mathbb{R}_+$ as in \eqref{ghn}.

\begin{lem}
For all $a,s,b\in S$ with $sb,bs\in Z^1_+$ and $n\geq8\|s\|^3$,
\begin{equation}\label{RepLem2}
a<_sb\qquad\Rightarrow\qquad a<_{s_n}b,
\end{equation}
where $s_n=g(sb)b^*_n=b^*_ng(bs)$ and $b_n=h_n^2(bb^*)b=bh_n^2(b^*b)$.
\end{lem}

\begin{proof}
As in the proof of \eqref{RepLem}, $g(sb)b^*_n=b^*_ng(bs)$ and $g(sb)h_n(b^*b)\sqrt{b^*b}=g(sb)$.  Thus $g(sb)h_n^2(b^*b)b^*b=g(sb)$ and hence
\[s_nb=g(sb)b^*_nb=g(sb)h_n^2(b^*b)b^*b=g(sb)\in Z.\]
Likewise, $bs_n=bb^*_ng(bs)=bb^*h_n^2(bb^*)g(bs)=g(bs)\in Z$.  As $g(1)=1$ and $asb=a=bsa$, it follows that $ag(sb)=a=g(bs)a$ and hence
\[as_nb=ag(sb)=a=g(bs)a=bs_na.\]
By \eqref{a*<_bs}, $s_na=g(sb)b^*_na=g(sb)h_n^2(b^*b)b^*a\in Z\Phi[S]\Phi[S]\subseteq\Phi[S]$ and, likewise
\[as_n=ab^*_ng(sb)=ab^*h_n^2(bb^*)g(bs)\in\Phi[S]\Phi[S]Z\subseteq\Phi[S].\]
This shows that $a<_{s_n}b$, as required.
\end{proof}

An immediate consequence of \eqref{RepLem2} worth noting is the following.

\begin{prp}\label{ZWitness}
If $a<b\in Z$ then $a<_zb$, for some $z\in Z$.
\end{prp}

\begin{proof}
Just note that if $a<_sb\in Z$ and $sb,bs\in Z^1_+$ then \eqref{RepLem2} yields $a<_zb$ where $z=g(sb)h_n^2(bb^*)b\in Z$, for sufficiently large $n$.
\end{proof}

Now for the main result.

\begin{thm}
For all $a,b\in S$,
\[\tag{Predomain}\label{Predomain}a,b<c\qquad\Rightarrow\qquad\exists d\in S\ (a,b<d<c).\]
\end{thm}

\begin{proof}
Say $a<_rc$ and $b<_sc$.  Take $n\geq8(\|r\|\vee\|s\|)$ and let
\begin{align*}
c_n&=h_n^2(cc^*)c=ch_n^2(c^*c).\\
r'&=g(rc)c^*_n=c^*_ng(cr).\\
s'&=g(sc)c^*_n=c^*_ng(cs).
\end{align*}
By \eqref{RepLem2}, $a<_{r'}c$ and $a<_{s'}c$.  We claim that
\[a,b<_tc\quad\text{where}\quad t=(g(rc)\vee g(sc))c^*_n=c^*_n(g(cr)\vee g(cs)).\]
For the last equation, just note that Stone-Weierstrass yields polynomials $p_n(x,y)$ with $p_n(0,0)=0$ which converge uniformly to $x\vee y$ on $[0,1]\times[0,1]$ and hence
\[(g(rc)\vee g(sc))c^*_n=\lim_kp_k(g(rc),g(sc))c^*_n=\lim_kc^*_np_k(g(cr),g(cs))=c^*_n(g(cr)\vee g(cs)).\]
As in the proof of \eqref{RepLem2}, we see that $r'c=g(rc)$ and $s'c=g(sc)$ and hence
\begin{align*}
tc&=(g(rc)\vee g(sc))c^*_nc=\lim_kp_k(g(rc),g(sc))c^*_nc=\lim_kp_k(g(rc),g(sc))\\
&=g(rc)\vee g(sc)\in Z.
\end{align*}
Likewise, $ct=g(cr)\vee g(cs)\in Z$.  Also, as $a=ar'_nc=ag(rc)$,
\[\|a-atc\|=\|ag(rc)^k-ag(rc)^k(g(rc)\vee g(sc))\|\leq\|g(rc)^k-g(rc)^k(g(rc)\vee g(sc))\|\rightarrow0\]
as $k\rightarrow\infty$, and hence $atc=a$.  Likewise, $cta=a$ and $tsb=b=bst$.  Further note that $at=ac^*_n=ac^*h_n^2(cc^*)(g(cr)\vee g(ct))\in\Phi[S]$ and, likewise, $ta,bt,tb\in\Phi[S]$.  This proves the claim and then \eqref{asbInt} in the proof of \eqref{Interpolation} yields
\[a,b<_tcg(tc)<_{h_3^2(tc)t}c.\qedhere\]
\end{proof}

It follows that dominated sums behave as expected.

\begin{cor}
For any $a,b,c\in S$,
\begin{equation}\label{DomSums}
a,b<c\qquad\Rightarrow\qquad a+b<c.
\end{equation}
\end{cor}

\begin{proof}
First note that if $a,b<c$ then $a+b$ is a member of $S$.  Indeed, if $a<_rc$ and $b<_sc$ then $a+b=arc+bsc=(ar+bs)c\in(\Phi[S]+\Phi[S])S\subseteq S$.  By \eqref{Predomain} (or its proof), we can actually take $r=s$.  Then $(a+b)s=as+bs\in\Phi[S]+\Phi[S]\subseteq\Phi[S]$,  $(a+b)sc=asc+bsc=a+b$, etc. so $a+b<_sc$.
\end{proof}

\section{Compatibility}\label{Compatibility}
While domination encodes compact containment, we will also need to examine a somewhat similar relation encoding `slice unions'.

\begin{dfn}
For any $a,b,s\in S$, we define
\begin{align*}
a\sim_sb\qquad&\Leftrightarrow\qquad a< s^*\quad\text{and}\quad sb,bs\in\Phi[S].\\
a\sim b\qquad&\Leftrightarrow\qquad\exists s\in S\ (a\sim_sb).
\end{align*}
When $a\sim b$ we say that $a$ is \emph{compatible} with $b$.
\end{dfn}

This generalises the usual notion for inverse semigroups (see \cite{Lawson1998}).  Again the motivating situation to consider is $\langle A\rangle=\langle\rho\rangle_\mathsf{r}$, for some corical Fell bundle $\rho$.

\begin{prp}\label{FellSim}
If $\rho:B\twoheadrightarrow\Gamma$ is a corical Fell bundle, $\langle A\rangle=\langle\rho\rangle_\mathsf{r}$ and $a,b\in S^>$,
\begin{equation}\label{FellSimEq}
a\sim b\qquad\Leftrightarrow\qquad\mathrm{cl}(\mathrm{supp}(a)\cup\mathrm{supp}(b))\text{ is a slice}.
\end{equation}
\end{prp}

\begin{proof}
If $a\sim_sb$ then $a<s^*$ so $\mathrm{supp}(a)\Subset O^{-1}$, where $O=s^{-1}[B^\times]=(s^*[B^\times])^{-1}$, which is open because $B$ is corical.  As $a,b\in S^>$, we know that $\mathrm{cl}(\mathrm{supp}(a))$ and $\mathrm{cl}(\mathrm{supp}(b))$ are compact slices.  If $\mathrm{cl}(\mathrm{supp}(a)\cup\mathrm{supp}(b))$ were not a slice then either
\[\mathrm{cl}(\mathrm{supp}(a))^{-1}\mathrm{cl}(\mathrm{supp}(b))\nsubseteq\Gamma^0\quad\text{or}\quad\mathrm{cl}(\mathrm{supp}(b))\mathrm{cl}(\mathrm{supp}(a))^{-1}\nsubseteq\Gamma^0.\]
In the former case we would have $\mathrm{supp}(sb)\subseteq O\mathrm{supp}(b)\nsubseteq\Gamma^0$, while in the latter we would have $\mathrm{supp}(bs)\subseteq\mathrm{supp}(b)O\nsubseteq\Gamma^0$, contradicting $sb\in\Phi[S]$ or $bs\in\Phi[S]$ respectively.  Thus $\mathrm{cl}(\mathrm{supp}(a)\cup\mathrm{supp}(b))$ must have been a slice.

Conversely, say $K=\mathrm{cl}(\mathrm{supp}(a)\cup\mathrm{supp}(b))$ is a slice.  Note $K$ is compact, again because $a,b\in S^>$.  By \cite[Proposition 6.3]{BiceStarling2018}, $K$ is contained in some open slice $O$.  Take any continuous $\lambda:\Gamma\rightarrow[0,1]$ with
\[\mathrm{cl}(\mathrm{supp}(a))\subseteq\mathrm{supp}(\lambda)\subseteq O.\]
Taking any $t>a$, setting $s(\gamma)=\lambda(\gamma^{-1})t(\gamma^{-1})$ yields $s\in S$ such that $a<s^*$ and $sb,bs\in\Phi[S]$, as $\mathrm{supp}(sb)\cup\mathrm{supp}(bs)\subseteq O^{-1}O\cup OO^{-1}\subseteq\Gamma^0$, i.e. $a\sim_sb$.
\end{proof}

\begin{rmk}\label{GeneralCompatibility}
If we were to allow more general $b\in S$ above (while still restricting to dominated $a\in S^>$), then we would have a similar characterisation, namely
\[a\sim b\qquad\Leftrightarrow\qquad O\cup\mathrm{supp}(b)\text{ is a slice, for some open }O\Supset\mathrm{supp}(a).\]
On the other hand, when $a\notin S^>$, we certainly can not have $a\sim b$.  However, if $a\in\mathcal{S}_\mathsf{c}(\rho)$ and $O\cup\mathrm{supp}(b)$ is a slice, for some open $O\Supset\mathrm{supp}(a)$, then we can at least find $a_1,\ldots,a_n\in S$ that are compatible with each other and with $b$ such that $a=\sum_{k=1}^na_k$, by essentially the same argument as in the proof of \eqref{SumsCompactSupports} below.
\end{rmk}

Let us now return to the general case where $\langle A\rangle$ is merely a structured C*-algebra.  In this case, we can list some basic properties of compatibility as follows.

\begin{prp}\label{CompatibilityProperties}
For all $a,b,c,s\in S$,
\begin{align}
\label{Dom=>Sim}a<_sb\qquad&\Rightarrow\qquad a\sim_sb.\\
\label{SimAux}a\sim_sc>b\qquad&\Rightarrow\qquad a\sim_sc.\\
\label{2Dom=>Sim}a,b<c\qquad&\Rightarrow\qquad a\sim b.\\
\label{SimInt}\exists c\in S\ (a<c\sim b)\qquad&\Leftrightarrow\qquad a\sim b.\\
\label{StarComp}ab^*,a^*b\in\Phi[S]\qquad&\Leftarrow\qquad a\sim b.\\
\label{PhiComp}\forall d\in\Phi[S]\ (ad,da\sim b)\qquad&\Leftarrow\qquad a\sim b.\\
\label{StarSim}\exists s\in S^1\ (a<_ss^*\text{ and }bs,sb\in\Phi[S])\qquad&\Leftarrow\qquad a\sim b.\hspace{70pt}
\end{align}
\end{prp}

\begin{proof}\
\begin{itemize}
\item[\eqref{Dom=>Sim}] If $a<_sb$ then $sb,bs\in Z\subseteq\Phi[A]$ and $a<_{b^*}s^*$, by \eqref{a*<_bs}, so $a\sim_sb$.

\item[\eqref{SimAux}] If $a\sim_sc$ and $b<_tc$, for some $t\in S$, then $sb=sctb\in\Phi[S]\Phi[S]\subseteq\Phi[S]$ and $bs=btcs\in\Phi[S]\Phi[S]\subseteq\Phi[S]$ and hence $a\sim_sb$.

\item[\eqref{2Dom=>Sim}] If $a,b<c$ then $a\sim c>b$, by \eqref{Dom=>Sim}, so $a\sim b$, by \eqref{SimAux}.

\item[\eqref{SimInt}] If $a<c\sim_sb$ then $a<c<s^*$ so $a<s^*$, by \eqref{Transitivity}, hence $a\sim_sb$.  Conversely, if $a\sim_sb$ then $a<s^*$ and hence $a<c<s^*$, for some $c\in S$, by \eqref{Interpolation}, so $a<c\sim_sb$.

\item[\eqref{StarComp}] If $a\sim_sb$ then $a<_ts^*$, for some $t\in S$, so $ab^*=ats^*b^*\in\Phi[S]\Phi[S]^*\subseteq\Phi[S]$ and  $a^*b=a^*t^*sb\in\Phi[S]^*\Phi[S]\subseteq\Phi[S]$.

\item[\eqref{PhiComp}] Just note that $a<s^*$ implies $ad,da<s^*$, by \cite[Proposition 3.8]{Bice2020Rep}.

\item[\eqref{StarSim}] If $a\sim_tb$ then $a<t^*$ so $a<_ss^*<t^*$, for some $s\in S^1$, by \eqref{1Interpolation}.  Then $s^*\sim_tb$ so $bs,sb\in\Phi[S]$, by \eqref{StarComp}. \qedhere
\end{itemize}
\end{proof}

Let us denote the finite compatible sums of $S$ by
\[S^\sim_\Sigma=\{{\textstyle\sum}_{k=1}^ns_k:s_1,\ldots,s_n\in S\text{ and }\forall j,k\leq n\ (s_j\sim s_k)\}\]
(for an idea of what $S^\sim_\Sigma$ is meant to represent, see \eqref{SumsCompactSupports} below).
Note $S^\sim_\Sigma\subseteq S^>_\Sigma$, as $a\sim b$ implies $a\in S^>$.  By the definition of a structured C*-algebra, $S_+\subseteq\Phi[S]$ and this now easily extends to their finite compatible sums.

\begin{prp}
The *-squares of $S^\sim_\Sigma$ are contained in $\Phi[S]$, i.e.
\begin{equation}\label{SumStarSquares}
S^\sim_{\Sigma+}\subseteq\Phi[S].
\end{equation}
\end{prp}

\begin{proof}
If $a=\sum_{k=1}^ns_k$ and $s_j\sim s_k$, for all $j,k\leq n$, then $s_j^*s_k\in\Phi[S]$, for all $j,k\leq n$, and hence $a^*a=\sum_{j,k\leq n}s_j^*s_k\in\Phi[S]$.
\end{proof}

For suitable $\Phi$ and $Z$, pairwise compatibility implies simultaneous compatibility.

\begin{lem}
If $\Phi$ is normal and $Z$ is binormal and bistable then, for all $a,b\in S$, we have $c\in a\Phi[S]\cap\Phi[S]a\cap b\Phi[S]\cap\Phi[S]b$ such that $a^>\cap b^>\subseteq c^>$.  Consequently,
\begin{equation}\label{SimultaneousCompatibility}
a\sim b_1,\ldots,b_n\qquad\Rightarrow\qquad\exists s\in S\ (a\sim_sb_1,\ldots,b_n).
\end{equation}
\end{lem}

\begin{proof}
By \cite[Lemma 1.13]{Bice2020Rings}, $s<_{a'}a$ and $s<_{b'}b$ implies $s<_{a'aa'}\Phi(ab')b$.  Taking $d=aa'aa'\Phi(ab')bb'b$, it follows that $a^>\cap b^>\subseteq d^>$, by \cite[Proposition 3.9]{Bice2020Rep}.  Moreover, $d\in aZZ\subseteq aZ$, $d\in ZZ\Phi[S]Zb\subseteq\Phi[S]b$ and
\begin{align*}
d&=aa'aa'bb'\Phi(ab')b=aa'aa'b\Phi(b'a)b'b=aa'aa'bb'b\Phi(b'a)=bb'aa'aa'b\Phi(b'a)\\
&\in bb'Zb\Phi[S]\subseteq bZ\Phi[S]\subseteq b\Phi[S].
\end{align*}
Likewise, we have $c\in Za\cap\Phi[S]d\cap d\Phi[S]$ with $a^>\cap d^>\subseteq c^>$.  Then $c\in Za\subseteq\Phi[S]a$, $c\in\Phi[S]d\subseteq\Phi[S]\Phi[S]b\subseteq\Phi[S]b$ and $c\in d\Phi[S]\subseteq(aZ\cap b\Phi[S])\Phi[S]\subseteq a\Phi[S]\cap b\Phi[S]$, as well as $a^>\cap b^>\subseteq a^>\cap d^>\subseteq c^>$.  This proves the first part.

Now if $a\sim_{s_k}b_k$, for all $k\leq n$, then what we just proved yields $s>a^*$ with
\[s\in\bigcap_{k\leq n}s_k\Phi[S]\cap\Phi[S]s_k.\]
For all $k\leq n$, it follows that $sb_k\in\Phi[S]s_kb_k\in\Phi[S]\Phi[S]\subseteq\Phi[S]$ and, likewise, $b_ks\in\Phi[S]$ and hence $a\sim_sb_k$.
\end{proof}

We say $S$ \emph{has compatible sums} when, for all $a,b\in S$,
\[\tag{Compatible Sums}a\sim b\qquad\Rightarrow\qquad a+b\in S.\]

\begin{prp}
If $\Phi$ is normal and $Z$ is binormal and bistable then
\[S\text{ has compatible sums}\qquad\Rightarrow\qquad S^\sim_\Sigma\subseteq S.\]
\end{prp}

\begin{proof}
If $S$ has compatible sums then, by definition, sums of compatible pairs are in $S$.  Now assume that sums of pairwise compatible $n$-tuples lie in $S$ and take pairwise compatible $a_0,\ldots,a_n\in S$.  Then $a_0\sim a_1,\ldots,a_n$ and hence $a_0\sim_sa_1,\ldots,a_n$, for some $s\in S$, by \eqref{SimultaneousCompatibility}.  Then $a_0\sim_sa_1+\ldots+a_n$ and hence $a_0+\ldots+a_n\in S$, i.e. sums of pairwise compatible $(n+1)$-tuples lie in $S$.  Thus, by induction, $S^\sim_\Sigma\subseteq S$.
\end{proof}

Under the same conditions, we can show $\sim$ is symmetric on $S^>$, as one would expect from the characterisation in \autoref{FellSim} for corical Fell bundle sections.

\begin{prp}
If $\Phi$ is normal and $Z$ is binormal and bistable then
\[\tag{Symmetry}\label{Symmetry}a\sim b\in S^>\qquad\Rightarrow\qquad b\sim a.\]
\end{prp}

\begin{proof}
As $b\in S^>$, we have $c,c',t\in S$ with $b<_{t^*}t<_{c'}c$.  As in \eqref{asbInt}, we may further assume that $tc',c't\in Z$.  As $a\sim b$, we also have $s\in S$ with $a<_{s^*}s$ and $s^*b,bs^*\in\Phi[S]$.  Let
\begin{align*}
u&=t-ts^*ss^*s-ss^*ss^*t+ss^*ss^*ts^*ss^*s+2\Phi(ts^*)ss^*s-\Phi(ts^*)ss^*ss^*ss^*s\\
&=t-ts^*ss^*s-ss^*ss^*t+ss^*ss^*ts^*ss^*s+2ss^*s\Phi(s^*t)-ss^*ss^*ss^*s\Phi(s^*t)
\end{align*}
(using the normality of $\Phi$ for this equality).  We claim that $b<_{t^*}u$.  Indeed, thanks to the binormality and bistability of $Z$, all the terms in $t^*u$, $ut^*$ and $uc'$ lie in $Z$.  Also $u=uc'c\in Zc\subseteq S$ and, as $bs^*,s^*b\in\Phi[S]$ and hence $bs^*ss^*s=ss^*bs^*s=ss^*ss^*b$,
\[bt^*u=b-bs^*ss^*s-ss^*ss^*b+ss^*ss^*bs^*ss^*s+2\Phi(bs^*)ss^*s-\Phi(bs^*)ss^*ss^*ss^*s=b.\]
Likewise, $ut^*b=b$ and hence $b<_{t^*}u$.  Further note that
\[a^*u=a^*t-a^*ts^*ss^*s-a^*t+a^*ts^*ss^*s+2\Phi(a^*t)-\Phi(a^*t)=\Phi(a^*t)\in\Phi[S].\]
Likewise, $ua^*\in\Phi[S]$ and hence $b\sim_ua$.
\end{proof}

\begin{cor}\label{SimultaneousSimInt}
If $\Phi$ is normal and $Z$ is binormal and bistable, for any compatible $a_1,\ldots,a_n\in S$, we have compatible $b_1,\ldots,b_n\in S$ with $a_k<b_k$, for all $k\leq n$.
\end{cor}

\begin{proof}
If $a_j\sim a_k$, for all $j,k\leq n$, then \eqref{SimultaneousCompatibility} and \eqref{Interpolation} yield $b_1\in S^>$ with $a_1<b_1\sim a_2,\ldots a_n$.  By \eqref{Symmetry}, $a_2\sim b_1,a_3,\dots,a_n$ and so again we have $b_2\in S^>$ with $a_2<b_2\sim b_1,a_3,\ldots,a_n$.  Continuing in this way and again appealing to \eqref{Symmetry} at the end we get $b_k>a_k$ with $b_j\sim b_k$, for all $j,k\leq n$.
\end{proof}

\begin{prp}
If $Z$ is bistable and $a\in S^\sim_\Sigma$ then we have $z\in Z^1_+$ with
\[za=\Phi(a)=az.\]
\end{prp}

\begin{proof}
First assume $a\in S^>$.  Taking $s\in S^1$ with $a<_ss^*$, we see that
\[\Phi(s^*)sa=\Phi(s^*sa)=\Phi(a)=\Phi(ass^*)=as\Phi(s^*).\]
As $Z$ and hence $Z^1_+$ is bistable and $s^*s,ss^*\in Z^1_+$, \autoref{a*a=aa*} yields
\[Z^1_+\ni\Phi(s^*)s=\Phi(\Phi(s^*)s)=\Phi(s^*)\Phi(s)=\Phi(s)\Phi(s^*)=\Phi(s\Phi(s^*))=s\Phi(s^*).\]
This proves the result for $a\in S^>$.  Note that this then implies
\[\Phi(a)=\Phi(\Phi(a))=\Phi(az)=\Phi(a)z=z\Phi(a).\]

For the general result, take compatible $a_1,\ldots,a_n\in S$, for some $n\geq2$.  Then we have $(s_j^k)\in S^{>1}$, for $j,k\leq n$ with $j\neq k$, such that
\[a_j<_{s_j^k}s_j^{k*}\qquad\text{and}\qquad s_j^ka_k,a_ks_j^k\in\Phi[S],\]
Let $z_j^k=\Phi(s_j^k)\Phi(s_j^k)^*=\Phi(s_j^k)^*\Phi(s_j^k)\in Z^1_+$, by bistability.  Further let
\[z_j=\bigwedge_{j\neq k\leq n}z_j^k\qquad\text{and}\qquad z=\bigvee_{j\leq n}z_j.\]
For all $j\leq n$, we see that $\Phi(a_j)=\Phi(a_j)z_j^k$, whenever $j\neq k\leq n$, so $\Phi(a_j)=\Phi(a_j)z_j$ and hence $\Phi(a_j)z=\Phi(a_j)$, as $z_j\leq z\in Z^1_+$.  Moreover, $\Phi(a_j)\Phi(s_k^j)=a_j\Phi(s_k^j)$, by \eqref{DominatedBistability}, so $(a_j-\Phi(a_j))z_k^j=0=(a_j-\Phi(a_j))z_k$, as $z_k\leq z_k^j$, whenever $k\neq j$.  As $(a_j-\Phi(a_j))z_j=\Phi(a_j)-\Phi(a_j)=0$ too, it follows that $(a_j-\Phi(a_j))z=0$.  Then
\[a_jz=\Phi(a_j)z+(a_j-\Phi(a_j))z=\Phi(a_j).\]
As this holds for all $j\leq n$, we see that $\Phi(a)=\sum_{k=1}^n\Phi(a_j)=\sum_{k=1}^na_jz=az$.  A dual argument yields $\Phi(a)=za$.
\end{proof}

\begin{cor}\label{BistableCompatibleSums}
If $S\subseteq\mathrm{cl}(S^\sim_\Sigma)$ then
\[Z\text{ is productive}\qquad\Leftrightarrow\qquad Z\text{ is bistable}\qquad\Rightarrow\qquad\Phi\text{ is bistable}\]
\end{cor}

\begin{proof}
We already saw in \autoref{ProductiveProp} that both $Z$ and $\Phi$ are bistable when $Z$ is productive.  Conversely, say $Z$ is bistable and $S\subseteq\mathrm{cl}(S^\sim_\Sigma)$, so any $a\in S$ is a limit of some $(a_n)\subseteq S^\sim_\Sigma$.  By the above result, we have $(z_n)\subseteq Z^1_+$ with $z_na_n=\Phi(a_n)=a_nz_n\rightarrow\Phi(a)$.  As $(z_n)\subseteq Z^1$ and $a_n\rightarrow a$, it follows that $az_n,z_na\rightarrow\Phi(a)$ too so $\Phi(a)\in\mathrm{cl}(aZ)\cap\mathrm{cl}(Za)$.  This shows that $Z$ is productive.
\end{proof}

Let us finish by illuminating what $S^\sim_\Sigma$ is for sections of corical Fell bundles.

\begin{prp}
If $\langle A\rangle=\langle\rho\rangle_\mathsf{r}$, for some corical Fell bundle $\rho:B\twoheadrightarrow\Gamma$, then
\begin{equation}\label{SumsCompactSupports}
S^\sim_\Sigma=\mathcal{S}_\mathsf{c}(\rho)\qquad\text{and}\qquad S^>_\Sigma=\mathcal{C}_\mathsf{c}(\rho).
\end{equation}
\end{prp}

\begin{proof}
If $a=\sum_{k=1}^na_k$, for compatible $a_1,\ldots,a_n\in S$, then $\mathrm{cl}(\mathrm{supp}(a_j)\cup\mathrm{supp}(a_k))$ is slice, for all $j,k\leq n$, by \eqref{FellSimEq}.  Thus $\mathrm{cl}(\mathrm{supp}(a))\subseteq\bigcup_{k=1}^n\mathrm{cl}(\mathrm{supp}(a_k))$ is also slice, which is compact, as $a_1,\ldots,a_n\in S^>$, so $a\in\mathcal{S}_\mathsf{c}(\rho)$.  This shows $S^\sim_\Sigma\subseteq\mathcal{S}_\mathsf{c}(\rho)$.

For the converse, take $a\in\mathcal{S}_\mathsf{c}(\rho)$.  For all $\gamma\in\Gamma$, we have $a_\gamma\in S$ with $a_\gamma(\gamma)\in B^\times$ (see \autoref{CtsSections}), and we can then take open $O_\gamma\Subset a_\gamma^{-1}[B^\times]$ containing $\gamma$.  As $\mathrm{cl}(\mathrm{supp}(a))$ is compact, we have $\gamma_1,\ldots,\gamma_n\in O$ with $\mathrm{cl}(\mathrm{supp}(a))\subseteq\bigcup_{k=1}^nO_{\gamma_k}$.  We can then take a continuous partition of unity $\lambda_1,\dots,\lambda_n,\lambda_{n+1}:\Gamma\rightarrow[0,1]$ subordinate to $O_{\gamma_1},\ldots,O_{\gamma_n},\Gamma\setminus\mathrm{cl}(\mathrm{supp}(a))$ and define $a_k(\gamma)=\lambda_k(\gamma)a(\gamma)$, for all $\gamma\in\Gamma$.  Then $a=\sum_{k=1}^na_k$ and $\mathrm{supp}(a_k)\subseteq O_{\gamma_k}\Subset a_{\gamma_k}^{-1}[B^\times]$ and hence $a_k<a_{\gamma_k}$, for all $k\leq n$.  Also $\mathrm{cl}(\mathrm{supp}(a_j)\cup\mathrm{supp}(a_k))\subseteq\mathrm{cl}(\mathrm{supp}(a))$ and hence $a_j\sim a_k$, for all $j,k\leq n$, by \eqref{FellSimEq}, so $a\in S^\sim_\Sigma$.  This shows $\mathcal{S}_\mathsf{c}(\rho)\subseteq S^\sim_\Sigma$ and hence $S^\sim_\Sigma=\mathcal{S}_\mathsf{c}(\rho)$.

As dominated elements have relatively compact supports, so do their finite sums, i.e. $S^>_\Sigma\subseteq\mathcal{C}_\mathsf{c}(\rho)$.  On the other hand, \eqref{CcSc} yields $\mathcal{C}_\mathsf{c}(\rho)\subseteq\mathcal{S}_\mathsf{c}(\rho)_\Sigma\subseteq(S^\sim_\Sigma)_\Sigma\subseteq S^>_\Sigma$, by what we just proved.  This shows that $S^>_\Sigma=\mathcal{C}_\mathsf{c}(\rho)$.
\end{proof}

\section{Ultrafilters}\label{Ultrafilters}

We have already seen bistability and (bi)normality playing a key role in the previous section, and they will continue to do so when we start examining ultrafilters.  Accordingly, let us introduce the following terminology.

\begin{dfn}
A \emph{well-structured C*-algebra} is a structured C*-algebra $(A,S,Z,\Phi)$ such that $\Phi$ is normal and bistable and $Z$ (or $Z^1_+$) is binormal and bistable.
\end{dfn}

Structured C*-algebras consisting of sections of Fell bundles are indeed well-structured, thanks to \autoref{FellWellStructured}.  Well-structured C*-algebras also satisfy the key assumptions required in \cite{Bice2020Rep} and \cite{Bice2020Rings}, allowing us to make free use of the theory developed there.  For example, $N$ in \cite[\S3]{Bice2020Rep}, which corresponds to our $\Phi[S]$, is required to be `$Z^1_+$-trinormal', but this holds in well-structured C*-structures because $\Phi$ is normal/shiftable -- see \cite[Proposition 1.6]{Bice2020Rings}.  Incidentally, a further `additively invariant absolute metric' is required in \cite[\S6]{Bice2020Rings}, which for us is just the usual metric $\|a-b\|$ on $A$, but this will not play any role until we consider Weyl seminorms in the next section.

Again to avoid repeating our basic hypotheses we make the following assumption.\\

\begin{center}
\textbf{From now on $(A,S,Z,\Phi)$ is a well-structured C*-algebra.}
\end{center}
\vspace{10pt}

Recall that a \emph{filter} is a non-empty down-directed up-set, i.e.
\[\tag{Filter}\label{Filter}a,b\in U\qquad\Leftrightarrow\qquad\exists c\in U\ (c<a,b).\]
An \emph{ultrafilter} is a maximal proper filter.  Ultrafilters are denoted by
\[\mathcal{U}(S)=\{U\subseteq S:U\text{ is an ultrafilter}\}.\]
The canonical topology on $\mathcal{U}(S)$ is given by the basis $(\mathcal{U}_a)_{a\in S}$ where
\[\mathcal{U}_a=\{U\in\mathcal{U}(S):a\in U\}.\]

\begin{thm}\label{UltrafilterGroupoid}
$\mathcal{U}(S)$ is an \'etale groupoid with inverse $U^*=\{u^*:u\in U\}$ and
\begin{align*}
T\cdot U&=(TU)^<=\{a>tu:t\in T\text{ and }u\in U\}\quad\text{when}\quad0\notin TU.
\end{align*}
\end{thm}

\begin{proof}
By \cite[Theorem 6.1 and Theorem 10.7]{Bice2020Rep}, $\mathcal{U}(S)$ is an \'etale groupoid with product $T\cdot U=(TU)^<$ but where the inverse of $U\in\mathcal{U}(S)$ is instead given by
\[U^{-1}=\{s\in S:U\ni a<_sb\}.\]
However, we know that $a<_sb$ is equivalent to $a<_{b^*}s^*$, by \eqref{a*<_bs}, and hence
\[U^{-1}=\{u^*:u\in U\}=U^*.\]
The last thing to note is $0\notin TU$ precisely when $T\cdot U$ is defined which, by definition, means $\mathsf{s}(T)=\mathsf{r}(U)$, i.e. $(T^*T)^<=(UU^*)^<$.  Indeed, if $\mathsf{s}(T)=\mathsf{r}(U)=V$ then
\[V=V\cdot V=(\mathsf{s}(T)\mathsf{r}(U))^<=(T^*TUU^*)^<\]
so $0\in TU$ would imply $0\in V$, which is impossible for any proper filter.  Conversely, if $\mathsf{s}(T)\neq\mathsf{r}(U)$ then $0\in\mathsf{s}(T)\mathsf{r}(U)$, as shown in the last part of the proof of \cite[Theorem 10.7]{Bice2020Rep}, and hence $0\in T\mathsf{s}(T)\mathsf{r}(U)U\subseteq TU$.
\end{proof}

Ultrafilters will constitute the base points of our Weyl bundle, to be constructed in \autoref{TheWeylBundle}.  The motivation for this comes from the fact that, for corical Fell bundles, points of the base groupoid correspond precisely to ultrafilters of their sections.

\begin{prp}
If $\langle A\rangle=\langle\rho\rangle_\mathsf{r}$, for some corical Fell bundle $\rho:B\twoheadrightarrow\Gamma$, then we have \'etale groupoid isomorphism $\iota_\rho:\Gamma\rightarrow\mathcal{U}(S)$ given by $\iota_\rho(\gamma)=S_\gamma$ where
\begin{equation}\label{gamma->Sgamma}
S_\gamma=\{a\in S:a(\gamma)\in B^\times\}
\end{equation}
\end{prp}

\begin{proof}
Essentially the same as the proof of \cite[Theorem 5.3]{BiceClark2020}.
\end{proof}

In \cite{Bice2020Rep}, we also examined cosets, i.e. $U\subseteq S$ such that
\[\tag{Coset}UU^*U\subseteq U=U^<.\]
By \cite[Theorem 6.1]{Bice2020Rep}, the non-empty cosets also form an \'etale groupoid, which contain the filters as an ideal subgroupoid, by \cite[Proposition 10.5]{Bice2020Rep}.

Here, the units in the filter groupoid are the $\Phi$-invariant ones.

\begin{prp}\label{UnitFilters}
The unit filters are the non-empty cosets $U$ with $\Phi[U]\subseteq U$.
\end{prp}

\begin{proof}
By \cite[Proposition 10.6]{Bice2020Rep}, the unit filters are precisely the non-empty cosets $U$ that are generated by their diagonal, meaning that $U=(U\cap\Phi[S])^<$.

So if $U$ is a unit filter and $u\in U$ then we have $d\in U\cap\Phi[S]$ with $d<u$ and hence $d<\Phi(u)$, by \cite[Proposition 1.11]{Bice2020Rings}, as $Z$ is bistable.  Thus $\Phi(u)\in U^<\subseteq U$, showing that $\Phi[U]\subseteq U$.

Conversely, say $U$ is a coset with $\Phi[U]\subseteq U$.  For any $u\in U$, we have $t\in U$ with $t<u$ and hence $\Phi(t)<u$, again by \cite[Proposition 1.11]{Bice2020Rings}, as $\Phi[S]$ is bistable.  As $\Phi(t)\in U\cap\Phi[S]$, this shows that $(U\cap\Phi[S])^<\subseteq U$ so $U$ is a unit filter.
\end{proof}

Let us denote unit ultrafilters by $\mathcal{U}^0$ so, by \autoref{UnitFilters},
\[\mathcal{U}^0=\{U\in\mathcal{U}(S):\Phi[U]\subseteq U\}.\]
General filters can also be characterised in a similar way.

\begin{cor}
The filters are precisely the subsets $U\subseteq S$ such that
\begin{equation}\label{FilterChar}
\Phi[UU^*]U\subseteq U=U^<.
\end{equation}
\end{cor}

\begin{proof}
If $U\subseteq S$ is a filter then $\mathsf{r}(U)\supseteq UU^*$ is a unit filter (as long as $U\neq\emptyset$) so $\Phi[\mathsf{r}[U]]\subseteq\mathsf{r}[U]$, by the above result, and hence $\Phi[UU^*]U\subseteq\mathsf{r}[U]U\subseteq U$.

Conversely, say $\Phi[UU^*]U\subseteq U=U^<$.  For $U$ to be a filter it suffices that $U$ is directed, as we are already assuming $U^<\subseteq U$.  Accordingly, take $t,u\in U$.  As $U\subseteq U^<$, we have $a,b,s,c\in U$ with $a<t$ and $b<_{s^*}c<u$.  By \cite[Lemma 1.14]{Bice2020Rings}, $\Phi(as^*)c<t$ and, by \cite[Proposition 3.8]{Bice2020Rep}, $\Phi(as^*)c<u$.  As $\Phi(as^*)c\in\Phi[UU^*]U\subseteq U$, this shows that $U$ is directed.
\end{proof}

Among filters, the ultrafilters have a prime-like characterisation.

\begin{thm}
The ultrafilters are the proper non-empty filters $U\subseteq S$ such that
\begin{equation}\label{Ultra+}
a+b\in U\qquad\Rightarrow\qquad a^<\subseteq U\quad\text{or}\quad b^<\subseteq U.
\end{equation}
\end{thm}

\begin{proof}
Looking for a contradiction, say we have an ultrafilter $U$ with $a+b\in U$, even though we also have $q,r\in S\setminus U$ with $a<q$ and $b<r$.  As $<$ is interpolative, we have a sequences $(p_n),(q_n)\subseteq U$ with $a<q_{n+1}<_{p_n^*}q_n<q$, for all $n$.  Replacing each $p_n$ with $p_nq_n^*p_n$ if necessary, we may also assume that $p_{n+1}<_{q_n^*}p_n$, by \cite[Proposition 4.9]{Bice2020Rep}.  Then we get a filter $T$ containing $U$ and $q$ given by
\[T=\{t>\Phi(vp_n^*)q_n:v\in U\text{ and }n\in\mathbb{N}\},\]
again thanks to \cite[Proposition 1.11 and Lemma 1.14]{Bice2020Rings} and \cite[Propositions 3.7 and 3.8]{Bice2020Rep}.  As $U$ is an ultrafilter and $q\in T\setminus U$, it follows that $T=S$ and hence $\Phi(vp_j^*)q_j=0$, for some $v\in U$ and $j\in\mathbb{N}$.  Taking $s$ with $a<_sq_j$, it follows that $\Phi(vp_j^*)a=\Phi(vp_j^*)q_jsa=0$.  Likewise, we have $w\in U$ and $k\in\mathbb{N}$ with $\Phi(wp_k^*)b=0$.  Note that we can assume $j=k$, e.g. if $j<k$ then
\[\Phi(wp_j^*)b=\Phi(wp_k^*q_kp_j^*)b=\Phi(wp_k^*)q_kp_j^*b=q_kp_j^*\Phi(wp_k^*)b=0.\]
Taking $x\in U$ with $x<v,w$ it then follows that $\Phi(xp_j^*)a=0=\Phi(xp_j^*)b$ which does indeed yield the contradiction $0=\Phi(xp_j^*)(a+b)\in\Phi[UU^*]U\subseteq U$.

Conversely, say $U$ is a filter but not an ultrafilter, so $U$ is contained in a strictly larger proper filter $T$.  Take $t\in T\setminus U$ and $u,w\in U$ with $u<w$.  As $T$ is a filter, we can further take $q,r,s\in T$ with $q<_{r'}r<_{s'}s<t$ and $s<_{u'}u$.  Note
\[(u-us'ss's)+us'ss's=u\in U.\]
By \eqref{SubLem}, $u-us'ss's<w-wr'r$.  If we had $w-wr'r\in T$ then this would imply $0=(w-wr'r)r'q\in TT^*T\subseteq T$, a contradiction, so $w-wr'r\notin T\supseteq U$.  Next note that $us'ss'=us'su'us'\in\Phi[S]$, by \eqref{Diagonal}, as $us'su'\in uZu'\subseteq Z$, $su'\in Z$ and $su'us'\in sZs'\subseteq Z$, as $Z$ is binormal.  Thus, by \cite[Proposition 3.8]{Bice2020Rep}, $us'ss's<t\notin U$, showing that \eqref{Ultra+} fails for $U$.
\end{proof}

To analyse the relationship between $<$ and the topology of $\mathcal{U}(S)$ in more detail, we need the following result.  This strengthens \cite[Lemma 7.1]{Bice2020Rings} and corresponds to the trapping condition of Lawson-Lenz --  see \cite[Theorem 5.19]{LawsonLenz2013}.

For any $a\in S$, define
\[a^\perp=\{s<b:a^>\cap b^>=\{0\}\}.\]

\begin{lem}
For any $a,b,s,t\in S$,
\begin{equation}\label{Trapping}
\tag{Trapping}a<b\ \ \text{and}\ \ s<t\quad\Rightarrow\quad\exists c\in b^>\cap s^\perp\ (a^>\cap t^\perp\subseteq c^>).
\end{equation}
Moreover, if $a\not<t$ then we can choose $c\neq0$.
\end{lem}

\begin{proof}
Take $b'\in S$ with $a<_{b'}b$.  By \eqref{Interpolation}, we have $t',u,u',v,v'\in S$ with $s<_{u'}u<_{v'}v<_{t'}t$.  Let
\[c=a-v\Phi(v'a)\qquad\text{and}\qquad d=b-\Phi(bu')u.\]
We claim that $c<_{b'}b$ and $c<_{b'}d$.  First note that
\[ca^*a=aa^*a-v\Phi(v'a)a^*a=(aa^*-v\Phi(v'a)a^*)a\in(\Phi[S]-\Phi[S])S\subseteq S.\]
So $c=\lim_ncf(a^*a)^n\in S$ and, likewise, $d\in S$.  Next note that, as $\Phi$ is shiftable,
\begin{equation}\label{bu'u}
\Phi(bu')u=\Phi(vv'bu')vv'u=v\Phi(v'bu'vv'u)=v\Phi(v'bu'u).
\end{equation}
As $Z$ is binormal, $b'vv'bu'u\in b'ZbZ\subseteq Z$ and hence $b'\Phi(bu')u=b'v\Phi(v'bu'u)\in Z$, as $Z$ is bistable.  Thus $b'd=b'b-b'\Phi(bu')u\in Z-Z\subseteq Z$.  Again as $Z$ is binormal and bistable, $db'=bb'-\Phi(bu')ub'\in Z-Z\subseteq Z$.  We also see that 
\begin{align*}
v\Phi(v'a)b'&=v\Phi(v'ab'b)b'=v\Phi(v'a)b'bb'=vb'b\Phi(v'a)b'=v\Phi(b'bv'a)b'\\
&=vb'\Phi(bv'ab')=vb'\Phi(bv')ab'\in Z\Phi[S]\subseteq\Phi[S]
\end{align*}
and hence $cb'=ab'-v\Phi(v'a)b'\in\Phi[S]-\Phi[S]\subseteq\Phi[S]$.  As $b'\Phi(bu')u\in Z\subseteq\Phi[S]$,
\begin{align*}
v\Phi(v'a)b'\Phi(bu')u&=v\Phi(v'ab'\Phi(bu')u)=v\Phi(v'\Phi(au')u)=v\Phi(v'\Phi(au')u)v'v\\
&=vv'\Phi(\Phi(au')uv')v=\Phi(au')vv'uv'v=\Phi(au')u.
\end{align*}
Also $ab'\Phi(bu')u=\Phi(ab'bu')u=\Phi(au')u$ so $cb'\Phi(bu')u=0$ and thus $cb'd=cb'b=c$.  By \cite[Proposition 3.4]{Bice2020Rep}, this yields $c<_{b'}b$ and $c<_{b'}d$, proving the claim.

Next we claim $d^>\cap s^>=\{0\}$.  Indeed, if $e<_{d'}d$ and $e<s(<_{u'}u<_{v'}v)$ then
\begin{align*}
v'\Phi(bu')ud'e&=v'\Phi(bu'vv')ud'e=\Phi(v'bu'v)v'ud'e=\Phi(v'bu'vv'u)d'e\\
&=\Phi(v'bu'u)d'e=\Phi(v'b)d'eu'u=\Phi(v'b)d'e
\end{align*}
so $v'e=v'dd'e=v'bd'e-v'\Phi(bu')ud'e=v'bd'e-\Phi(v'bd'e)$ and hence $\Phi(v'e)=0$.  Thus $e=vv'e=v\Phi(v'e)=0$.  This proves the claim, and hence $c\in b^>\cap s^\perp$.

Now take $f\in a^>\cap t^\perp$.  Then we have $a',g\in S$ with $f<_{a'}a$ and $f<g$ where $g^>\cap t^>=\{0\}$.  But then \cite[Proposition 3.8]{Bice2020Rep} and \cite[Lemma 1.14]{Bice2020Rings} yield
\[v\Phi(v'a)a'f=v\Phi(v'f)\in g^>\cap t^>=\{0\}\]
and hence $ca'f=aa'f=f$.  Also $a'c=a'a-a'v\Phi(v'a)\in Z$, as $Z$ is binormal and bistable, and $fa'c=ca'fa'c=ca'ca'f$.  Also, as above in \eqref{bu'u},
\[ca'=aa'-v\Phi(v'a)a'=aa'-\Phi(vv'at')ta'\in Z-Z\subseteq Z.\]
By \cite[Proposition 3.4]{Bice2020Rep}, this yields $f<_{a'}c$.  This shows that $a^>\cap t^\perp\subseteq c^>$.

Lastly note $a\not<t$ implies $a\neq v\Phi(v'a)$, by \cite[Proposition 3.8]{Bice2020Rep}, so $c\neq0$.
\end{proof}

Let us write $B\sim C$ to mean $b\sim c$, for all $b\in B$ and $c\in C$.

Recall that $\Subset$ denotes compact containment.

\begin{thm}\label{LCHultrafilters}
$\mathcal{U}(S)$ is locally compact and Hausdorff.  Moreover, for all $a,b\in S$,
\begin{align}
\label{Usub}a^>\subseteq b^>\ \qquad&\Leftrightarrow\qquad\mathcal{U}_a\subseteq\mathcal{U}_b.\\
\label{USub}\exists c<b\ (a^>\subseteq c^>)\qquad&\Leftrightarrow\qquad\mathcal{U}_a\Subset\mathcal{U}_b.\\
\label{USlice}a^>\sim b^>\qquad&\Leftrightarrow\qquad\mathcal{U}_a\cup\mathcal{U}_b\text{ is a slice}.
\end{align}
\end{thm}

\begin{proof}
The fact that $\mathcal{U}(S)$ is locally compact and Hausdorff follows from the theory in \cite{BiceStarling2021}, as explained in \cite[Theorem 7.2]{Bice2020Rings}.  Alternatively, for local compactness, one could appeal to \cite[Corollary 3.31]{Bice2021} (thanks to \eqref{1Interpolation} \textendash\, see \autoref{<*Rem}), while for being Hausdorff one could appeal to \cite[Remark 8.5]{Bice2020Rings}.

For \eqref{Usub}, say $a^>\subseteq b^>$ and $U\in\mathcal{U}_a$.  Then $a\in U\subseteq U^<$ so we have $c\in U$ with $c<a$ and hence $c<b$.  Thus $b\in U^<\subseteq U$ so $U\in\mathcal{U}_b$, which shows that $\mathcal{U}_a\subseteq\mathcal{U}_b$.  Conversely, say $a^>\nsubseteq b^>$ so we have $s<a$ with $s\not<b$.  Use \eqref{Interpolation} to get $(a_n)\subseteq S$ with $s<a_{n+1}<a_n<a$, for all $n\in\mathbb{N}$.  For all $t<b$ and $n\in\mathbb{N}$, \eqref{Interpolation} and \eqref{Trapping} yields non-zero $t_n\in a_n^>\cap t^\perp$ with $a_{n+1}^>\cap c^\perp\subseteq t_n^>$, for some $c$ with $t<c<b$.  We claim that $D=\{t_n:t<b\text{ and }n\in\mathbb{N}\}$ is down-directed.  To see this, take $t_m,u_n\in D$, so we have $c,d\in S$ with $t<c<b$, $u<d<b$, $a_{m+1}^>\cap c^\perp\subseteq t_n^>$ and $a_{n+1}^>\cap d^\perp\subseteq u_n^>$.  By \eqref{Predomain}, we have $e<b$ with $c,d<e$ so $e^\perp\subseteq c^\perp\cap d^\perp$ and hence
\[e_{m+n}\in a_{m+n}^>\cap e^\perp\subseteq a_{m+1}^>\cap a_{n+1}\cap c^\perp\cap d^\perp\subseteq t_m^>\cap u_n^>.\]
This proves the claim.  By Kuratowski-Zorn, $D$ then extends to an ultrafilter $U$ with $(a_n)\subseteq U$ and $U\cap b^>=\emptyset$.  Thus $U\in\mathcal{U}_a\setminus\mathcal{U}_b$ witnesses $\mathcal{U}_a\nsubseteq\mathcal{U}_b$, proving \eqref{Usub}.

For \eqref{USub}, first recall from \cite[Theorem 2.11]{BiceStarling2021} or \cite[Theorem 2.26]{BiceStarling2020HTight} that $c<b$ implies $\mathcal{U}_c\Subset\mathcal{U}_b$.  If $a^>\subseteq c^>$ then \eqref{Usub} yields $\mathcal{U}_a\subseteq\mathcal{U}_c\Subset\mathcal{U}_b$.  Conversely, say $\mathcal{U}_a\Subset\mathcal{U}_b$.  For each $U\in\mathcal{U}_b$, we have $c\in U$ with $c<b$ so $U\in\mathcal{U}_c$.  As $\mathcal{U}_a\Subset\mathcal{U}_b$, we can cover $\mathcal{U}_a$ with $\mathcal{U}_{c_1},\cdots,\mathcal{U}_{c_k}$, for some $c_1,\cdots,c_k<b$.  By \eqref{Predomain}, we have $c\in S$ with $c_1,\ldots,c_k<c<b$ so $\mathcal{U}_a\subseteq\mathcal{U}_{c_1}\cup\cdots\cup\mathcal{U}_{c_k}\subseteq\mathcal{U}_c$ and hence $a^>\subseteq c^>$, by \eqref{Usub} again, which proves \eqref{USub}.

For \eqref{USlice}, first note that, thanks to interpolation,
\[a^>\sim b^>\qquad\Leftrightarrow\qquad a^{>*}b^>\cup a^>b^{>*}\subseteq\Phi[S].\]
Thus $a^>\sim b^>$ implies that, for all $U\in\mathcal{U}_a$ and $V\in\mathcal{U}_b$,
\[U^*V\cap\Phi[S]\neq\emptyset\neq UV^*\cap\Phi[S]\]
and hence $U^*V,V^*U\in\mathcal{U}^0$, by \cite[Proposition 6.2]{Bice2020Rep}.  So $\mathcal{U}_a^{-1}\mathcal{U}_b\cup\mathcal{U}_a\mathcal{U}_b^{-1}\subseteq\mathcal{U}^0$ and hence $\mathcal{U}_a\cup\mathcal{U}_b$ is a slice.

Conversely, say $cd^*\notin\Phi[S]$, for some $c<a$ and $d<b$.  Let $e=cd^*-\Phi(cd^*)\neq0$.  By \eqref{1Interpolation}, we have $a',b',s,t\in S$ with
\[c<_ss^*<_{a'}a\qquad\text{and}\qquad d<_tt^*<_{b'}b.\]
Letting $u=s^*t-\Phi(s^*t)$, \cite[Lemma 8.2]{Bice2020Rings} yields $e<u$.  Using \eqref{Interpolation}, we obtain a filter in $e^<$ containing $u$.  As $e\neq0$, the filter is proper and then Kuratowski-Zorn yields an ultrafilter extension $U$.  As $U\not\ni0=\Phi(u)\in\Phi[U]$, \autoref{UnitFilters} says that $U\notin\mathcal{U}^0$.  By \cite[Proposition 5.6 and Theorem 10.7]{Bice2020Rep}, $V=(Ub)^<\in\mathcal{U}(S)$ with $\mathsf{r}(V)=\mathsf{r}(U)$, as $ubb'=u\in U$.  Moreover, $s^*tb<a$, by \cite[Proposition 3.8]{Bice2020Rep}, and $\Phi(s^*t)b<a$, by \cite[Lemma 1.14]{Bice2020Rings}, and hence $ub<a$, by \eqref{DomSums}, so $a\in V$, i.e. $V\in\mathcal{U}_a$.  Also note $b^*u^*u<_{b'^*}b^*$, by \cite[Proposition 3.4]{Bice2020Rep}, as $b'^*b^*u^*u\in Z\Phi[S]\subseteq\Phi[S]$ and $b^*b'^*b^*u^*u=b^*u^*ub'^*b^*=b^*u^*u$, because $t<_{b'^*}b^*$ and $u=s^*t-\Phi(s^*t)$.  Thus $W=V^*\cdot U=(b^*U^*U)^<\in\mathcal{U}_{b^*}=\mathcal{U}_b^{-1}$ and $U=V\cdot V^*\cdot U=V\cdot W\in\mathcal{U}_a\mathcal{U}_b^{-1}$.  In particular, $\mathcal{U}_a\mathcal{U}_b^{-1}\nsubseteq\mathcal{U}^0$ so $\mathcal{U}_a\cup\mathcal{U}_b$ is not a slice.  Likewise, if we had $c^*d\notin\Phi[S]$ instead then $\mathcal{U}_a^{-1}\mathcal{U}_b\nsubseteq\mathcal{U}^0$ so $\mathcal{U}_a\cup\mathcal{U}_b$ would again fail to be a slice.  Thus $a^>\sim b^>$ whenever $\mathcal{U}_a\cup\mathcal{U}_b$ is slice.
\end{proof}

We can further show that the construction of the ultrafilter groupoid is functorial with respect to structure-preserving morphisms.

\begin{thm}\label{UltraFunctoriality}
Any structure-preserving morphism $\pi$ from one well-structured C*-algebra $(A,S,Z,\Phi)$ to another $(A',S',Z',\Phi')$ defines an \'etale morphism $\underline{\pi}$ in the opposite direction, from an open subgroupoid of $\mathcal{U}(S')$ to $\mathcal{U}(S)$, given by
\begin{equation}\label{underlinepi}
\underline{\pi}(U)=\pi^{-1}[U]^<\quad\text{when}\quad\pi^{-1}[U]\neq\emptyset
\end{equation}
\end{thm}

\begin{proof}
First we must show $\underline{\pi}$ is well-defined, i.e. $T=\pi^{-1}[U]^<$ is an ultrafilter when $U\in\mathcal{U}(S')$ and $\pi^{-1}[U]\neq\emptyset$.  First note $T=T^<$ is immediate from \eqref{Interpolation}.  Next take $a,b,c\in T$, so we have $d,e,f\in\pi^{-1}[U]$ with $d<a$, $e<b$ and $f<c$.  Then $\Phi(de^*)f<\Phi(ab^*)c$, by \cite[Propositions 3.7]{Bice2020Rep} and \cite[Proposition 1.11]{Bice2020Rings}.  As $\pi$ is structure-preserving, $\pi(\Phi(de^*)f)=\Phi'(\pi(d)\pi(e)^*)\pi(f)\in\Phi'[UU^*]U\subseteq U$, by \eqref{FilterChar}, and hence $\Phi(ab^*)c\in\pi^{-1}[U]^<$.  This shows that $\Phi[TT^*]T\subseteq T$ and hence $T$ is also a filter, again by \eqref{FilterChar}.

To see that $T$ is an ultrafilter, take $a,b\in S$ with $a+b\in T$.  Again this means we have $c<a+b$ with $\pi(c)\in U$.  One immediately sees that any structure-preserving morphism respects the domination relation so $\pi(a)+\pi(b)=\pi(a+b)>\pi(c)$ and hence $\pi(a)+\pi(b)\in U$.  By \eqref{Ultra+}, either $\pi(a)^<\subseteq U$ or $\pi(b)^<\subseteq U$.  In the former case, for any $d>a$, \eqref{Interpolation} yields $e$ with $a<e<d$ and then $\pi(a)<\pi(e)$ so $\pi(e)\in U$, showing that $a^<\subseteq\pi^{-1}[U]^<$.  Likewise, in the latter case $b^<\subseteq\pi^{-1}[U]^<$, which shows that $T$ is an ultrafilter, again by \eqref{Ultra+}.

Next note that, for any $T\in\mathcal{U}(S)$ and $U\in\mathcal{U}(S')$,
\[\underline{\pi}(U)=T\qquad\Leftrightarrow\qquad\pi[T]\subseteq U.\]
Indeed, if $\underline{\pi}(U)=T$ then $T=\pi^{-1}[U]^<$ and hence $\pi[T]=\pi[\pi^{-1}[U]^<]\subseteq U^<\subseteq U$.  Conversely, if $\pi[T]\subseteq U$ then $T=T^<\subseteq\pi^{-1}[U]^<=\underline{\pi}(U)$ and hence $T=\underline{\pi}(U)$, as $T$ and $\underline{\pi}(U)$ are ultrafilters.

Now to see that $\underline{\pi}$ is a functor, say $T=\underline{\pi}(U)$ and $V=\underline{\pi}(W)$, i.e. $\pi[T]\subseteq U$ and $\pi[V]\subseteq W$ and hence $\pi[(TV)^<]\subseteq(UW)^<$.  Thus if $0\notin UW$ then $0\notin TV$, i.e. if $U\cdot W$ is defined then so is $T\cdot V$, in which case $\underline{\pi}(U)\cdot\underline{\pi}(W)=T\cdot V=\underline{\pi}[U\cdot W]$.

To show that $\underline{\pi}$ is star-bijective, take $U\in\mathcal{U}(S')$ and $T\in\mathcal{U}(S)$ with $\underline{\pi}(U)=\mathsf{s}(T)$.  As $\mathsf{s}(T)$ is a unit, it contains some $d\in\Phi[S]$, which means $U$ contains $\pi(d)\in\Phi[S']$ so $U$ is also a unit (see \cite[Proposition 6.2]{Bice2020Rep}).   Now, for any $t\in T$, \cite[Proposition 5.6 and Theorem 10.7]{Bice2020Rep} yield an ultrafilter $V=(\pi(t)U)^<\in\mathcal{U}(S')$ with $\mathsf{s}(V)=U$.  By \cite[Proposition 5.5]{Bice2020Rep}, $(t\mathsf{s}(T))^<=T$ and hence
\[\pi[T]=\pi[(t\mathsf{s}(T))^<]\subseteq(\pi(t)U)^<=V.\]
Thus $T=\underline{\pi}(V)$, showing that $\underline{\pi}$ is star-surjective.  On the other hand, if we had another $W\in\mathcal{U}(S')$ with $\mathsf{s}(W)=U$ and $\underline{\pi}(W)=T$ then, in particular, $\pi(t)\in W$ and hence $W=(\pi(t)\mathsf{s}(W))^<=(\pi(t)U)^<=V$, again by \cite[Proposition 5.5]{Bice2020Rep}.  This shows that $\underline{\pi}$ is also star-injective and hence star-bijective.

Next we claim that, for any $a\in S$,
\begin{equation}\label{piInverse}
\underline{\pi}^{-1}[\mathcal{U}_a]=\bigcup_{b<a}\mathcal{U}_{\pi(b)}.
\end{equation}
Indeed, $U\in\underline{\pi}^{-1}[\mathcal{U}_a]$ means $\underline{\pi}(U)\in\mathcal{U}_a$, i.e. $a\in\underline{\pi}(U)=\pi^{-1}[U]^<$, which is saying that $\pi(b)\in U$, for some $b<a$, i.e. $U\in\mathcal{U}_{\pi(b)}$.  This proves the claim, from which it follows that $\underline{\pi}$ is continuous with open domain
\[\mathrm{dom}(\underline{\pi})=\{U\in\mathcal{U}(S'):\pi^{-1}[U]\neq\emptyset\}=\bigcup_{a\in S^>}\mathcal{U}_{\pi(a)}.\]

Finally, we show $\underline{\pi}$ is proper by proving $\underline{\pi}^{-1}$ preserves compact inclusions (see \cite[Lemma V-5.18]{GierzHofmannKeimelLawsonMisloveScott2003}).  To see this, take open $O,N\subseteq\mathcal{U}(S)$ with $O\Subset N$.  Each $U\in N$ must be contained in a basic open set contained in $N$, i.e. we must have $a\in U$ with $\mathcal{U}_a\subseteq N$.  Then we can take $b\in U$ with $b<a$, which implies $U\in\mathcal{U}_b\Subset\mathcal{U}_a\subseteq N$.  As $O\Subset N$, we can cover $O$ with finitely many such $\mathcal{U}_b$'s, i.e. we have $a_1,\ldots,a_n,b_1,\ldots,b_n\in S$ such that $b_k<a_k$, for all $k\leq n$, and
\[O\subseteq\bigcup_{k\leq n}\mathcal{U}_{b_k}\Subset\bigcup_{k\leq n}\mathcal{U}_{a_k}\subseteq N.\]
For each $k\leq n$, take $c_k$ with $b_k<c_k<a_k$ and hence $\pi(b_k)<\pi(c_k)$.  Then
\[\underline{\pi}^{-1}[O]\subseteq\bigcup_{k\leq n}\mathcal{U}_{\pi(b_k)}\Subset\bigcup_{k\leq n}\mathcal{U}_{\pi(c_k)}\subseteq\underline{\pi}^{-1}[N],\]
by \eqref{USub} and \eqref{piInverse}.  Thus $\underline{\pi}$ is indeed an \'etale morphism from $\mathcal{U}(S')$ to $\mathcal{U}(S)$.
\end{proof}

\begin{rmk}
The additive structure above was only needed for maximality, i.e. to show that $\pi^{-1}[U]^<$ is an ultrafilter rather than just an arbitrary filter.  To see that the additive structure is crucial for this, consider the well-structured C*-algebras $A=S=Z=\mathbb{C}\oplus\mathbb{C}$ and $A'=S'=Z'=\mathbb{C}$ (with identity expectations).  Then the map $\pi(\alpha,\beta)=\alpha\beta$ preserves the *-semigroup structure but not the additive structure.  Moreover, we see that $U=\mathbb{C}\setminus\{0\}$ is an ultrafilter in $S'$ whose premiage $\pi^{-1}[U]=\pi^{-1}[U]^<=\{(\alpha,\beta)\in\mathbb{C}\oplus\mathbb{C}:\alpha\neq0\neq\beta\}$ is a non-maximal filter in $S$.
\end{rmk}

\begin{rmk}
One might be tempted to define $\underline{\pi}$ as just the preimage $\pi^{-1}[U]$ rather than the up-set it generates.  The problem is that $\pi^{-1}[U]$ might not even be round.  For example, consider the well-structured C*-algebras where $A=S=\mathbb{C}\oplus\mathbb{C}$, $Z=\mathbb{C}(1,1)$ and $A'=S'=Z'=\mathbb{C}\oplus\mathbb{C}$ (again with identity expectations).  Then the identity map $\pi$ on $\mathbb{C}\oplus\mathbb{C}$ is a structure-preserving morphism from $A$ to $A'$ and $U=\{(\alpha,\beta):\alpha\neq0\}$ is an ultrafilter in $S'$.  However, $U=\pi^{-1}[U]$ is not an ultrafilter in $S$, e.g. $(1,0)\in U$ but the only $s\in S$ with $s(1,0)\in Z$ is $s=(0,0)$, which is not in $U$, and hence $(1,0)$ does not $Z$-dominate anything in $U$.
\end{rmk}

For functoriality, we must also observe that $\pi\mapsto\underline\pi$ respects composition.

\begin{prp}
If $(A,S,Z,\Phi)$, $(A',S',Z',\Phi')$ and $(A'',S'',Z'',\Phi'')$ are well-structured C*-algebras and $\pi:A\rightarrow A'$ and $\pi':A'\rightarrow A''$ are structure preserving,
\begin{equation}\label{UnderComp}
\underline{\pi'\circ\pi}=\underline\pi\circ\underline\pi'.
\end{equation}
\end{prp}

\begin{proof}
For any $U\in\mathcal{U}(S'')$, it follows from \eqref{Interpolation} that
\[\underline{\pi'\circ\pi}(U)=(\pi'\circ\pi)^{-1}[U]^<=\pi^{-1}[\pi'^{-1}[U]]^{<<}\subseteq\pi^{-1}[\pi'^{-1}[U]^<]^<=\underline\pi\circ\underline\pi'(U).\]
But inclusion on ultrafilters implies equality so $\underline{\pi'\circ\pi}(U)=\underline\pi\circ\underline\pi'(U)$.
\end{proof}

\section{Weyl Seminorms}\label{WeylSeminorms}

Let us reiterate our standing assumption.

\begin{center}
\textbf{$\langle A\rangle=(A,S,Z,\Phi)$ is a well-structured C*-algebra.}
\end{center}
For any $a\in A$ and $U\subseteq S$, we define
\[\|a\|_U=\inf\{\|\Phi(as)b\|:U\ni u<_sb\}.\]
When $U$ is down-directed, $\|\cdot\|_U$ is a seminorm, by \cite[Proposition 3.4]{Bice2020Rings}.  When $U^<\in\mathcal{U}(S)$, we call $\|\cdot\|_U$ a \emph{Weyl seminorm}.  For corical Fell bundles, Weyl seminorms correspond to norms of sections at particular points of the base groupoid.

\begin{prp}
If $\langle A\rangle=\langle\rho\rangle_\mathsf{r}$, for a corical Fell bundle $\rho:B\twoheadrightarrow\Gamma$, then
\begin{equation}\label{NormSgamma}
\|a\|_{S_\gamma}=\|a(\gamma)\|,
\end{equation}
for all $a\in A$ and $\gamma\in\Gamma$, where $S_\gamma=\{s\in S:\gamma\in s^{-1}[B^\times]\}$ as in \eqref{gamma->Sgamma}.
\end{prp}

\begin{proof}
Take $\gamma\in\Gamma$ and $u,s,b\in S$ with $S_\gamma\ni u<_sb$.  In particular, note that $0\neq u(\gamma)=u(\gamma)s(\gamma^{-1})b(\gamma)$ and hence $s(\gamma^{-1})b(\gamma)=1_{\mathsf{s}(\gamma)}$, as $sb\in Z$, so
\[\|\Phi(as)b\|\geq\|(\Phi(as)b)(\gamma)\|=\|(as)(\mathsf{r}(\gamma))b(\gamma)\|=\|a(\gamma)s(\gamma^{-1})b(\gamma)\|=\|a(\gamma)\|.\]
Thus $\|a(\gamma)\|\leq\inf\{\|\Phi(as)b\|:S_\gamma\ni u<_sb\}=\|a\|_{S_\gamma}$.

Conversely, for every $\delta>0$, we have an open neighbourhood $O$ of $\gamma$ such that $\|a(\alpha)\|<\|a(\gamma)\|+\delta$, for all $\alpha\in O$, as the norm is upper semicontinuous on $B$.  We then have $u,s,b\in S_\gamma$ with $u<_sb$, $sb\in Z^1$ and $\mathrm{supp}(b)\subseteq O$.  For any $\alpha\in O$,
\[\|(\Phi(as)b)(\alpha)\|=\|(as)(\mathsf{r}(\alpha))b(\alpha)\|=\|a(\alpha)s(\alpha^{-1})b(\alpha)\|\leq\|a(\alpha)\|<\|a(\gamma)\|+\delta,\]
and hence $\|\Phi(as)b\|\leq\|a(\gamma)\|+\delta$.  On the other hand, for any $\alpha\in\Gamma\setminus O$, we immediately see that $\|(\Phi(as)b)(\alpha)\|=\|(as)(\mathsf{r}(\alpha))b(\alpha)\|=0$.  As $\Phi(as)b$ is supported on a slice, this yields $\|\Phi(as)b\|\leq\|a(\gamma)\|+\delta$.  This shows that $\|a\|_{S_\gamma}\leq\|a(\gamma)\|$.
\end{proof}

With the above situation in mind, let us return to the general situation where $\langle A\rangle=(A,S,Z,\Phi)$ is just a well-structured C*-algebra.  First it will be convenient to show that $\|\cdot\|_U$ can be characterised in a number of similar ways, even for general (i.e. non-ultrafilter) $U\subseteq S$.  For example, $\|a\|_U$ could have be defined dually as
\[\|a\|_U=\inf\{\|b\Phi(sa)\|:U\ni u<_sb\},\]
thanks to the shiftability of $\Phi$ (see \cite[Proposition 3.3]{Bice2020Rings}).  As before, it would also make no difference if $<$ were defined with $Z^1_+$ rather than $Z$.  Consequently, all the assumptions in \cite[\S6]{Bice2020Rings} will apply and we can use all the resulting theory, where we take the canonical metric $\|a-b\|$ for the metric $\rho$ in \cite{Bice2020Rings}.

\begin{prp}
For all $U\subseteq S$ and $a\in A$,
\begin{equation}\label{||||U1}
\|a\|_U=\inf\{\|\Phi(as)b\|:U\ni u<_sb\text{ and }sb,bs\in Z^1_+\}.
\end{equation}
\end{prp}

\begin{proof}
Denote the infimum above by $\|a\|_U'$.  As this infimum is taken over a smaller set, certainly $\|a\|_U\leq\|a\|_U'$.  Conversely, if $U\ni u<_sb$ then, by \autoref{ZDomination}, we have $s'\in S$ with $bs',s'b\in Z^1_+$ and $u<_{s'}b$.  Specifically, as in the proof, we can take $s'=ss^*b^*g(bss^*b^*)$ and then
\begin{align*}
\|\Phi(as')u\|&=\|\Phi(ass^*b^*g(bss^*b^*))b\|\\
&=\|\Phi(as)s^*b^*g(bss^*b^*)b\|\\
&=\|s^*b^*g(bss^*b^*)\Phi(as)b\|\\
&\leq\|s^*b^*g(bss^*b^*)\|\|\Phi(as)b\|.
\end{align*}
As $\|s^*b^*g(bss^*b^*)\|^2=\|g(bss^*b^*)bss^*b^*g(bss^*b^*)\|\leq\sup_{x\in\mathbb{R}_+}(x\vee x^{-1})^2x\leq1$, this shows that $\|a\|'_U\leq\|a\|_U$ and hence $\|a\|'_U=\|a\|_U$.
\end{proof}

It would also make no difference if $U$ were replaced with $U^<$.

\begin{prp}
For all $U\subseteq S$ and $a\in A$,
\begin{equation}\label{||||U<}
\|a\|_U=\|a\|_{U^<}.
\end{equation}
\end{prp}

\begin{proof}
If $U\ni u<_sb$ and $bs,sb\in Z^1_+$ then \eqref{asbInt} implies $u<_sbg_n(sb)<_{sh_n(bs)}b$, for any $n\in\mathbb{N}$, where $g_n(x)=(nx-n+1)\vee0$ and $h_n(x)=(\frac{n}{n-1})^2x\wedge x^{-1}$ on $\mathbb{R}_+$.
Note that $\|\Phi(ash_n(bs))b\|=\|\Phi(as)h_n(bs)b\|=\|h(bs)\Phi(as)b\|\leq\|h_n(bs)\|\|\Phi(as)b\|$.  As $\|h_n(bs)\|\leq\sup_{x\in\mathbb{R}_+}g_n(x)\leq\frac{n}{n-1}\rightarrow1$, as $n\rightarrow\infty$, this and \eqref{||||U1} shows that $\|a\|_{U^<}\leq\|a\|_U$.  The reverse inequality is follows from the definition of $\|\cdot\|_U$.
\end{proof}

As long as $U\neq\emptyset$, \eqref{||||U1} and the following implies that $\|a\|_U\leq\|a\|$.

\begin{prp}
For all $a,s,b\in A$ with $\|sb\|\leq1$,
\begin{equation}\label{asbleqa}
\|\Phi(as)b\|\leq\|a\|.
\end{equation}
\end{prp}

\begin{proof}
To prove \eqref{asbleqa} just note that
\begin{align*}
\|\Phi(as)b\|^2&=\|\Phi(as)bb^*\Phi(s^*a^*)\|\\
&=\|\Phi(\Phi(as)bb^*s^*a^*)\|\\
&\leq\|\Phi(as)bb^*s^*a^*\|\\
&\leq\|\Phi(as)b\|\|b^*s^*a^*\|\\
&\leq\|\Phi(as)b\|\|a\|.\qedhere
\end{align*}
\end{proof}

If we restrict to witnesses in the unit ball, we can simplify the definition of $\|\cdot\|_U$.

\begin{prp}
For all $U\subseteq S$ and $a\in A$,
\begin{align}
\label{||||US1}\|a\|_U&=\inf\{\|\Phi(as)\|:U\ni u<_sb\in S^1\}\\
\label{||||US1*}&=\inf\{\|\Phi(as)\|:U\ni u<_ss^*\in S^1\}.
\end{align}
\end{prp}

\begin{proof}
If $U\ni u<_sb\in S^1$ then $\|\Phi(as)b\|\leq\|\Phi(as)\|$ and hence
\[\|a\|_U\leq\inf\{\|\Phi(as)\|:U\ni u<_sb\in S^1\}\leq\inf\{\|\Phi(as)\|:U\ni u<_ss^*\in S^1\}.\]
Conversely, say $U^<\ni u<_tc$.  By \eqref{1Interpolation}, we have $s,v\in S^1$ such that $U\ni v<_ss^*<u<_tc$.  Then $s^*<_tc$ so \eqref{a*<_bs} yields $s<_ct$ and hence
\[\|\Phi(as)\|=\|\Phi(atcs)\|=\|\Phi(at)cs\|\leq\|\Phi(at)c\|.\]
This shows that $\inf\{\|\Phi(as)\|:U\ni u<_ss^*\in S^1\}\leq\|a\|_{U^<}=\|a\|_U$, by \eqref{||||U<}.
\end{proof}

For any $a\in A$ and $U\subseteq S$, let
\[a_U=\{b\in A:\|a-b\|_U=0\}\]

\begin{prp}
For any $a\in A$ and $U\subseteq S$, $a_U$ is closed.
\end{prp}

\begin{proof}
First note that
\[a_U=a+0_U.\]
Thus it suffices to show that $0_U$ is closed.  Accordingly, take $(a_n)\subseteq0_U$ with $a_n\rightarrow a$.  For all $n\in\mathbb{N}$, if $U\ni u<_tc$ and $\|tc\|\leq1$ then \eqref{asbleqa} yields
\[\|\Phi(at)c\|\leq\|\Phi((a-a_n)t)c\|+\|\Phi(a_nt)c\|\leq\|a-a_n\|+\|\Phi(a_nt)c\|.\]
Taking infima yields $\|a\|_U\leq\|a-a_n\|\rightarrow0$ and hence $a\in0_U$.
\end{proof}

When $U$ is down-directed, $\|\cdot\|_U$ is a seminorm and, in particular, $0_U=0_U+0_U$.  By the above result, $0_U$ is then a closed subspace of $A$ which yields the quotient space $A/0_U$.  Each $a_U=a+0_U$ will then be an element of this quotient space whose quotient norm is precisely the Weyl seminorm of $a$ at $U$, as the following shows.

\begin{prp}
For all $a\in A$ and $U\subseteq S$,
\begin{equation}\label{||a||U}
\|a\|_U=\inf_{b\in a_U}\|b\|.
\end{equation}
\end{prp}

\begin{proof}
If $b\in a_U$, $U\ni u<_tc$ and $\|tc\|\leq1$ then \eqref{asbleqa} yields
\[\|\Phi(at)c\|\leq\|\Phi((a-b)t)c\|+\|\Phi(bt)c\|\leq\|\Phi((a-b)t)c\|+\|b\|.\]
Taking infima yields $\|a\|_U\leq\|b\|$ so $\|a\|_U\leq\inf_{b\in a_U}\|b\|$.  Conversely, if $U\ni u<_tc$ then $\Phi(at)c\in a_U$, by \cite[Proposition 3.8]{Bice2020Rings}, so $\inf_{b\in a_U}\|b\|\leq\|a\|_U$.
\end{proof}

We will need the following observation to define the Weyl bundle product.

\begin{prp}
For all $a\in A$ and $U\subseteq S$ with $\emptyset\neq U\subseteq U^<$,
\begin{equation}\label{aUU>}
a_U\cap U^>\neq\emptyset.
\end{equation}
\end{prp}

\begin{proof}
For any $b,c,s,u\in S$ with $U\ni u<_sb<c$, \cite[Proposition 3.8]{Bice2020Rings} yields
\begin{equation}\label{asbinaU}
\Phi(as)b\in a_U.
\end{equation}
Moreover, $\Phi(as)=\lim_n\Phi(asf_n(s^*s))=\lim_n\Phi(as)f_n(s^*s)\in\Phi[A]\Phi[S]\subseteq\Phi[S]$, as $\Phi[S]$ is an ideal of $\Phi[A]$, and hence $\Phi(as)b<c$, by \cite[Proposition 3.8]{Bice2020Rep}.
\end{proof}

Let us introduce the following notation for $T\subseteq S$.
\begin{align*}
T^Z&=\{z\in Z:\exists t\in T\ (tz=t)\}.\\
{}^ZT&=\{z\in Z:\exists t\in T\ (zt=t)\}.
\end{align*}
Finally, we show that compatible sums always attain their norm at some ultrafilter.

\begin{thm}\label{NormAttaining}
For any $a\in S^\sim_\Sigma$, we have $U\in\mathcal{U}(S)$ such that
\[\|a\|=\|a\|_U.\]
\end{thm}

\begin{proof}
Let
\[Y=\{y\in Z^1_+:\exists z\in Z^1_+\ (y\ll z\text{ and }\|az\|<\|a\|)\}.\]
We first claim $Y$ is $\ll$-round, i.e. for any $y\in Y$, we have $y'\in Y$ with $y\ll y'$.  Indeed, if $y\in Y$ then we have $z\in Z^1_+$ with $y\ll z\text{ and }\|az\|<\|a\|$.  Letting $g_n(x)=\frac{n+1}{n}x\wedge1$ on $\mathbb{R}_+$, we see that $\|g_n(z)-z\|\leq\frac{1}{n+1}$ so $\|ag_n(z)\|<\|a\|$, for some $n\in\mathbb{N}$.  Letting $h(x)=((n+1)x-n)\vee0$ on $\mathbb{R}_+$, we see that $y\ll h(z)$ (as $y\ll z$ and $h(1)=1$) and $h(z)\ll g_n(z)$ so $h(z)\in Y$, proving the claim.

As any $y,z\in Z$ commutes with $a^*a\in\Phi[S]$, by \eqref{SumStarSquares},
\begin{equation}\label{veeNorm}
\|a(y\vee z)\|^2=\|(y\vee z)a^*a(y\vee z)\|=\|(ya^*ay)\vee(za^*az)\|=\|ay\|^2\vee\|az\|^2.
\end{equation}
If $y,z\in Y$, then we have $y',z'\in Z^1_+$ with $y\ll y'$, $z\ll z'$ and $\|ay'\|,\|az'\|<\|a\|$.  By \eqref{veeNorm}, it then follows that $\|a(y'\vee z')\|<\|a\|$ so $y\vee z\in Y$, as $y\vee z\ll y'\vee z'$.  By the claim in the previous paragraph, we have $x\in Y$ with $y\vee z\ll x$ or, equivalently, $y\ll x$ and $z\ll x$.  So $Y$ is up-directed w.r.t. the relation $y=yz$.

Now take pairwise compatible $a_1,\ldots,a_n\in S$ with $a=\sum_{k=1}^na_k$.  By \eqref{SimultaneousCompatibility} and \eqref{1Interpolation}, we have $b_k,c_k\in S^>$ with $a_k<c_k<_{b_k^*}b_k$ and $a_j^*b_k,a_jb_k^*\in\Phi[S]$, for all $j,k\leq n$.  Then \eqref{1Interpolation} yields $(b_k^m)_{k\leq n}^{m\in\mathbb{N}}$ with $b_k^1=b_k$ and
\[a_k<b_k^{m+1}<_{b_k^{m*}}b_k^m,\]
for all $k\leq n$ and $m\in\mathbb{N}$.  For each $k\leq n$, we then have a down-directed subset
\[T_k=\{b_k^m-b_k^my:m\in\mathbb{N}\text{ and }y\in Y\}.\]
Indeed, for any $l,m\in\mathbb{N}$ and $y,z\in Y$, we can take $j>l,m$ and $x\in Y$ with $y\ll x$ and $z\ll x$ and then \eqref{SubLem} yields $b_k^j-b_k^jx<b_k^l-b_k^ly$ and $b_k^j-b_k^jx<b_k^m-b_k^mz$.  If we had $0\in\bigcap_{k\leq n}T_k$ then we would have $m\in\mathbb{N}$ and $y\in Y$ such that $b_k^m=b_k^my$, for all $k\leq n$, and hence
\[a=a(\textstyle\bigvee_{k\leq n}b_k^{m*}b_k^m)=a(\textstyle\bigvee_{k\leq n}b_k^{m*}b_k^m)y=ay.\]
But then $\|a\|=\|ay\|<\|a\|$, a contradiction.  Thus we have some $k\leq n$ with $0\notin T_k$ and then Kuratowski-Zorn yields $U\in\mathcal{U}(S)$ with $T_k\subseteq U$.

By \cite[Proposition 3.10]{Bice2020Rings}, $\|a\|_U=\|ab_k^*b_k\|_U$, as $b_k^*b_k\in U^Z$.  As $a_jb_k^*\in\Phi[S]$, for all $j\leq n$, $ab_k^*\in\Phi[S]$ and hence $ab_k^*b_k\in U^>$, by \cite[Proposition 3.8]{Bice2020Rep}, as $b_k=b_k-b_k0\in U\cap S^>$.  If we had $\|a\|_U<\|a\|$ then we would have $z\in U^Z\cap Z^1_+$ with $\|ab_k^*b_kz\|<\|a\|$, by \cite[Proposition 3.11]{Bice2020Rings}.  Then we would have $x,y\in U^Z\cap Z^1_+$ with $x\ll y\ll b_k^*b_kz$ and $\|ay\|=\|ab_k^*b_kzy\|\leq\|ab_k^*b_kz\|<\|a\|$ so $y\in Y$.  But then $0=b_kx-b_kx=(b_k-b_ky)x\in T_kU^Z\subseteq U$, a contradiction.  Thus $\|a\|_U<\|a\|$.
\end{proof}

\section{The Weyl Bundle}\label{TheWeylBundle}
Again let us reiterate our standing assumption.
\begin{center}
\textbf{$(A,S,Z,\Phi)$ is a well-structured C*-algebra.}
\end{center}
For any $a\in A$ and $U\subseteq S$, let
\[[a,U]=(a_U,U^<).\]
If $U^<$ is ultrafilter, we call $[a,U]$ a \emph{Weyl pair}.  We denote these Weyl pairs by
\[\mathcal{W}(A)=\{[a,U]:a\in A\text{ and }U\in\mathcal{U}(S)\}.\]
The topology on $\mathcal{W}(A)$ is generated by $(a^\delta_s)$, for $a\in A$, $s\in S$ and $\delta>0$, where
\[a^\delta_s=\{[b,U]:b\in A,\,U\in\mathcal{U}_s\text{ and }\|a-b\|_U<\delta\}.\]

\begin{prp}\label{WABanach}
$\mathsf{p}:\mathcal{W}(A)\rightarrow\mathcal{U}(S)$ is a Banach bundle where
\begin{align*}
[a,U]+[b,U]&=[a+b,U]\\
\lambda[a,U]&=[\lambda a,U]\\
\|[a,U]\|&=\|a\|_U.\\
\mathsf{p}([a,U])&=U^<.
\end{align*}
\end{prp}

\begin{proof}
By \cite[Theorem 4.4]{Bice2020Rings}, $\mathsf{p}$ is open and continuous.  We also immediately verify that the given operations are well-defined and turn each fibre $\mathsf{p}^{-1}\{U\}\approx A/0_U$ into a Banach space.  By \cite[Proposition 6.2]{Bice2020Rings}, addition is continuous.  For all $a\in A$, $s\in S$, $\delta>0$ and $\lambda\in\mathbb{C}$, we see that $\lambda(a^\delta_s)\subseteq(\lambda a)^{\lambda\delta}_s$ so $x\mapsto\lambda x$ is also continuous.  By \cite[Proposition 4.5]{Bice2020Rings}, the norm $x\mapsto\|x\|_U$ is upper semicontinuous.  It only remains to verify the last defining property of a Banach bundle.  Accordingly, take nets $(U_\lambda)_{\lambda\in\Lambda}\subseteq\mathcal{U}(S)$ and $(a_\lambda)_{\lambda\in\Lambda}\subseteq A$ with $U_\lambda\rightarrow U$ in $\mathcal{U}(S)$ and $\|a_\lambda\|_{U_\lambda}\rightarrow 0$.  Given any $u\in U$ and $\delta>0$, this means that $[a_\lambda,U_\lambda]\in0^\delta_u$, for all sufficiently large $\lambda$, and hence $[a_\lambda,U_\lambda]\rightarrow[0,U]$, as required.
\end{proof}

Next we consider the natural product and involutive structure on $\mathcal{W}(A)$.

\begin{prp}\label{WA*category}
$\mathcal{W}(A)$ is a topological *-category where $[a,U]^*=[a^*,U^*]$ and
\begin{equation}\label{WAproducts}
[a,U][b,V]=[ab,UV]
\end{equation}
when $a\in U^>$, $b\in V^>$ and $0\notin UV$.  Units and their scalar multiples are given by
\begin{align}
\label{Wunits}\mathcal{W}^0&=\{[z,U]:U\in\mathcal{U}^0\text{ and }z\in U^Z={}^ZU\}.\\
\label{CWunits}\mathbb{C}\mathcal{W}^0&=\{[z,U]:U\in\mathcal{U}^0\text{ and }z\in Z\}.
\end{align}
\end{prp}

\begin{proof}
By \cite[Proposition 4.1 and Theorem 4.4]{Bice2020Rings}, $\mathcal{W}(A)$ is a topological category whose units $\mathcal{W}^0$ are characterised as above in \eqref{Wunits}.  Also, $[a,U]^*=[a^*,U^*]$ is immediately seen to be an involution, which is continuous because $(a^\delta_s)^*=(a^*)_{s^*}^\delta$, for all $a\in A$, $s\in S$ and $\delta>0$.

It follows that any scalar multiple of a unit is of the form $[\lambda z,U]$, for $\lambda\in\mathbb{C}$ and $z\in U^Z$.  It only remains to show that, conversely, $[z,U]\in\mathbb{C}\mathcal{W}^0$ whenever $z\in Z$.  To see this, note first that \autoref{ZWitness} implies unit ultrafilters correspond to ultrafilters in $Z$, i.e. the map $U\mapsto U\cap Z$ is a bijection from $\mathcal{U}^0$ to $\mathcal{U}(Z)$ (with inverse $U\mapsto U^<)$.  By Gelfand, we can identify $Z$ with $C_0(X)$, for some space $X$.  Then \cite[Theorem 5.3]{BiceClark2020} tells us that every ultrafilter in $Z$ is of the form
\[Z_x=\{z\in Z:x\in\mathrm{supp}(z)\},\]
for some $x\in X$.  Now take $U\in\mathcal{U}^0$, so have $x\in X$ with $U\cap Z=Z_x$.  For any $z\in Z$ and $\varepsilon>0$, we have a neighbourhood $N$ of $x$ on which $z$ is $\varepsilon$-close to $z(x)$.  By Urysohn's lemma, we have $y,y'\in Z^1_+$ such that $\mathrm{supp}(y')\subseteq y^{-1}\{1\}\cap N$ and $y'$ is $1$ on some neighbourhood of $x$.  This means $y'y=y'\in U^Z$ and hence
\[\|z-z(x)y\|_U\leq\|(z-z(x)y)y'\|=\|(z-z(x))y'\|\leq\sup_{x'\in N}|z(x')-z(x)|\leq\varepsilon.\]
Thus $\|[z,U]-z(x)1_U\|=\|[z,U]-[z(x)y,U]\|<\varepsilon$ (noting that $y$ is also $1$ on a neighbourhood of $x$ so $y\in U^Z$ and hence $1_U=[y,U]$).  As $\varepsilon>0$ was arbitrary, it follows that $[z,U]=z(x)1_U\in\mathbb{C}\mathcal{W}^0$, as required.
\end{proof}

Note that \eqref{WAproducts} is not valid for arbitrary $a$ and $b$.  But it is valid more generally if at least one of $a$ or $b$ is small enough, i.e. for any $U,V\in\mathcal{U}(S)$ with $0\notin UV$,
\[a\in U^>\text{ or }b\in V^>\qquad\Rightarrow\qquad[a,U][b,V]=[ab,UV],\]
thanks to \cite[Corollary 3.7]{Bice2020Rings}.  In other words, to calculate $[a,U][b,V]$ via \eqref{WAproducts}, we must first replace $a$ or $b$ with a smaller representative of $a_U$ or $b_V$, i.e. an element of $a_U\cap U^>$ or $b_V\cap V^>$.  This is always possible, by \eqref{aUU>}.

With unit ultrafilters, we can instead use the expectation $\Phi$.

\begin{prp}
For any $a,b\in A$ and $U\in\mathcal{U}(S)$,
\begin{equation}
[a,U][b,\mathsf{s}(U)]=[a\Phi(b),U]\qquad\text{and}\qquad[a,\mathsf{r}(U)][b,U]=[\Phi(a)b,U].
\end{equation}
\end{prp}

\begin{proof}
If $z\in U^Z\cap U^>$ then $\Phi(b)z\in b_{\mathsf{s}(U)}\cap U^>$, by \cite[Proposition 3.13]{Bice2020Rings}, so
\[[a,U][b,\mathsf{s}(U)]=[a,U][\Phi(b)z,\mathsf{s}(U)]=[a\Phi(b)z,U]=[a\Phi(b),U].\]
This proves the first equation and the second follows dually.
\end{proof}

Next we characterise the core of $\mathcal{W}(A)$ as in \cite[Proposition 4.2]{Bice2020Rings}.  The crucial difference is that our metric $\|a-b\|$ is not discrete, and we must instead appeal to another continuous functional calculus argument.

\begin{thm}\label{Corical}
The core of $\mathcal{W}(A)$ is an open topological groupoid given by
\begin{equation}\label{Core}
\mathcal{W}(A)^\times=\{[a,U]:a\in U\in\mathcal{U}(S)\}
\end{equation}
\end{thm}

\begin{proof}
By \cite[Proposition 4.1]{Bice2020Rings}, $[a,U]\in\mathcal{W}(A)$ is invertible when $a\in U\in\mathcal{U}(S)$.  Conversely, we must show that all invertibles are of this form.

Accordingly, take $[a,U]\in\mathcal{W}(A)^\times$, so we have $[b,U^*]\in\mathcal{W}(A)$ such that $[a,U][b,U^*]$ and $[b,U^*][a,U]$ are units.  By \cite[Proposition 3.8]{Bice2020Rings}, we can replace $a$ and $b$ with $\Phi(at')t$ and $t'\Phi(tb)$ respectively, where $U\ni s<_{t'}t<_{u'}u$, so that $a\in U^>$, $b\in U^{*>}$ and $ab,ba\in\Phi[S]$.  Further replacing $b$ with $a^*b^*b$ if necessary, we may assume that $ba\in\Phi[S]_+$.  Now $[b,U^*][a,U]=[ba,\mathsf{s}(U)]$ is a unit so we must have $z\in U^Z$ with $\|ba-z\|_{\mathsf{s}(U)}=0$.
By \cite[Proposition 3.13]{Bice2020Rings}, we then have $z'\in U^Z$ with $\|baz'-zz'\|<1/2$.  Take $z''\in{}^ZU$ with $z''z'z=z''$ and let $g(x)=4x\wedge x^{-1}$ on $\mathbb{R}_+$ so $z''=bag(ba)z''=z''g(ba)ba$.  Taking $w\in U$ with $b<w^*$, we may take yet another $u\in U$ with $u<w$ and $u=uz''$.  Then $ub<ww^*\in\Phi[S]$ so $ub\in\Phi[S]$.  Letting $s=z''g(ba)b$, we see that $us<ww^*\in\Phi[S]$ so
\[us\in\Phi[S],\quad usa=uz''g(ba)ba=uz''=u\quad\text{and}\quad sa=z''g(ba)ba=z''\in Z.\]

Above we can replace $b$ with $bb^*a^*$ instead so that $ab\in\Phi[S]_+$.  Then a dual argument yields $t\in S$ and $v\in U$ with $tv\in\Phi[S]$, $atv=v$ and $at\in Z$.  Taking $w\in U$ with $w<u,v$, we then see that $ws,tw\in\Phi[S]$ and $wsa=w=atw$ and hence $w<a$, by \eqref{st<}.  Thus $a\in w^<\subseteq U$, as required, which proves \eqref{Core}.

To see that $\mathcal{W}(A)^\times$ is open and a topological groupoid, take $[a,U]\in\mathcal{W}(A)^\times$.  As we just showed, we can assume that $a\in U^>$ and we have $b,u\in S$ with $U\ni u<_ba$ and $ab,ba\in Z^1_+$, so $[a,U]^{-1}=[b,U^*]$.  Take any neighbourhood of $[b,U^*]$ which, by \cite[Proposition 4.3]{Bice2020Rings}, we can assume is of the form $b^\varepsilon_{v^*}$, for some $v\in U$ and $\varepsilon>0$.  Taking $w\in U$ with $w<u,v$, we claim that $a^\delta_w\subseteq\mathcal{W}(A)^\times$ and $(a^\delta_w)^{-1}\subseteq b^\varepsilon_{v^*}$, for sufficiently small $\delta>0$.  More precisely, let $\delta'=\delta\|b\|+\delta\|b\|^2(\|a\|+\delta)$ and choose $\delta>0$ small enough that $\delta'<1$ and
\[((1-\delta')^{-2}+2\delta')(\|b\|+\delta\|b\|^2)-\|b\|<\varepsilon.\]

To prove the claim, take $[a',U']\in a_w^\delta$.  In particular, $w\in U'$ and hence $a\in U'^>$, as $w<u<a\in S^>$.  Replacing $a'$ if necessary, we may also assume that $a'\in U'^>$.  By \cite[Proposition 3.11]{Bice2020Rings}, it follows that $\|a-a'\|_{U'}=\inf_{z\in U^Z}\|az-a'z\|$ so we may take $z\in U^Z$ with $\|az-a'z\|<\delta$.  Making $z$ smaller if necessary, we may assume that $baz=z\in Z^1_+$ so $\|z-ba'z\|=\|baz-ba'z\|\leq\delta\|b\|$ and hence
\begin{align*}
\|z^2-a'^*b^*ba'z^2\|&=\|z^2-za'^*b^*ba'z\|\\
&\leq\|z^2-zba'z\|+\|zba'z-za'^*b^*ba'z\|\\
&\leq\|z-zba'\|\|z\|+\|z-za'^*b^*\|\|ba'z\|\\
&\leq\delta\|b\|+\delta\|b\|^2\|a'z\|\\
&\leq\delta\|b\|+\delta\|b\|^2(\|a\|+\delta)\\
&=\delta'
\end{align*}
Let $b'=a'^*b^*b$ so that $b'a'\in\Phi[A]_+$, $\|z^2-b'a'z^2\|\leq\delta'$ and
\[\|zb-zb'\|=\|zb-za'^*b^*b\|\leq\|z-za'^*b^*\|\|b\|\leq\delta\|b\|^2.\]
As above, we can take $z'\in U^Z$ with $z'z=z'$ and $g(x)=(1-\delta')^{-2}x\wedge x^{-1}$ on $\mathbb{R}_+$ so $z'=b''a'z'$, where $b''=g(b'a')b'$.  Thus $[b'',U'^*]$ is a left inverse of $[a',U']$.  A dual argument shows that $[a',U']$ has a right inverse which must then be $[b'',U'^*]$ as well, i.e. $[a',U']^{-1}=[b'',U'^*]$.  Also
\begin{align*}
\|z'b''-z'b'\|&\leq(\|z'g(b'a')-z'b'a'\|+\|z'b'a'-z'\|)\|zb'\|\\
&\leq(\sup_{x\in[0,1]}(g(x)-x)+\|z^2b'a'-z^2\|)(\|zb\|+\delta\|b\|^2)\\
&\leq((1-\delta')^{-2}-(1-\delta')+\delta')(\|b\|+\delta\|b\|^2).
\end{align*}
By our choice of $\delta$, it follows that
\begin{align*}
\|b-b''\|_{U'^*}&\leq\|b-b'\|_{U'^*}+\|b'-b''\|_{U'^*}\\
&\leq\|zb-zb'\|+\|z'b''-z'b'\|\\
&\leq\delta\|b\|^2+((1-\delta')^{-2}-1+2\delta')(\|b\|+\delta\|b\|^2)\\
&\leq\varepsilon
\end{align*}
This shows that $[a',U']^{-1}=[b'',U'^*]\in b^\varepsilon_{v^*}$, proving the claim.  Thus $\mathcal{W}(A)^\times$ is indeed open and the inverse map is continuous on $\mathcal{W}(A)^\times$.  As we already know that the product is continuous, $\mathcal{W}(A)^\times$ is therefore a topological groupoid.
\end{proof}

Putting the above results together, we have the following.

\begin{thm}\label{WeylCoricalFell}
$\mathsf{p}:\mathcal{W}(A)\rightarrow\mathcal{U}(S)$ is a corical Fell bundle.
\end{thm}

\begin{proof}
By \autoref{LCHultrafilters}, \autoref{WA*category} and \autoref{WABanach}, the ultrafilters $\mathcal{U}(S)$ form a locally compact Hausdorff \'etale groupoid and the Weyl pairs $\mathcal{W}(A)$ form both a topological *-semigroupoid and a Banach bundle.  Bilinearity of the product and antilinearity of the involution on $\mathcal{W}(A)$ follow immediately from the same properties of $A$.  From the definitions we also immediately see that $\mathsf{p}$ is a *-isocofibration.  Submultiplicativity of the norm on $\mathcal{W}(A)$ follows from \cite[Proposition 3.12]{Bice2020Rings} (with $a=c=0$).

Now take any $[a,U]\in\mathcal{W}(A)$ with $a\in U^>$.  For any $z\in U^Z$, we have $y\in U^{Z1}$ with $yz=y$ and then by \cite[Proposition 3.10]{Bice2020Rings},
\[\|[a,U]\|^2=\|a\|_U^2\leq\|ay\|^2=\|y^*a^*ay\|=\|y^*a^*azy\|\leq\|a^*az\|.\]
Then \cite[Proposition 3.13]{Bice2020Rings} yields $\|[a^*a,U^*U]\|=\inf_{z\in U^Z}\|a^*az\|\geq\|[a,U]\|^2$ and hence $\|[a^*a,U^*U]\|=\|[a,U]\|^2$, by submultiplicativity.  Moreover,
\[\mathsf{p}([\sqrt{a^*a},U^*U])=(U^*U)^<=\mathsf{p}([a^*,U^*][a,U])\]
and $\sqrt{a^*a}\in(U^*U)^>$, as $a\in U^>$, and hence
\[[\sqrt{a^*a},U^*U]^*[\sqrt{a^*a},U^*U]=[a^*a,U^*U]=[a^*,U^*][a,U].\]
Thus $\mathsf{p}$ is a Fell bundle.  By \autoref{Corical}, $\mathsf{p}$ is also corical.
\end{proof}

We call the Fell bundle $\mathsf{p}=\mathsf{p}_{\langle A\rangle}$ above the \emph{Weyl bundle} of $\langle A\rangle=(A,S,Z,\Phi)$.

Before moving on, we show that the Weyl bundle construction is functorial.

\begin{thm}
Say we have structured C*-algebras $(A,S,Z,\Phi)$ and $(A',S',Z',\Phi')$ with corresponding Weyl bundles $\mathsf{p}:\mathcal{W}(A)\rightarrow\mathcal{U}(S)$ and $\mathsf{p}':\mathcal{W}(A')\rightarrow\mathcal{U}(S')$.  From any structure-preserving morphism $\pi:A\rightarrow A'$ we can define a unital bundle morphism $\overline\pi:\underline\pi^\mathsf{p}\mathcal{W}(A)\rightarrow\mathcal{W}(A')$ $($where $\underline\pi(U)=\pi^{-1}[U]^<$ as in \eqref{underlinepi}$)$ by
\[\overline\pi(U,[a,\underline\pi(U)])=[\pi(a),U].\]
\end{thm}

\begin{proof}
Take $a,b\in A$ and $U\in\mathrm{dom}(\underline\pi)$ so $U\in\mathcal{U}(S')$ and $\pi^{-1}[U]\neq\emptyset$.  Note that $\pi^{-1}[U]\ni u<_ss^*\in S^1$ implies $U\ni\pi(u)<_{\pi(s)}\pi(s)^*\in S'^1$ and
\[\|\Phi((\pi(a)-\pi(b))\pi(s))\|=\|\pi(\Phi((a-b)s))\|\leq\|\Phi((a-b)s)\|.\]
Taking infima, it follows from \eqref{||||U<} and \eqref{||||US1*} that
\[\|\pi(a)-\pi(b)\|_U\leq\|a-b\|_{\pi^{-1}[U]}.\]
In particular, $\|a-b\|_{\pi^{-1}[U]}=0$ implies $\|\pi(a)-\pi(b)\|_U=0$, i.e. $[a,\underline\pi(U)]=[b,\underline\pi(U)]$ implies $[\pi(a),U]=[\pi(b),U]$, showing that $\overline\pi$ is well-defined.  The fact that $\overline\pi$ is a fibre-wise linear *-homomorphism then follows immediately from the fact that $\pi$ is C*-algebra homomorphism.

For any $\delta>0$, the above inequality also means that $\|a-b\|_{\pi^{-1}[U]}<\delta$ implies $\|\pi(a)-\pi(b)\|_U<\delta$.  We then see that any $(U,[a,\underline\pi(U)])\in\pi^\mathsf{p}\mathcal{W}(A)$ has a neighbourhood $\mathcal{U}_u\times a^\delta$, where $a^\delta=\bigcup_{s\in S}a^\delta_s$ and $u$ is any element of $U$, such that
\[\overline\pi[\mathcal{U}_u\times a^\delta]\subseteq\pi(a)_u^\delta.\]
By \cite[Proposition 4.3]{Bice2020Rings}, the sets of the form $\pi(a)_u^\delta$ form a neighbourhood base at $[\pi(a),U]$ so this shows that $\overline\pi$ is also continuous and thus a bundle morphism.

To see that $\overline\pi$ is unital, take any $(U,[z,\underline\pi(U)])\in(\underline\pi^\mathsf{p}\mathcal{W}(A))^0$.  Then $U\in\mathcal{U}(S')^0$ so $\underline\pi(U)\in\mathcal{U}(S)^0$ and $[z,\underline\pi(U)]\in\mathcal{W}(A)^0$.  By \eqref{Wunits}, we may assume that $z\in\underline\pi(U)^Z$, i.e. $z\in Z$ and we have $s\in\underline\pi(U)$ with $sz=s$.  Then $\pi(z)\in\pi[Z]\subseteq Z'$, $\pi(s)\in U$ and $\pi(s)\pi(z)=\pi(s)$ so $\pi(z)\in U^{Z'}$ and hence $[\pi(z),U]\in\mathcal{W}(A')^0$, again by \eqref{Wunits}.
\end{proof}

Let $\mathbf{WSC^*}$ be the full subcategory of $\mathbf{SC^*}$ consisting of those triples $(\langle A'\rangle,\pi,\langle A\rangle)$ where $\langle A\rangle$ and $\langle A'\rangle$ are well-structured.  Also recall $\mathbf{CFell}$ is the full subcategory of $\mathbf{1Fell}$ consisting of those quadruples $(\rho',\beta,\phi,\rho)$ where $\rho$ and $\rho'$ are corical.

\begin{thm}\label{WeylBundleFunctoriality}
We have a functor $\mathsf{Sp}:\mathbf{WSC^*}\rightarrow\mathbf{CFell}$ given by
\[\mathsf{Sp}(\langle A'\rangle,\pi,\langle A\rangle)=(\mathsf{p}_{\langle A'\rangle},\overline\pi,\underline\pi,\mathsf{p}_{\langle A\rangle}),\]
\end{thm}

\begin{proof}
We need to show that, whenever $\langle A\rangle=(A,S,Z,\Phi)$, $\langle A'\rangle=(A',S',Z',\Phi')$ and $\langle A''\rangle=(A'',S'',Z'',\Phi'')$ are well-structured C*-algebras with associated Weyl bundles $\mathsf{p}$, $\mathsf{p}'$ and $\mathsf{p}''$, and $\pi:A\rightarrow A'$ and $\pi':A'\rightarrow A''$ are structure preserving,
\[(\mathsf{p}'',\overline{\pi'\circ\pi},\underline{\pi'\circ\pi},\mathsf{p})=\mathsf{Sp}(\langle A''\rangle,\pi'\circ\pi,\langle A\rangle)=(\mathsf{p}'',\overline{\pi'}\bullet\overline\pi,\underline\pi\circ\underline{\pi'},\mathsf{p})\]
We already know that $\underline{\pi'\circ\pi}=\underline\pi\circ\underline{\pi'}$, by \eqref{UnderComp}, so we just need to show that
\[\overline{\pi'\circ\pi}=\overline{\pi'}\bullet\overline\pi.\]
To see this, take any $(U,[a,\underline{\pi'\circ\pi}(U)])\in\underline{\pi'\circ\pi}^\mathsf{p}\mathcal{W}(A)\subseteq\mathcal{U}(S'')\times\mathcal{W}(A)$ and note
\begin{align*}
\overline{\pi'}\bullet\overline\pi(U,[a,\underline{\pi'\circ\pi}(U)])&=\overline{\pi'}\bullet\overline\pi(U,[a,\underline\pi\circ\underline{\pi'}(U)])\\
&=\overline{\pi'}(U,\overline\pi(\underline{\pi'}(U),[a,\underline\pi(\underline{\pi'}(U))]))\\
&=\overline{\pi'}(U,[\pi(a),\underline{\pi'}(U)])\\
&=[\pi'(\pi(a)),U]\\
&=\overline{\pi'\circ\pi}(U,[a,\underline{\pi'\circ\pi}(U)]).
\end{align*}
We also immediately see that $\mathsf{Sp}$ takes units (i.e. identity morphims) to units so $\mathsf{Sp}$ is indeed a functor.
\end{proof}

We call $\mathsf{Sp}$ the \emph{spectral} or \emph{spatialisation} functor.

\section{The Weyl Representation}\label{TheWeylRepresentation}

Yet again let us reiterate our standing assumption.
\begin{center}
\textbf{$\langle A\rangle=(A,S,Z,\Phi)$ is a well-structured with Weyl bundle $\mathsf{p}=\mathsf{p}_{\langle A\rangle}$.}
\end{center}
For any $a\in A$, we define a section $\widehat{a}:\mathcal{U}(S)\rightarrow\mathcal{W}(A)$ of $\mathsf{p}$ by
\[\widehat{a}(U)=[a,U].\]
We call the map $a\mapsto\widehat{a}$ the \emph{Weyl representation}.

Note $\Phi$ defines a corresponding $2$-norm on $A$ given by
\[\|a\|_2=\sqrt{\|\Phi(a^*a)\|}.\]
Let $A_2=(A,\|\cdot\|_2)$ denote $A$ considered as a normed space w.r.t. $\|\cdot\|_2$.

\begin{prp}\label{2Contractive}
The Weyl representation is contraction from $A_2$ to $\mathcal{C}_2(\mathsf{p})$.
\end{prp}

\begin{proof}
We immediately see that $\widehat{a}$ is continuous (see \cite[Proposition 4.6]{Bice2020Rings}), $\widehat{\lambda a}=\lambda\widehat{a}$ and $\widehat{a+b}=\widehat{a}+\widehat{b}$, for all $a,b\in A$ and $\lambda\in\mathbb{C}$.  It only remains to show that $\|\widehat{a}\|_\mathsf{2}\leq\|a\|_2$, for all $a\in A$.  To see this, take any $a\in A$ and $U_1,\ldots,U_n\in\mathcal{U}(S)$ with the same source $U\in\mathcal{U}(S)$.  Then $U_j\cdot U_k^*=(U_jU_k^*)^<$ is unit ultrafilter if and only if $j=k$ so, whenever $j\neq k$, \cite[Corollary 8.4]{Bice2020Rings} yields $u_j^k\in U_j$ and $u^j_k\in U_k$ with $\Phi(u_j^ku^{j*}_k)=0$.  Taking $u_k\in U_k^>$ with $u_k<u^j_k$, whenever $k\neq j\leq n$, it follows that $\Phi(u_j^*u_k)=0$ whenever $j\neq k$.  By \eqref{1Interpolation}, we may further assume $U_k\ni v_k<_{u_k^*}u_k\in S^1$, for some $v_k$.  Then $u_k\Phi(u_k^*a)\in a_{U_k}\cap U_k^>$, for all $k\leq n$, so
\begin{align*}
\Big\|\sum_{k\leq n}\widehat{a}(U_k)^*\widehat{a}(U_k)\Big\|_U&\leq\Big\|\sum_{k\leq n}\Phi(a^*u_k)u_k^*u_k\Phi(u_k^*a)\Big\|\\
&\leq\Big\|\sum_{k\leq n}\Phi(a^*u_k)\Phi(u_k^*a)\Big\|\\
&\leq\|\Phi(a^*a)\|,
\end{align*}
by \eqref{nKadison}.  This shows that $\|\widehat{a}\|_\mathsf{2}\leq\|a\|_2$, as required.
\end{proof}

Now we can do the same for the original norm on $A$ and the $\mathsf{b}$-norm on $\mathcal{C}(\mathsf{p})$.

\begin{prp}\label{bNormContraction}
The Weyl representation is contraction from $A$ to $\mathcal{C}_\mathsf{b}(\mathsf{p})$.
\end{prp}

\begin{proof}
For all $a\in A$, we must show that $\|\widehat{a}\|_\mathsf{b}\leq\|a\|$, i.e. for all $f\in\mathcal{F}(\mathsf{p})$,
\[\|\widehat{a}f\|_\mathsf{2}\leq\|a\|\|f\|_2.\]
It suffices to consider $f$ supported on some distinct $U_1,\ldots,U_n\in\mathcal{U}(S)$ which all have the same source $U\in\mathcal{U}(S)$.  Take such an $f$, so we have $f_1,\ldots,f_n\in S^>$ with $f(U_k)=[f_k,U_k]$, for all $k\leq n$.  Replacing each $f_k$ with $u_k\Phi(u_k^*f_k)$ as above if necessary, we may assume that $\Phi(f_j^*f_k)=0$ whenever $j\neq k$ and hence
\[\Phi(e^*e)=\sum_{k\leq n}f_k^*f_k,\]
where $e=\sum_{k\leq n}f_n$.  Then $\|f\|_2=\inf_{z\in U^Z}\|z\Phi(e^*e)z\|=\inf_{z\in U^Z}\|ez\|_2$ so, for any $\varepsilon>0$, replacing each $f_k$ with $f_kz$ if necessary, we can further assume that
\[\|e\|_2\leq\|f\|_2+\varepsilon.\]

For any $T\in\mathcal{U}(S)$ with $\mathsf{s}(T)=U$, we see that
\[\widehat{a}f(T)=\sum_{k\leq n}\widehat{a}(T\cdot U_k^*)f(U_k)=[\sum_{k\leq n}af_k,T]=[ae,T].\]
As $\|ae\|_2=\sqrt{\|\Phi(e^*a^*ae)\|}\leq\|a\|\sqrt{\|\Phi(e^*e)\|}=\|a\|\|e\|_2$, \autoref{2Contractive} yields
\[\|\widehat{a}f\|_2\leq\|\widehat{ae}\|_2\leq\|ae\|_2\leq\|a\|\|e\|_2\leq\|a\|(\|f\|_2+\varepsilon).\]
As $\varepsilon>0$ was arbitrary, this shows that $\|\widehat{a}f\|_2\leq\|a\|\|f\|_2$, as required.
\end{proof}

We also know that, under the Weyl representation, $\Phi$ corresponds to the canonical expectation on $\mathcal{C}_\mathsf{b}(\mathsf{p})$, thanks to \cite[Proposition 8.6]{Bice2020Rings}, i.e. for all $a\in A$,
\begin{equation}\label{PhiPreserving}
\widehat{\Phi(a)}=\widehat{a}_{\mathcal{U}^0}.
\end{equation}
If we restrict to $C^*(S^>)$ then the Weyl representation also respects products and is thus a C*-algebra homomorphism.  We let $\widehat{B}=\{\widehat{b}:b\in B\}$, for any $B\subseteq A$.

\begin{thm}\label{WeylRep}
The Weyl representation of $C^*(S^>)=\mathrm{cl}(S^>_\Sigma)$ is a C*-algebra homomorphism onto the reduced C*-algebra $\mathcal{C}_\mathsf{r}(\mathsf{p})$.  Moreover,
\[\mathcal{S}_\mathsf{r}(\mathsf{p})=\widehat{\mathrm{cl}(S^\sim_\Sigma)}\qquad\text{and}\qquad\mathcal{Z}_\mathsf{r}(\mathsf{p})=\widehat{Z}.\]
\end{thm}

\begin{proof}
We immediately see that the Weyl representation $a\mapsto\widehat{a}$ preserves sums, scalar products and the involution $^*$ on the entirety of $A$.  By \cite[Theorem 4.12]{Bice2020Rings}, it also preserves the product of any $a\in A$ and $s\in S^>$, i.e.
\[\widehat{\,as\,}=\widehat{a}\widehat{s}\qquad\text{and}\qquad\widehat{\,sa\,}=\widehat{s}\widehat{a}.\]
This extends to $s\in S^>_\Sigma$, by sum preservation, and then to $s\in\mathrm{cl}(S^>_\Sigma)$, as $a\mapsto\widehat{a}$ is also a contraction.  Thus the $a\mapsto\widehat{a}$ is a C*-algebra homomorphism on $C^*(S^>)$.

Next we claim that $\widehat{S^\sim_\Sigma}\subseteq\mathcal{S}_\mathsf{c}(\mathsf{p})$.  To see this, take compatible $a_1,\ldots,a_n\in S$ with $a=\sum_{k=1}^na_k$.  By \autoref{SimultaneousSimInt} and \eqref{Interpolation}, we have $b_k,c_k\in S$ such that $a_k<b_k<c_k\sim c_j$, for all $j,k\leq n$.  By \eqref{USub} and \cite[Proposition 4.9]{Bice2020Rings}, $\mathrm{supp}(\widehat{a_k})\subseteq\mathcal{U}_{b_k}\Subset\mathcal{U}_{c_k}$, for all $k\leq n$.  By \eqref{USlice}, $\bigcup_{k\leq n}\mathcal{U}_{c_k}$ is a slice and
\[\mathrm{supp}(\widehat{a})\subseteq\bigcup_{k\leq n}\mathrm{supp}(\widehat{a_k})\subseteq\bigcup_{k\leq n}\mathcal{U}_{b_k}\Subset\bigcup_{k\leq n}\mathcal{U}_{c_k}.\]
Thus $\mathrm{cl}(\mathrm{supp}(\widehat{a}))$ is also a slice and hence $\widehat{a}\in\mathcal{S}_\mathsf{c}(\mathsf{p})$, proving the claim.

It then follows that $\widehat{\mathrm{cl}(S^\sim_\Sigma)}\subseteq\mathrm{cl}_\mathsf{b}(\widehat{S^\sim_\Sigma})\subseteq\mathrm{cl}_\mathsf{b}(\mathcal{S}_\mathsf{c}(\mathsf{p}))=\mathcal{S}_\mathsf{r}(\mathsf{p})$, as $a\mapsto\widehat{a}$ is a contraction.  Also, sum preservation then yields $\widehat{S^>_\Sigma}=\widehat{S^>}_\Sigma\subseteq\mathcal{S}_\mathsf{c}(\mathsf{p})_\Sigma\subseteq\mathcal{C}_\mathsf{c}(\mathsf{p})$ and so again $\widehat{\mathrm{cl}(S^>_\Sigma)}\subseteq\mathrm{cl}_\mathsf{b}(\mathcal{C}_\mathsf{c}(\mathsf{p}))=\mathcal{C}_\mathsf{r}(\mathsf{p})$.  As $Z^>\subseteq\Phi[S]$, \eqref{PhiPreserving} then yields $\widehat{Z^>}\subseteq\mathcal{D}_\mathsf{c}(\mathsf{p})(=\mathcal{D}(\rho)\cap\mathcal{C}_\mathsf{c}(\rho))$ so $\widehat{Z^\gg}\subseteq\widehat{Z^>}\cap\widehat{Z}\subseteq\mathcal{Z}_\mathsf{c}(\mathsf{p})$, by \eqref{CWunits}.  Then \eqref{clZgg} yields
\[\widehat{Z}=\widehat{\mathrm{cl}(Z^\gg)}\subseteq\mathrm{cl}_\mathsf{b}(\widehat{Z^\gg})\subseteq\mathrm{cl}_\mathsf{b}(\mathcal{Z}_\mathsf{c}(\mathsf{p}))=\mathcal{Z}_\mathsf{r}(\mathsf{p}).\]

Moreover, $\widehat{Z}$ separates points of $\mathcal{U}^0$ and vanishes nowhere because, for any distinct $T,U\in\mathcal{U}^0$, we have $y\in T^Z$ and $z\in U^Z$ with $yz=0$ and hence $\widehat{z}(T)=0_T$ and $\widehat{z}(U)=1_U$.  By \autoref{WeylCoricalFell}, $\mathsf{p}$ is corical and hence categorical so $\mathsf{p}|_{\mathbb{C}\mathcal{W}^0}$ is trivial, by \autoref{CategoricalChars}.  This means $\mathcal{Z}_\mathsf{r}(\mathsf{p})$ can be identified with $C_0(\mathcal{U}^0)$ so, by Stone-Weierstrass, $\widehat{Z}$ must actually be the entirety of $\mathcal{Z}_\mathsf{r}(\mathsf{p})$, i.e. $\widehat{Z}=\mathcal{Z}_\mathsf{r}(\mathsf{p})$.

Next we wish to show that $\mathcal{S}_\mathsf{r}(\mathsf{p})=\widehat{\mathrm{cl}(S^\sim_\Sigma)}$.  To see this, take $s\in\mathcal{S}_\mathsf{r}(\mathsf{p})$.  For each $n\in\mathbb{N}$, let $K_n=s^{-1}\{b\in B:\|b\|\geq2^{-n}\}$.  As $K_n$ is a compact slice, it is contained in some open slice $O_n$, by \cite[Proposition 6.3]{BiceStarling2018}.  For every $U\in K_n$, we have $a\in U^>$ with $s(U)=[a,U]=\widehat{a}(U)$.  By replacing $a$ with another element of $a_U$ if necessary, we may assume that $a<b$, for some $b\in U$ with $\mathcal{U}_b\subseteq O_n\cap O_{n+1}$.  As $s$ and $\widehat{a}$ are continuous, and the norm is upper semicontinuous, it follows that $\|s(T)-\widehat{a}(T)\|<2^{-n}$ for all $T$ in some neighbourhood of $U$.  As $K_n$ is compact, we can then obtain $a_n^1,\ldots,a_n^{k_n},b_n^1,\ldots,b_n^{k_n}\in S^>$ and open slices $O_n^1,\ldots,O_n^{k_n}\Subset O_n$ covering $K_n$ such that, for each $j\leq k_n$, $a_n^j<b_n^j$, $\mathcal{U}_{b_n^j}\subseteq O_n\cap O_{n+1}$ and $\|s(T)-\widehat{a_n^j}(T)\|<2^{-n}$, for all $T\in O_n^j$.  Then $\mathcal{U}_{b_n^i}\cup\mathcal{U}_{b_n^j}\subseteq O_n$ is a slice and hence $a_n^i\sim a_n^j$, by \eqref{USlice}, for all $i,j\leq k_n$.  Likewise, $\mathcal{U}_{b_n^i}\cup\mathcal{U}_{b_{n+1}^j}\subseteq O_{n+1}$ is a slice and hence $a_n^i\sim a_{n+1}^j$, by \eqref{USlice}, for all $i\leq k_n$ and $j\leq k_{n+1}$.

For each $n\in\mathbb{N}$, take a continuous partition of unity $f^1,\dots,f^{k_n},f^{k_{n+1}}$ of $[0,1]$-valued functions on $\mathcal{U}(S)$ with supports in $O_n^1,\ldots,O_n^{k_n},\mathcal{U}(S)\setminus K_n$.  As $\widehat{Z}=\mathcal{Z}_\mathsf{r}(\mathsf{p})$, for each $j\leq k_n$, we have $z^j\in Z$ with $\mathrm{supp}(\widehat{z^j})\subseteq\mathsf{r}[O_n^j]$ and $\widehat{z^j}(\mathsf{r}(U))=f^j(U)$, for all $U\in O_n^j$.  Letting $s_n^j=z^ja_n^j$, we see that, for all $U\in\mathcal{U}(S)$.
\[\widehat{s_n^j}(U)=\widehat{z^j}(\mathsf{r}(U))\widehat{a_n^j}(U)=f^j(U)\widehat{a_n^j}(U).\]
Also $s_n^i\sim s_n^j$, for all $i,j\leq k_n$, and $s_n^i\sim s_{n+1}^j$, for all $i\leq k_n$ and $j\leq k_{n+1}$, by \eqref{PhiComp}.  Letting $s_n=\sum_{j=1}^{k_n}s_n^j$, for all $n\in\mathbb{N}$, it follows that
\[\|s-\widehat{s_n}\|_\infty<2^{-n}.\]
As the terms making up $s_n$ and $s_{n+1}$ are all compatible, \autoref{NormAttaining} yields
\[\|s_n-s_{n+1}\|\leq\|\widehat{s_n}-\widehat{s_{n+1}}\|_\infty\leq2^{1-n}.\]
Thus $(s_n)$ is Cauchy with limit $t\in\mathrm{cl}(S^\sim_\Sigma)$ and
\[\|\widehat{t}-\widehat{s}_n\|_\infty\leq\|\widehat{t}-\widehat{s}_n\|_\mathsf{b}\leq\|t-s_n\|\rightarrow0.\]
Thus $\widehat{t}$ is the $\infty$-limit of $\widehat{s_n}$, which is precisely $s$, i.e. $s=\widehat{t}\in\widehat{\mathrm{cl}(S^\sim_\Sigma)}$.  This shows that $\mathcal{S}_\mathsf{r}(\mathsf{p})\subseteq\widehat{\mathrm{cl}(S^\sim_\Sigma)}$ and hence $\mathcal{S}_\mathsf{r}(\mathsf{p})=\widehat{\mathrm{cl}(S^\sim_\Sigma)}$ (see above for the reverse inclusion).

It then follows from \eqref{CcSc} that
\[\mathcal{C}_\mathsf{c}(\mathsf{p})\subseteq\mathcal{S}_\mathsf{c}(\mathsf{p})_\Sigma\subseteq\widehat{\mathrm{cl}(S^\sim_\Sigma)}_\Sigma\subseteq\widehat{\mathrm{cl}(S^>_\Sigma)}.\]
But we already know that $\widehat{\mathrm{cl}(S^>_\Sigma)}\subseteq\mathcal{C}_\mathsf{r}(\mathsf{p})$ and that $a\mapsto\widehat{a}$ is a C*-algebra homomorphism on $\mathrm{cl}(S^>_\Sigma)$ so $\widehat{\mathrm{cl}(S^>_\Sigma)}$ must be the entire reduced C*-algebra $\mathcal{C}_\mathsf{r}(\mathsf{p})$.
\end{proof}

As usual, we call $\Phi$ \emph{faithful} on $B\subseteq A$ if $\Phi(b^*b)=0$ implies $b=0$.

\begin{prp}\label{FaithfulWeyl}
The Weyl representation is a C*-algebra isomorphism on $\Phi[C^*(S^>)]$.  It is a C*-algebra isomorphism on $C^*(S^>)$ precisely when $\Phi$ is faithful on $C^*(S^>)$.
\end{prp}

\begin{proof}
By \autoref{NormAttaining}, the Weyl representation is isometric and hence an isomorphism on $Z^>$ and hence on its closure $\Phi[C^*(S^>)]$.  Thus, for any $a\in C^*(S^>)$,
\[\Phi(a^*a)=0\qquad\Leftrightarrow\qquad(\widehat{a}^*\widehat{a})_{\mathcal{U}^0}=\widehat{\Phi(a^*a)}=0\qquad\Leftrightarrow\qquad\widehat{a}=0.\]
Thus $\Phi$ is faithful on $C^*(S^>)$ precisely when the kernel of the Weyl representation on $C^*(S^>)$ is $\{0\}$, which means it is a C*-algebra isomorphism on $C^*(S^>)$.
\end{proof}

\section{Categorical Duality}\label{CategoricalDuality}

Our results so far indicate that we have a duality between certain Fell bundles and structured C*-algebras.  Here we work out the final details so that we can formalise this in category theoretic terms.

First let us show that the Weyl bundle construction of any structured C*-algebra $\langle\rho\rangle_\mathsf{r}$ coming from a corical Fell bundle $\rho$ recovers the original bundle $\rho$.

Recall the \'etale groupoid isomorphism $\iota_\rho:\Gamma\rightarrow\mathcal{U}(S)$ from \eqref{gamma->Sgamma}, which then defines the pullback bundle $\rho_{\iota_\rho}:\iota_\rho^\rho\mathcal{W}(A)\twoheadrightarrow\Gamma$.

\begin{prp}
If $\rho:B\twoheadrightarrow\Gamma$ is a corical Fell bundle and $\langle A\rangle=\langle\rho\rangle_\mathsf{r}$ then we have a bundle isomorphism $\iota^\rho:\iota_\rho^\rho\mathcal{W}(A)\rightarrow B$ from $\rho_{\iota_\rho}$ to $\rho$ given by
\[\iota^\rho(\gamma,[a,S_\gamma])=a(\gamma).\]
\end{prp}

\begin{proof}
Recall from \eqref{NormSgamma} that $\|a(\gamma)-b(\gamma)\|=\|a-b\|_{S_\gamma}$.  In particular, $a(\gamma)=b(\gamma)$ iff $[a,S_\gamma]=[b,S_\gamma]$, which shows that $\iota^\rho$ is both well-defined and injective.  It is also surjective, thanks to \autoref{CtsSections}.  To see that $\iota^\rho$ preserves the product, take any $[a,S_\alpha],[b,S_\beta]\in\mathcal{W}(A)$.  By replacing $a$ and $b$ other elements of $a_{S_\alpha}$ and $b_{S_\beta}$ if necessary, we may assume that $a\in S_\alpha^>$ and $b\in S_\beta^>$.  In particular, $\mathrm{supp}(a)$ and $\mathrm{supp}(b)$ are contained in slices which themselves contain $\alpha$ and $\beta$ respectively so
\[\iota^\rho(\alpha\beta,[ab,S_{\alpha\beta}])=ab(\alpha\beta)=a(\alpha)b(\beta)=\iota^\rho(\alpha,[a,S_\alpha])\iota^\rho(\beta,[b,S_\beta]).\]
In a similar manner we see that all the other operations are also preserved.

Next note that, for any $a\in A$, the pullback section $\iota_\rho^\rho(\widehat{a}):\Gamma\rightarrow\mathcal{W}(A)$ given by
\[\iota_\rho^\rho(\widehat{a})(\gamma)=(\gamma,\widehat{a}(\iota_\rho(\gamma)))=(\gamma,[a,S_\gamma])\]
 is continuous, as both $\iota_\rho:\Gamma\rightarrow\mathcal{U}(S)$ and $\widehat{a}:\mathcal{U}(S)\rightarrow\mathcal{W}(A)$ are continuous.  The values of these pullback sections certainly cover all of $\iota_\rho^\rho\mathcal{W}(A)$.  Moreover, if we compose $\iota_\rho^\rho(\widehat{a})$ with $\iota^\rho$ then we see that
\begin{equation}\label{iotarhorho}
\iota^\rho(\iota_\rho^\rho(\widehat{a})(\gamma))=\iota^\rho(\gamma,[a,S_\gamma])=a(\gamma),
\end{equation}
i.e. we get back the original section $a\in A$, which is certainly continuous.  Thus $\iota^\rho$ must be continuous, by \cite[Ch II Proposition 13.16]{DoranFell1988}.  As $\iota^\rho$ also preserves the norm, its inverse is also continuous, by \cite[Ch II Proposition 13.17]{DoranFell1988}.
\end{proof}

Now we can define $\iota:\mathbf{CFell}^0\rightarrow\mathbf{CFell}^\times$ by
\[\iota(\rho,\mathsf{p}_B,\mathrm{id}_\Gamma,\rho)=(\rho,\iota^\rho,\iota_\rho,\mathsf{p}_{\langle\rho\rangle_\mathsf{r}}).\]

\begin{prp}
$\iota$ is a natural isomorphism from $\mathsf{Sp}\mathsf{Ab}|_{\mathbf{CFell}}$ to $\mathrm{id}_{\mathbf{CFell}}$.
\end{prp}

\begin{proof}
It only remains to prove naturality.  Accordingly, take corical Fell bundles $\rho:B\twoheadrightarrow\Gamma$ and $\rho':B'\twoheadrightarrow\Gamma'$ and let $\langle\rho\rangle_\mathsf{r}=\langle A\rangle=(A,S,Z,\Phi)$ and $\langle\rho'\rangle_\mathsf{r}=\langle A'\rangle=(A',S',Z',\Phi')$.  We must show that, for all $(\rho',\beta,\phi,\rho)\in\mathbf{CFell}$,
\[(\rho',\beta,\phi,\rho)(\rho,\iota^{\rho},\iota_{\rho},\mathsf{p}_{\langle\rho\rangle_\mathsf{r}})=(\rho',\iota^{\rho'},\iota_{\rho'},\mathsf{p}_{\langle\rho'\rangle_\mathsf{r}})\mathsf{Sp}\mathsf{Ab}(\rho',\beta,\phi,\rho).\]
Note $(\rho',\beta,\phi,\rho)(\rho,\iota^{\rho},\iota_{\rho},\mathsf{p}_{\langle\rho\rangle_\mathsf{r}})=(\rho',\beta\bullet\iota^{\rho},\iota_{\rho}\circ\phi,\mathsf{p}_{\langle\rho\rangle_\mathsf{r}})$ while
\begin{align*}
(\rho',\iota^{\rho'},\iota_{\rho'},\mathsf{p}_{\langle\rho'\rangle_\mathsf{r}})\mathsf{Sp}\mathsf{Ab}(\rho',\beta,\phi,\rho)&=(\rho',\iota^{\rho'},\iota_{\rho'},\mathsf{p}_{\langle\rho'\rangle_\mathsf{r}})\mathsf{Sp}(\langle\rho'\rangle_\mathsf{r},\beta_{\rho_\phi}^{\rho'}\!\circ\phi^\rho_\mathsf{r},\langle\rho\rangle_\mathsf{r})\\
&=(\rho',\iota^{\rho'},\iota_{\rho'},\mathsf{p}_{\langle\rho'\rangle_\mathsf{r}})(\mathsf{p}_{\langle\rho'\rangle_\mathsf{r}},\overline{\beta_{\rho_\phi}^{\rho'}\!\circ\phi^\rho_\mathsf{r}},\underline{\beta_{\rho_\phi}^{\rho'}\!\circ\phi^\rho_\mathsf{r}},\mathsf{p}_{\langle\rho\rangle_\mathsf{r}})\\
&=(\rho',\iota^{\rho'}\bullet\overline{\beta_{\rho_\phi}^{\rho'}\!\circ\phi^\rho_\mathsf{r}},\underline{\beta_{\rho_\phi}^{\rho'}\!\circ\phi^\rho_\mathsf{r}}\circ\iota_{\rho'},\mathsf{p}_{\langle\rho\rangle_\mathsf{r}}).
\end{align*}
So we need to show that $\beta\bullet\iota^{\rho}=\iota^{\rho'}\bullet\overline{\beta_{\rho_\phi}^{\rho'}\!\circ\phi^\rho_\mathsf{r}}$ and $\iota_{\rho}\circ\phi=\underline{\beta_{\rho_\phi}^{\rho'}\!\circ\phi^\rho_\mathsf{r}}\circ\iota_{\rho'}$.

For the latter equation note that, for all $\gamma\in\Gamma'$,
\begin{align*}
\underline{\beta_{\rho_\phi}^{\rho'}\!\circ\phi^\rho_\mathsf{r}}\circ\iota_{\rho'}(\gamma)&=\underline{\beta_{\rho_\phi}^{\rho'}\!\circ\phi^\rho_\mathsf{r}}(S'_\gamma)\\
&=(\beta_{\rho_\phi}^{\rho'}\!\circ\phi^\rho_\mathsf{r})^{-1}[S'_\gamma]^<\\
&=\{s\in S:\beta_{\rho_\phi}^{\rho'}\!\circ\phi^\rho_\mathsf{r}(s)\in S'_\gamma\}^<\\
&=\{s\in S:\beta_{\rho_\phi}^{\rho'}\!\circ\phi^\rho_\mathsf{r}(s)(\gamma)\in B'^\times\}^<\\
&=\{s\in S:\beta(\gamma,s(\phi(\gamma)))\in B'^\times\}^<\\
&\subseteq\{s\in S:s(\phi(\gamma))\in B^\times\}^<\\
&=S_{\phi(\gamma)}.
\end{align*}
As inclusion on ultrafilters implies equality, it follows that
\[\underline{\beta_{\rho_\phi}^{\rho'}\!\circ\phi^\rho_\mathsf{r}}\circ\iota_{\rho'}(\gamma)=S_{\phi(\gamma)}=\iota_\rho(\phi(\gamma))=\iota_\rho\circ\phi(\gamma)\]
and hence $\underline{\beta_{\rho_\phi}^{\rho'}\!\circ\phi^\rho_\mathsf{r}}\circ\iota_{\rho'}=\iota_\rho\circ\phi$, as required.

For the former equation note that, for all $\gamma\in\Gamma'$ and $a\in A$,
\[\beta\bullet\iota^{\rho}(\gamma,[a,S_{\phi(\gamma)}])=\beta(\gamma,\iota^\rho(\phi(\gamma),[a,S_{\phi(\gamma)}]))=\beta(\gamma,a(\phi(\gamma))),\]
On the other hand,
\begin{align*}
\iota^{\rho'}\bullet\overline{\beta_{\rho_\phi}^{\rho'}\!\circ\phi^\rho_\mathsf{r}}(\gamma,[a,S_{\phi(\gamma)}])&=\iota^{\rho'}(\gamma,\overline{\beta_{\rho_\phi}^{\rho'}\!\circ\phi^\rho_\mathsf{r}}(S'_\gamma,[a,S_{\phi(\gamma)}]))\\
&=\iota^{\rho'}(\gamma,[\beta_{\rho_\phi}^{\rho'}\!\circ\phi^\rho_\mathsf{r}(a),S'_\gamma])\\
&=\beta_{\rho_\phi}^{\rho'}\!\circ\phi^\rho_\mathsf{r}(a)(\gamma)\\
&=\beta(\gamma,a(\phi(\gamma))).
\end{align*}
This shows that $\beta\bullet\iota^{\rho}=\iota^{\rho'}\bullet\overline{\beta_{\rho_\phi}^{\rho'}\!\circ\phi^\rho_\mathsf{r}}$ as well, as required.
\end{proof}

\begin{dfn}
We call well-structured $(A,S,Z,\Phi)$ \emph{sum-structured} if
\[\tag{Sum-Structured}S\subseteq\mathrm{cl}(S^\sim_\Sigma)\qquad\text{and}\qquad A\subseteq\mathrm{cl}(S^>_\Sigma).\]
\end{dfn}

Note that $S\subseteq\mathrm{cl}(S^\sim_\Sigma)$ and $A\subseteq\mathrm{cl}(S^>_\Sigma)$ imply that, in fact,
\[A=\mathrm{cl}(S^>_\Sigma)=C^*(S^>)=C^*(S).\]

\begin{prp}
If $\langle A\rangle=(A,S,Z,\Phi)$ is sum-structured, the Weyl representation $\omega_{\langle A\rangle}:A\rightarrow\mathcal{C}_\mathsf{r}(\mathsf{p}_{\langle A\rangle})$ is a structure-preserving morphism from $\langle A\rangle$ to $\langle\mathsf{p}_{\langle A\rangle}\rangle_\mathsf{r}$.
\end{prp}

\begin{proof}
If $\langle A\rangle=(A,S,Z,\Phi)$ is sum-structured then \autoref{WeylRep} tells us that $\omega_{\langle A\rangle}$ is a C*-algebra homomorphism from $A=\mathrm{cl}(S^>_\Sigma)$ onto $\mathcal{C}_\mathsf{r}(\mathsf{p}_{\langle A\rangle})$ such that
\[\omega_{\langle A\rangle}[S]=\widehat{S}\subseteq\widehat{\mathrm{cl}(S^\sim_\Sigma)}=\mathcal{S}_\mathsf{r}(\mathsf{p}_{\langle A\rangle})\qquad\text{and}\qquad\omega_{\langle A\rangle}[Z]=\widehat{Z}=\mathcal{Z}_\mathsf{r}(\mathsf{p}_{\langle A\rangle}).\]
As noted in \eqref{PhiPreserving}, we also have $\widehat{\Phi(a)}=\widehat{a}_{\mathcal{U}^0}=(\mathsf{p}_{\langle A\rangle})_\mathsf{r}^0(\widehat{a})$, for all $a\in A$.  This shows that $\omega_{\langle A\rangle}$ is structure-preserving.
\end{proof}

Let $\mathbf{SSC^*}$ denote the full subcategory of $\mathbf{WSC^*}$ consisting of $(\langle A'\rangle,\pi,\langle A\rangle)$ such that $\langle A\rangle$ and $\langle A'\rangle$ are sum-structured C*-algebras.  Define $\omega:\mathbf{SSC^*}^0\rightarrow\mathbf{SSC^*}$ by
\[\omega(\langle A\rangle,\mathrm{id}_{\langle A\rangle},\langle A\rangle)=(\langle\mathsf{p}_{\langle A\rangle}\rangle_\mathsf{r},\omega_{\langle A\rangle},\langle A\rangle).\]

\begin{prp}
$\omega$ is a natural transformation from $\mathrm{id}_{\mathbf{SSC^*}}$ to $\mathsf{Ab}\mathsf{Sp}|_{\mathbf{SSC^*}}$.
\end{prp}

\begin{proof}
We must show that, for any structure-preserving morphism $\pi:A\rightarrow A'$ between sum-structured $\langle A\rangle=(A,S,Z,\Phi)$ and $\langle A'\rangle=(A',S',Z',\Phi')$,
\begin{align*}
(\langle\mathsf{p}_{\langle A'\rangle}\rangle_\mathsf{r},\omega_{\langle A'\rangle},\langle A'\rangle)(\langle A'\rangle,\pi,\langle A\rangle)&=\mathsf{Ab}\mathsf{Sp}(\langle A'\rangle,\pi,\langle A\rangle)(\langle\mathsf{p}_{\langle A\rangle}\rangle_\mathsf{r},\omega_{\langle A\rangle},\langle A\rangle)
\end{align*}
This amounts to showing that
\begin{equation}\label{MorphEq}
\omega_{\langle A'\rangle}\circ\pi=\overline\pi_{\mathsf{p}_{\langle A\rangle\underline\pi}}^{\mathsf{p}_{\langle A'\rangle}}\!\circ\underline\pi^{\mathsf{p}_{\langle A\rangle}}_\mathsf{r}\circ\omega_{\langle A\rangle}.
\end{equation}
To see this note that, for any $a\in A$ and $U\in\mathcal{U}(S')$ with $\pi^{-1}[U]\neq\emptyset$,
\begin{align*}
(\overline\pi_{\mathsf{p}_{\langle A\rangle\underline\pi}}^{\mathsf{p}_{\langle A'\rangle}}\!\circ\underline\pi^{\mathsf{p}_{\langle A\rangle}}_\mathsf{r}\circ\omega_{\langle A\rangle}(a))(U)&=(\overline\pi_{\mathsf{p}_{\langle A\rangle\underline\pi}}^{\mathsf{p}_{\langle A'\rangle}}\!\circ\underline\pi^{\mathsf{p}_{\langle A\rangle}}_\mathsf{r}(\widehat{a}))(U)\\
&=\overline\pi(U,\widehat{a}(\underline\pi(U)))\\
&=\overline\pi(U,[a,\underline\pi(U)])\\
&=[\pi(a),U]\\
&=\widehat{\pi(a)}(U)\\
&=(\omega_{\langle A'\rangle}\circ\pi(a))(U).
\end{align*}
On the other hand, if $\pi^{-1}[U]=\emptyset$ then we claim that $[\pi(a),U]=0_U$.  Indeed, if we had $[\pi(a),U]\neq0_U$ then, as $A\subseteq S^>_\Sigma$, we would have $a_1,\ldots,a_n\in S^>$ with $\|a-\sum_{k=1}^na_k\|<\|\pi(a)\|_U$ and hence $[\sum_{k=1}^n\pi(a_k),U]\neq0_U$ so $[\pi(a_k),U]\neq0_U$, for some $k\leq n$.  Taking $b\in S$ with $a_k<b$, it follows that $\pi(a_k)<\pi(b)$ and hence $\pi(b)\in U$, by \cite[Proposition 3.14]{Bice2020Rings} (and the comments immediately after), so $b\in\pi^{-1}[U]$, contradicting $\pi^{-1}[U]=\emptyset$.  This proves the claim and so again
\[(\omega_{\langle A'\rangle}\circ\pi(a))(U)=[\pi(a),U]=0_U=(\overline\pi_{\mathsf{p}_{\langle A\rangle\underline\pi}}^{\mathsf{p}_{\langle A'\rangle}}\!\circ\underline\pi^{\mathsf{p}_{\langle A\rangle}}_\mathsf{r}\circ\omega_{\langle A\rangle}(a))(U),\]
by the definition of $\overline\pi_{\mathsf{p}_{\langle A\rangle\underline\pi}}^{\mathsf{p}_{\langle A'\rangle}}$ (see \eqref{BetaDef}), as $\pi^{-1}[U]=\emptyset$ means $U\notin\mathrm{dom}(\underline\pi)$.  Thus
\[\overline\pi_{\mathsf{p}_{\langle A\rangle\underline\pi}}^{\mathsf{p}_{\langle A'\rangle}}\!\circ\underline\pi^{\mathsf{p}_{\langle A\rangle}}_\mathsf{r}\circ\omega_{\langle A\rangle}(a)=\omega_{\langle A'\rangle}\circ\pi(a),\]
for all $a\in A$, proving \eqref{MorphEq}.
\end{proof}

\begin{dfn}
We call sum-structured $(A,S,Z,\Phi)$ \emph{faithfully structured} if $\Phi$ is faithful and $S$ has compatible sums.
\end{dfn}

\begin{prp}
$\langle\rho\rangle_\mathsf{r}$ is faithfully structured, for all corical $\rho:B\twoheadrightarrow\Gamma$.
\end{prp}

\begin{proof}
By \autoref{FellWellStructured} and \eqref{SumsCompactSupports}, $\langle\rho\rangle_\mathsf{r}$ is sum-structured.  By \autoref{SrChar} and \autoref{GeneralCompatibility}, $S=\mathcal{S}_\mathsf{r}(\rho)$ has compatible sums.  The canonical expectation is always faithful on $A=\mathcal{C}_\mathsf{r}(\rho)$ and hence $\langle\rho\rangle_\mathsf{r}$ is faithfully structured.
\end{proof}

Again let $\mathbf{FSC^*}$ denote the full subcategory of $\mathbf{SSC^*}$ consisting of $(\langle A'\rangle,\pi,\langle A\rangle)$ such that $\langle A\rangle$ and $\langle A'\rangle$ are faithfully structured C*-algebras.

\begin{prp}\label{Faithful=>NaturalIsomorphism}
If $\langle A\rangle$ is faithfully structured, $(\langle\mathsf{p}_{\langle A\rangle}\rangle_\mathsf{r},\omega_{\langle A\rangle},\langle A\rangle)\in\mathbf{FSC^*}^\times$.
\end{prp}

\begin{proof}
If $\langle A\rangle=(A,S,Z,\Phi)$ is faithfully structured then $\omega_{\langle A\rangle}$ is a C*-algebra isomorphism, by \autoref{FaithfulWeyl}.  As $S$ has compatible sums, $\omega_{\langle A\rangle}$ also maps $S=\mathrm{cl}(S_\Sigma^\sim)$ onto $\mathcal{S}_\mathsf{r}(\mathsf{p}_{\langle A\rangle})$ and $Z$ onto $\mathcal{Z}_\mathsf{r}(\mathsf{p}_{\langle A\rangle})$, by \autoref{WeylRep}.  Thus $\omega_{\langle A\rangle}$ is a structured C*-algebra isomorphism.
\end{proof}

\begin{thm}\label{TheAdjunction}
$(\mathsf{Sp}|_{\mathbf{SSC^*}},\mathsf{Ab}|_{\mathbf{CFell}},\iota,\omega)$ is an adjunction.  Restricting $\mathsf{Sp}$ further to $\mathbf{FSC^*}$ yields an equivalence with $\mathbf{CFell}$.
\end{thm}

\begin{proof}
It only remains to verify \eqref{Zigzag}.  Accordingly, take any corical Fell bundle $\rho:B\twoheadrightarrow\Gamma$ and let $\langle A\rangle=(A,S,Z,\Phi)=\langle\rho\rangle_\mathsf{r}$.  We first need to verify that
\[\mathsf{Ab}(\rho,\mathsf{p}_B,\mathrm{id}_\Gamma,\rho)=\mathsf{Ab}(\iota(\rho,\mathsf{p}_B,\mathrm{id}_\Gamma,\rho))\omega(\mathsf{Ab}(\rho,\mathsf{p}_B,\mathrm{id}_\Gamma,\rho)),\]
To see this note that
\begin{align*}
\mathsf{Ab}(\iota(\rho,\mathsf{p}_B,\mathrm{id}_\Gamma,\rho))\omega(\mathsf{Ab}(\rho,\mathsf{p}_B,\mathrm{id}_\Gamma,\rho))&=\mathsf{Ab}(\rho,\iota^\rho,\iota_\rho,\mathsf{p}_{\langle\rho\rangle_\mathsf{r}})\omega(\langle A\rangle,\mathrm{id}_A,\langle A\rangle)\\
&=(\langle\rho\rangle_\mathsf{r},(\iota^\rho)^\rho_{\mathsf{p}_{\langle\rho\rangle_\mathsf{r}\iota_\rho}}\!\circ\iota_{\rho\mathsf{r}}^\rho,\langle\mathsf{p}_{\langle\rho\rangle_\mathsf{r}}\rangle_\mathsf{r})(\langle\mathsf{p}_{\langle\rho\rangle_\mathsf{r}}\rangle_\mathsf{r},\omega_{\langle\rho\rangle_\mathsf{r}},\langle\rho\rangle_\mathsf{r})\\
&=(\langle A\rangle,(\iota^\rho)^\rho_{\mathsf{p}_{\langle A\rangle\iota_\rho}}\!\circ\iota_{\rho\mathsf{r}}^\rho,\langle\mathsf{p}_{\langle A\rangle}\rangle_\mathsf{r})(\langle\mathsf{p}_{\langle A\rangle}\rangle_\mathsf{r},\omega_{\langle A\rangle},\langle A\rangle)\\
&=(\langle A\rangle,(\iota^\rho)^\rho_{\mathsf{p}_{\langle A\rangle\iota_\rho}}\!\circ\iota_{\rho\mathsf{r}}^\rho\circ\omega_{\langle A\rangle},\langle A\rangle)\\
&=(\langle A\rangle,\mathrm{id}_{\langle A\rangle},\langle A\rangle)\\
&=\mathsf{Ab}(\rho,\mathsf{p}_B,\mathrm{id}_\Gamma,\rho),
\end{align*}
because $(\iota^\rho)^\rho_{\mathsf{p}_{\langle A\rangle\iota_\rho}}\!\circ\iota_{\rho\mathsf{r}}^\rho\circ\omega_{\langle A\rangle}(a)=(\iota^\rho)^\rho_{\mathsf{p}_{\langle A\rangle\iota_\rho}}\!\circ\iota_{\rho\mathsf{r}}^\rho(\widehat{a})=a$, for all $a\in A$, which in turn follows from the fact that $(\iota^\rho)^\rho_{\mathsf{p}_{\langle A\rangle\iota_\rho}}\!\circ\iota_{\rho\mathsf{r}}^\rho(\widehat{a})(\gamma)=\iota^\rho(\iota_{\rho}^\rho(\widehat{a})(\gamma))=a(\gamma)$, for all $\gamma\in\Gamma$, as already noted above in \eqref{iotarhorho}.

Next we must verify that, for any sum-structured $\langle A\rangle=(A,S,Z,\Phi)$,
\[\mathsf{Sp}(\langle A\rangle,\mathrm{id}_A,\langle A\rangle)=\iota(\mathsf{Sp}(\langle A\rangle,\mathrm{id}_A,\langle A\rangle))\mathsf{Sp}(\omega(\langle A\rangle,\mathrm{id}_A,\langle A\rangle)).\]
To see this note that
\begin{align*}
\iota(\mathsf{Sp}(\langle A\rangle,\mathrm{id}_A,\langle A\rangle))\mathsf{Sp}(\omega(\langle A\rangle,\mathrm{id}_A,\langle A\rangle))&=\iota(\mathsf{p}_{\langle A\rangle},\overline{\mathrm{id}_A},\underline{\mathrm{id}_A},\mathsf{p}_{\langle A\rangle})\mathsf{Sp}(\langle\mathsf{p}_{\langle A\rangle}\rangle_\mathsf{r},\omega_{\langle A\rangle},\langle A\rangle)\\
&=(\mathsf{p}_{\langle A\rangle},\iota^{\mathsf{p}_{\langle A\rangle}},\iota_{\mathsf{p}_{\langle A\rangle}},\mathsf{p}_{\langle\mathsf{p}_{\langle A\rangle}\rangle_\mathsf{r}})(\mathsf{p}_{\langle\mathsf{p}_{\langle A\rangle}\rangle_\mathsf{r}},\overline{\omega_{\langle A\rangle}},\underline{\omega_{\langle A\rangle}},\mathsf{p}_{\langle A\rangle})\\
&=(\mathsf{p}_{\langle A\rangle},\iota^{\mathsf{p}_{\langle A\rangle}}\bullet\overline{\omega_{\langle A\rangle}},\underline{\omega_{\langle A\rangle}}\circ\iota_{\mathsf{p}_{\langle A\rangle}},\mathsf{p}_{\langle A\rangle}).
\end{align*}
Looking at the groupoid functor part we see that, for any $U\in\mathcal{U}(S)$,
\begin{align*}
\underline{\omega_{\langle A\rangle}}\circ\iota_{\mathsf{p}_{\langle A\rangle}}(U)&=\omega_{\langle A\rangle}^{-1}\{a\in\mathcal{S}_\mathsf{r}(\mathsf{p}_{\langle A\rangle}):a(U)\in\mathcal{W}(A)^\times\}^<\\
&=\{a\in S:\widehat{a}(U)\in\mathcal{W}(A)^\times\}^<\\
&=\{a\in S:[a,U]\in\mathcal{W}(A)^\times\}^<\\
&=\{a\in S:a\in U\}^<\\
&=U^<=U,
\end{align*}
thanks to \eqref{Core}, showing that $\underline{\omega_{\langle A\rangle}}\circ\iota_{\mathsf{p}_{\langle A\rangle}}=\mathrm{id}_{\mathcal{U}(S)}$.  On the other hand, looking at the bundle morphism part we see that, for any $(U,[a,U])\in\mathrm{id}_{\mathcal{U}(S)}^{\mathsf{p}_{\langle A\rangle}}\mathcal{W}(A)$,
\begin{align*}
\iota^{\mathsf{p}_{\langle A\rangle}}\bullet\overline{\omega_{\langle A\rangle}}(U,[a,U])&=\iota^{\mathsf{p}_{\langle A\rangle}}(U,\overline{\omega_{\langle A\rangle}}(\iota_{\mathsf{p}_{\langle A\rangle}}(U),[a,U]))\\
&=\iota^{\mathsf{p}_{\langle A\rangle}}(U,[\omega_{\langle A\rangle}(a),\iota_{\mathsf{p}_{\langle A\rangle}}(U)])\\
&=\iota^{\mathsf{p}_{\langle A\rangle}}(U,[\widehat{a},\iota_{\mathsf{p}_{\langle A\rangle}}(U)])\\
&=\widehat{a}(U)=[a,U],
\end{align*}
showing that $\iota^{\mathsf{p}_{\langle A\rangle}}\bullet\overline{\omega_{\langle A\rangle}}=\mathsf{p}_{\mathcal{W}(A)}$.  Thus
\[\iota(\mathsf{Sp}(\langle A\rangle,\mathrm{id}_A,\langle A\rangle))\mathsf{Sp}(\omega(\langle A\rangle,\mathrm{id}_A,\langle A\rangle))=(\mathsf{p}_{\langle A\rangle},\mathsf{p}_{\mathcal{W}(A)},\mathrm{id}_{\mathcal{U}(S)},\mathsf{p}_{\langle A\rangle})=\mathsf{Sp}(\langle A\rangle,\mathrm{id}_A,\langle A\rangle).\]
Thus \eqref{Zigzag} holds, showing that $(\mathsf{Sp}|_{\mathbf{SSC^*}},\mathsf{Ab}|_{\mathbf{CFell}},\iota,\omega)$ is an adjunction.  By \autoref{Faithful=>NaturalIsomorphism}, restricting $\mathsf{Sp}$ to $\mathbf{FSC^*}$ yields an equivalence with $\mathbf{CFell}$.
\end{proof}

\addtocontents{toc}{\protect\addvspace{2.25em plus 1pt}}
\bookmarksetup{startatroot}

\section*{Concluding Remarks}

This completes our exposition of the duality between corical Fell bundles and faithfully structured C*-algebras, simultaneously extending the fundamental representation theorems of Dauns-Hofmann, Kumjian and Renault.  From a theoretical point of view, this already achieves the long-standing goal of providing a non-commutative extension of the classic Gelfand duality (see \cite{Mannucci2020} for a recent discussion of this).  If the Kumjian-Renault theory is anything to go by, it also has the potential to find concrete applications in future work on operator algebras and related fields.  Another possibility for future work is to generalise the duality even further and analyse the relationship between the Fell bundle and its associated structured C*-algebra in more detail.  Indeed, there are good reasons why one might want to extend the duality in various ways.

In contrast to the original Kumjian-Renault theory and many of its more recent extensions, we had no need to restrict the isotropy of our groupoids in any way.  However, the caveat here is that, the more non-trivial isotropy the groupoids have, the more restrictive the morphisms will be.  In the extreme case of morphisms between discrete groups, we see that star-bijectivity just amounts to bijectivity.  On the structured C*-algebra side, this corresponds to the fact our morphisms must preserve expectations.  If we want the duality to encompass more general group(oid) homomorphisms, then we would have to consider C*-algebra homomorphisms that may only preserve the range of the expectations, not their kernels.  Our previous work in \cite{Bice2020Rep} indicates that this could be feasible, as long as one is willing to consider certain relational `Zakrzewski morphisms' between the groupoids.

One might even try to eliminate the expectations altogether, replacing them with diagonal C*-subalgebras satisfying certain properties.  This could allow the duality to be extended to Fell bundles over even non-Hausdorff groupoids.  Again our previous work in \cite{Bice2020Rep}, as well as recent work of Exel and Pitts in \cite{ExelPitts2019}, shows that this is not beyond the realm of possibility, although the technical hurdles to overcome may be significant.

There are also a number of questions one could ask about properties of structured C*-algebras and what they correspond to in the associated Fell bundles.  For example, it is known that a full groupoid C*-algebra is simple if and only if the groupoid is effective and minimal (see \cite{BrownClarkFarthiingSims2014}).  For tight groupoids defined by inverse semigroups, there are elementary algebraic conditions on the inverse semigroup characterising effectiveness and minimality of the groupoid (see \cite{ExelPardo2016}).  One could naturally ask if similar algebraic conditions on the semigroup part of a structured C*-algebra characterise simplicity of the C*-algebraic part as well as effectiveness and minimality of the associated ultrafilter groupoid.

\bibliography{maths}{}
\bibliographystyle{alphaurl}

\end{document}